\documentclass[12pt]{amsart}

\makeatletter
\def\subsection{\@startsection{subsection}{2}%
  \z@{.5\linespacing\@plus.7\linespacing}{.5\linespacing}%
  {\normalfont\bfseries}}
\makeatother
\input{xy}
\xyoption{all}
\usepackage{amsmath}
\usepackage{amssymb}
\usepackage{latexsym}
\usepackage{enumerate}
\usepackage[active]{srcltx}
\usepackage{graphics}
\usepackage{epsfig}
\usepackage{url}

\usepackage[usenames]{color}

\theoremstyle{plain}% default
\newtheorem{theorem}[subsubsection]{Theorem}
\newtheorem*{theorem*}{Theorem}
\newtheorem{proposition}[subsubsection]{Proposition}
\newtheorem*{proposition*}{Proposition}
\newtheorem{corollary}[subsubsection]{Corollary}
\newtheorem{lemma}[subsubsection]{Lemma}
\theoremstyle{definition}
\newtheorem{definition}[subsubsection]{Definition}

\newcommand{\N}{{\mathbb N}}
\newcommand{\cO}{{\mathcal O}}

\newcommand{\T}{{\mathcal T}}
\newcommand{\Z}{{\mathbb Z}}

\newcommand{\idealgen}[1]{(#1)}
\newcommand{\gen}[1]{\langle#1\rangle}
\newcommand{\norm}[1]{\left| {#1}\right|}

\newcommand{\join}{\vee}
\newcommand{\bigjoin}{\bigvee}
\newcommand{\meet}{\wedge}

\newcommand{\qc}{\operatorname{qc}}
\newcommand{\Hom}{\operatorname{Hom}}
\newcommand{\End}{\operatorname{End}}

\newcommand{\supp}{\operatorname{supph}}
\newcommand{\supz}{\operatorname{supp}}

\providecommand{\kat}[1]{\text{\textbf{\textsl{#1}}}}
\newcommand{\Zar}{\kat{Zar}}
\newcommand{\RadId}{\kat{RadId}}
\newcommand{\fgRadId}{\kat{RadfgId}}
\newcommand{\CGLoc}{\kat{CGLoc}}

\newcommand{\RfGLoc}{\kat{RfGLoc}}
\newcommand{\fgRfGLoc}{\kat{fgRfGLoc}}
\newcommand{\Thick}{\kat{Thick}}

\newcommand{\op}{^{\text{{\rm{op}}}}}
\newcommand{\zeros}[1]{V(#1)}
\newcommand{\setofobjects}{S}

\newcommand{\Id}{\operatorname{Id}}
\newcommand{\id}{\operatorname{Id}}

\newcommand{\downarrowright}[1]{\downarrow
\rlap{\raise0.1cm\hbox{$\scriptstyle{#1}$}}}

\newcommand{\downarrowleft}[1]{\rlap{\kern-0.2cm
\raise0.1cm\hbox{$\scriptstyle{#1}$}}\downarrow}

\newcommand{\uparrowright}[1]{\uparrow
\rlap{\lower0.1cm\hbox{$\scriptstyle{#1}$}}}

\newcommand{\uparrowleft}[1]{\rlap{\kern-0.2cm
\lower0.1cm\hbox{$\scriptstyle{#1}$}}\uparrow}

\newcommand{\longmapsfrom}{\mathrel{\reflectbox{\ensuremath{\longmapsto}}}}

\def\umono{\ar@{_{(}->}[u]}
\def\uumono{\ar@{_{(}->}[uu]}

\def\lmono{\ar@{_{(}->}[l]}
\def\llmono{\ar@{_{(}->}[ll]}

\newcommand{\Loc}{\operatorname{Loc}}
\newcommand{\Coloc}{\operatorname{Coloc}}
\newcommand{\Spec}{\operatorname{Spec}}
\newcommand{\hocolim}{\operatornamewithlimits{hocolim}}

\begin{document}
\title[Hochster duality in derived categories]{Hochster duality in derived categories
and point-free  reconstruction of schemes}
\author{Joachim Kock}
\address{Universitat Aut\`onoma de Barcelona \\ Departament de Matem\`atiques\\
E-08193 Bellaterra, Spain}
\email{kock@mat.uab.es}

\author{Wolfgang Pitsch}
\address{Universitat Aut\`onoma de Barcelona \\ Departament de Matem\`atiques\\
E-08193 Bellaterra, Spain}
\email{pitsch@mat.uab.es}
\thanks{Both authors were supported by  FEDER/MEC  grant
MTM2010-20692  and  SGR grant SGR119-2009.}


%%%%%%%%%%%%%%%%%%%%%%%%%
%        ABSTRACT        %
%%%%%%%%%%%%%%%%%%%%%%%%%

\begin{abstract}
  For a commutative ring $R$, we exploit localization techniques and point-free
  topology to give an explicit realization of both the Zariski frame of $R$ (the
  frame of radical ideals in $R$) and its Hochster dual frame, as lattices in 
  the poset of localizing subcategories of the unbounded derived category
  $D(R)$.  This yields new conceptual proofs of the classical theorems of
  Hopkins-Neeman and Thomason.  Next we revisit and simplify Balmer's theory of
  spectra and supports for tensor triangulated categories from the viewpoint
  of frames and Hochster duality.    Finally we exploit our results to show how
  a coherent scheme $(X,\cO_X)$  can be 
  reconstructed from the tensor triangulated structure of its  derived category of
  perfect complexes.
\end{abstract}

\maketitle

\tableofcontents

%%%%%%%%%%%%%%%%%%%%%%%%%%%%%%%%%%%%%%%%%%%%%
%%%%%%%%%%%%%%%%%%%%%%%%%%%%%%%%%%%%%%%%%%%%%
\section*{Introduction}
%%%%%%%%%%%%%%%%%%%%%%%%%%%%%%%%%%%%%%%%%%%%%
%%%%%%%%%%%%%%%%%%%%%%%%%%%%%%%%%%%%%%%%%%%%%

One of the remarkable achievements of stable homotopy theory is the
classification of thick subcategories in the finite stable homotopy category by
Devinatz-Hopkins-Smith~\cite{MR960945}.  This result migrated to commutative
algebra in the work of Hopkins~\cite{MR932260} and Neeman~\cite{MR1174255}, was generalized to
the category of perfect complexes over a coherent (i.e.~quasi-compact
quasi-separated) scheme by Thomason~\cite{Thomason:1997}, and found a version in
modular representation theory in the work of
Benson-Carlson-Rickard~\cite{Benson-Carlson-Rickard}.  A theorem of a similar
flavor is the classification of radical thick tensor ideals in a tensor
triangulated category by Balmer~\cite{Balmer:2005}.  In each case, the thick
subcategories (or radical thick tensor ideals) are classified in terms of unions
of closed subsets with quasi-compact complement in a coherent scheme $X$.

What apparently was not noticed is that these classifying subsets are precisely
the open sets in the Hochster dual topology of the Zariski topology on
$X$, and that Hochster duality, originally a rather puzzling result of
Hochster~\cite{MR0251026}, has a very simple description in the setting of
point-free topology~\cite{Johnstone:Stone-spaces}, i.e.~working with frames of open sets instead of with
points.  That Hochster duality is involved in these classification results was
first noticed by Buan, Krause and Solberg~\cite{Buan-Krause-Solberg}, and
independently in \cite{kock:specletter} where the frame viewpoint was perhaps
first exploited; see also the recent~\cite{Iyengar-Krause:1105.1799}.

For $R$ a commutative ring, we denote by $D^\omega(R)$ its derived category of perfect
complexes.  The Zariski frame of $R$ is the frame
of radical ideals in $R$, or equivalently, the frame
of Zariski open sets in $\Spec R$.  The affine case of Thomason's
theorem can be phrased as follows:
\begin{theorem*} 
  The thick subcategories of $D^\omega(R)$ form a coherent 
  frame which is Hochster dual to the Zariski frame of $R$.
\end{theorem*}
This formulation is the starting point for our investigations: since
it states a clean conceptual relationship between two algebraic structures,
with no mention of point sets, there should be a conceptual and 
point-free explanation of it.
We achieve such an explanation as a byproduct of a more general
analysis of frames and lattices of thick subcategories, localizing
subcategories, and tensor ideals in derived categories of a
commutative ring, and more generally of a coherent scheme.

We work in the unbounded derived 
category $D(R)$.  The thick subcategories of $D^\omega(R)$ are in 
one-to-one correspondence with the {\em compactly generated} 
localizing subcategories of $D(R)$.  
% Our first main result is this:
A key ingredient in our proof of the
above theorem is the following (Proposition~\ref{thm comp-celleq-cyc}):
%%%%%%%%%%%%%%%%%%%%%%%%%%%%%%%%%%%%%%%
\begin{proposition*}
%%%%%%%%%%%%%%%%%%%%%%%%%%%%%%%%%%%%%%%
Every localizing subcategory of $D(R)$ generated by a \emph{finite} set of compact
  objects is of the form   $\Loc(R/I)$ for $I$  a finitely generated ideal of $R$,
  and depends only on its radical $\sqrt{I}$.
\end{proposition*}
%%%%%%%%%%%%%%%%%%%%%%%%%%%%%%%%%%%%%%%
Note that $R/I$ is typically {\em not} a compact object,
but it generates the same localizing subcategory as the Koszul complex of $I$, which {\em is} compact.
The radicals of finitely generated ideals  form a distributive lattice
called the {\em Zariski lattice} and we show: (Proposition~\ref{prop main}):

%%%%%%%%%%%%%%%%%%%%%%%%%%%%%%%%%%%%%%%
\begin{proposition*}
%%%%%%%%%%%%%%%%%%%%%%%%%%%%%%%%%%%%%%%
    These localizing subcategories form a distributive lattice 
    isomorphic to the opposite of the Zariski lattice.  The 
    correspondence is given by
\begin{eqnarray*}
  \Loc(R/I) & \leftrightarrow & \sqrt{I} .
\end{eqnarray*}
\end{proposition*}
%%%%%%%%%%%%%%%%%%%%%%%%%%%%%%%%%%%%%%%
This result contains the essence of Thomason's affine result quoted 
above.  More precisely, Thomason's result follows  by {\em 
coherence}, namely the fact that the frame of compactly generated 
localizing subcategories is determined by its finite part --- this is
the lattice counterpart of compact generation.

Having described  the dual of the Zariski frame explicitly inside 
$D(R)$, we proceed to show that also the Zariski frame itself can be realized 
inside $D(R)$, see Theorem~\ref{thm loccatisZar}:
%%%%%%%%%%%%%%%%%%%%%%%%%%%%%%%%%%%%%%%
\begin{theorem*}
%%%%%%%%%%%%%%%%%%%%%%%%%%%%%%%%%%%%%%%
    The localizing subcategories of $D(R)$ generated by modules of 
    the form $R_f$ form a coherent frame isomorphic to the Zariski 
    frame of $R$.
\end{theorem*}
Again by coherence, the essence of this correspondence
is in the finite part, where it is given by the surprisingly simple 
correspondence
$$
\Loc( R_{f_1},\ldots,R_{f_n}) \leftrightarrow 
\sqrt{(f_1,\ldots,f_n)} .
$$

The results explained so far make up Section~\ref{sec LocsubcatDR},
which finishes with a short explanation of the standard procedure of
extracting points; this is convenient for comparison with results
of Neeman and Thomason.

Before coming to the general case of a coherent scheme in
Section~\ref{sec RecSpecSch}, we need some
abstract theory of radical thick tensor ideals in a tensor
triangulated category.  We revisit Balmer's theory of spectra and
supports, and provide a substantial simplification of this theory,
using a point-free approach.  Large parts of Balmer's paper
\cite{Balmer:2005} are subsumed in the following single theorem 
(\ref{thm coherence}):
%%%%%%%%%%%%%%%%
\begin{theorem*}
%%%%%%%%%%%%%%%%%%%%%%%%%%%%%%%%%%%%%%%
In a tensor triangulated category $\T$,  the
radical thick tensor ideals  form a coherent frame,  provided there is 
only a set of them.
\end{theorem*}
%%%%%%%%%%%%%%%%%%%%%%%%%%%%%%%%%%%%%%%
This coherent frame we call the {\em Zariski frame} of $\T$ 
and denote it $\Zar(\T)$, as it is
constructed from the ring-like object $(\T,\otimes,\mathbf{1})$ in the same way
as classically a commutative ring $R$ gives rise to the frame of radical ideals in $R$.
The Zariski frame of $\T=D^\omega(R)$ is naturally identified with the frame
of compactly general localizing subcategories of $D(R)$ featured in our first 
main theorem.

Furthermore, just as in the case of rings as observed by
Joyal~\cite{Joyal:Cahiers1975} in the early 1970s, the Zariski frame
enjoys a universal property,
see Theorem~\ref{thm ZarisInitialTria}:

%%%%%%%%%%%%%%%%%%%%%%%%%%%%%%%%%%%%%%%
\begin{theorem*}
%%%%%%%%%%%%%%%%%%%%%%%%%%%%%%%%%%%%%%%
The support 
  \begin{eqnarray*}
    \T & \longrightarrow & \Zar(\T)  \\
    a & \longmapsto & \sqrt{a}
  \end{eqnarray*}
  is initial among supports.
\end{theorem*}
%%%%%%%%%%%%%%%%%%%%%%%%%%%%%%%%%%%%%%%

With these general results in hand, we can finally assemble our 
precise affine results to establish the following version of Thomason's 
theorem~\cite[Theorem 3.15]{Thomason:1997}:
%%%%%%%%%%%%%%%%%%%%%%%%%%%%%%%%%%%%%%%
\begin{theorem*}
%%%%%%%%%%%%%%%%%%%%%%%%%%%%%%%%%%%%%%%
  Let $X$ be a coherent scheme.  Then the Zariski frame of
  $D^\omega_{\qc}(X)$ is Hochster dual to the Zariski frame of $X$.
\end{theorem*}
%%%%%%%%%%%%%%%%%%%%%%%%%%%%%%%%%%%%%%%
Once again the point-free methods give a more elementary and conceptual
proof, avoiding for example technical tools such as Absolute Noetherian
Approximation.  

Finally, our explicit description of the Zariski frame inside $D(R)$ 
readily endows it with a sheaf of rings, 
encompassing the local structure necessary to reconstruct also the
structure sheaf of $X$:
%%%%%%%%%%%%%%%%%%%%%%%%%%%%%%%%%%%%%%%
\begin{theorem*}
%%%%%%%%%%%%%%%%%%%%%%%%%%%%%%%%%%%%%%%
%     A coherent scheme $(X,\cO_X)$ can be reconstructed from its
%     derived category $D^\omega(X)$ of perfect complexes.
    A coherent scheme $(X,\cO_X)$ can be reconstructed from the
    tensor triangulated category $D_{\qc}^\omega(X)$ of perfect complexes.
\end{theorem*}
%%%%%%%%%%%%%%%%%%%%%%%%%%%%%%%%%%%%%%%
\noindent
Balmer~\cite{MR1938458} had previously obtained such a
  reconstruction theorem in the special case where $X$ is
  topologically noetherian. The general case was obtained by Buan-Krause-Solberg in~\cite[Theorem 9.5]{Buan-Krause-Solberg}. However  both these results rely on point   sets and invoke Thomason's classification theorem, whereas  our
  reconstruction is more direct in the sense that it refers
  directly to the frame of open sets, and exploits the reduction
  to affine schemes to exhibit the sheaf locally as simply
  $\Loc(R_f) \mapsto R_f$. 

\bigskip

Indeed, the main characteristic of our work is that the approach
is essentially point-free.
Point-free methods have a
distinctive constructive flavor contrasted with point-set topology: in general,
points (e.g.~maximal ideals in rings) are only available through choice 
principles such as Zorn's lemma.
% 
% \tiny
% 
% From a purist viewpoint, introducing points is a ``bad 
% habit'' comparable to computing with rational numbers in terms of
% decimals: once they are there, they are difficult to get rid of, but 
% avoiding them in the first place leads to a more elegant treatment.
% 
% \normalsize
% 
In the present paper we assume the axiom of
choice, but exploit the fact that point-free
arguments tend to be simpler and more direct.

Another novelty is that we work with
unbounded chain complexes, which is the key to understanding Hochster duality:
while for compactly generated localizing subcategories this is mostly 
for convenience, its Hochster dual frame, consisting of localizing 
subcategories generated by localizations of the ring, is something we 
think cannot be realized inside $D^\omega(R)$.  Having this realization in $D(R)$
is of course the key point in getting at the structure sheaf in so 
elegant a way.

Finally it is noteworthy that while the classical proofs of 
the classification theorems relied on the
Tensor Nilpotence Theorem, see  for instance Rouquier~\cite{MR2681712} for a discussion of these, we follow Balmer in instead {\em deducing} the
Tensor Nilpotence Theorem, and tie it to the fact that the Zariski frame (like any coherent frame) gives rise to a Kolmogorov topological space.

To finish this introduction we comment on the role of Hochster duality
in these developments, leading to some questions about duality that are poorly 
understood.  Starting with a commutative ring $R$, on one hand we 
can study its ``internal'' structure,  performing the construction of the
Zariski frame on $(R, \cdot, 1)$, which gives the spectrum of $R$ with the 
Zariski topology.  On the other hand we can study its ``external'' structure
by performing the Zariski frame construction now of the ring-like 
object $(D^\omega(R), \otimes, \mathbf{1})$.  The Zariski frame of this object 
is
the Hochster dual of $\Spec(R)$!  
This duality puts in correspondence, from the short exact sequence
$$
0 \to I \to R \to R/I \to 0
$$
(for $I$ a finitely generated ideal) the Zariski open set $D(I)$ with the 
Hochster open set $\Loc(R/I)$.
The Zariski construction is essentially the same in both cases, and enjoy 
similar universal properties.  The duality is therefore somehow encoded
in taking $D^\omega$.  We think this phenomenon deserves further exploration.
When passing to the finer data of structure sheaves, the duality points towards a 
remarkable duality between local rings and domains, which has only been little 
studied;
see Johnstone~\cite[V.4]{Johnstone:Stone-spaces} for a starting point.

%%%%%%%%%%%%%%%%%%%%%%%%%%%%%%%%%%%%%%%%%%%%%
%%%%%%%%%%%%%%%%%%%%%%%%%%%%%%%%%%%%%%%%%%%%%
\section{Preliminaries}\label{sec framework}
%%%%%%%%%%%%%%%%%%%%%%%%%%%%%%%%%%%%%%%%%%%%%
%%%%%%%%%%%%%%%%%%%%%%%%%%%%%%%%%%%%%%%%%%%%%

%%%%%%%%%%%%%%%%%%%%%%%%%%%%%%%%%%%%%%%%%%%%%
\subsection{Localizing subcategories in triangulated categories}
%%%%%%%%%%%%%%%%%%%%%%%%%%%%%%%%%%%%%%%%%%%%%

In this subsection we review some basic properties of localizing categories. 
Expert readers can skip this subsection.  
Further details can be found in
\cite{MR1812507} or in \cite{MR2681709}.
Let $\T$ be a triangulated category, throughout this subsection
assumed to admit arbitrary
sums.  We denote by $\T^\omega$ the
full subcategory of \emph{compact} objects, namely those objects $c \in
\T$ for which the functor $\Hom_\T(c,-): \T \rightarrow
\kat{Ab}$ commutes with
arbitrary sums or indeed all homotopy colimits. 
Our main example will be $D(R)$, the derived category of a commutative ring $R$.

% \tiny

% Recall that in $D(R)$ the compact objects are the
% \emph{perfect} complexes, i.e.~quasi-isomorphic to bounded complexes of finitely generated
% projective modules.
% We denote the full
% subcategory of perfect complexes by $D^\omega(R)$.  It is
% triangulated, stable under \emph{finite} sums and
% retracts.
% Note that $R$ (considered as a complex concentrated in degree $0$) is compact,
% and that it is a generator, in the sense
% that the smallest triangulated category stable under arbitrary sums containing
% $R$ is $D(R)$, and the smallest triangulated category stable under finite sums
% and retracts containing $R$ is $D^\omega(R)$.

% % We now introduce our main object of study: localizing subcategories. 

% \normalsize

%%%%%%%%%%%%%%%%%%%%%%%%%%%%%%%%%%%%%%%%%%%%%
\begin{definition}\label{def locsubcat}
%%%%%%%%%%%%%%%%%%%%%%%%%%%%%%%%%%%%%%%%%%%%%
A \emph{localizing} subcategory is a full
triangulated subcategory $\mathcal{L} \subset \T$  stable under
arbitrary sums.   It is then automatically closed under 
retracts, by the Eilenberg swindle argument.
If $\setofobjects$ is a set of objects in $\T$,
then the smallest triangulated category containing $\setofobjects$ is called the
localizing subcategory generated by $\setofobjects$ and denoted by
$\Loc(\setofobjects)$.
Similarly, using products instead of sums we get the notion of a
\emph{colocalizing} subcategory and a colocalizing subcategory generated by
$\setofobjects$, denoted by $\Coloc(\setofobjects)$.
\end{definition}
%%%%%%%%%%%%%%%%%%%%%%%%%%%%%%%%%%%%%%%%%%%%%

Informally, $\Loc(\setofobjects)$ is the smallest category whose objects can be
built out of the objects in $\setofobjects$ using suspensions, extensions (triangles)
and arbitrary sums.  For instance, in the derived
category of a ring $R$ we have $\Loc(R) = D(R)$.  If two objects $a$ and $b$
generate the same localizing subcategory, $\Loc(a) = \Loc(b)$ then, inspired by
topology and Bousfield localizations (see below), we will say that they are \emph{cellularly
equivalent}.  
For other notions of cellularity in chain complexes, see \cite{MR2970877}.  
A full, triangulated subcategory $\mathcal{S}\subset \T$ is
called \emph{compactly generated} if it is of the form $\mathcal{S}=
\Loc(\setofobjects)$ for some set $\setofobjects$ of compact objects.  In this case
there is a uniform way to describe the objects in $\Loc(\setofobjects)$ which turns
the informal description into the following rigorous statement.

%%%%%%%%%%%%%%%%%%%%%%%%%%%%%%%%%%%%%%%%%%%%%
\begin{proposition}[Rouquier~\cite{MR2434186} Thm.~4.22 and Prop.~3.13]
    \label{prop objincomLoc}\label{cor compincompgen}
%%%%%%%%%%%%%%%%%%%%%%%%%%%%%%%%%%%%%%%%%%%%%
Let $\T$ be a triangulated category, and $\setofobjects$ a set of compact
objects.
Then for any object $a \in \Loc(\setofobjects)$, there
exists a sequence of maps
\[
F_0 \stackrel{f_0}{\longrightarrow} F_1 
\stackrel{f_1}{\longrightarrow} \cdots F_i \stackrel{f_i}{\longrightarrow}
F_{i+1} \cdots
\]
such that $F_0$ and the cone of each $f_i$ is a direct 
% such that $F_0$ and for each $i \geq 0$ the cone of $f_i$ is a  direct 
sum of copies of suspensions of objects in $\setofobjects$, and 
\[
a = \hocolim_{i} F_i.
\]

If moreover $a$ is compact, then
there is $n \in \mathbb{N}$ such that $a$ is a direct summand of
$F_n$, and $F_0$ and the cone of each $f_i$ (for $i<n$) can be taken to be
finite sums of objects in $\setofobjects$.  In
particular there exists a {\em finite} subset $K \subset \setofobjects$ such
that $a \in \Loc(K)$.
\end{proposition}
%%%%%%%%%%%%%%%%%%%%%%%%%%%%%%%%%%%%%%%%%%%%%

Given an object $a$ in a compactly generated localizing subcategory
$\Loc(\setofobjects)$, we will call any of the sequences provided by
Proposition~\ref{prop objincomLoc} a \emph{recipe} for $a$.
% If $a$ is compact,
% then the identity map $a \rightarrow \hocolim_{i} F_i$ has to factor through one
% of the steps, say $F_n$, and with a little more work to control the cones we get
% a \emph{finite} recipe for $a$ (see Rouquier~\cite{MR2434186} Prop.~3.13):
% 
% %%%%%%%%%%%%%%%%%%%%%%%%%%%%%%%%%%%%%%%%%%%%%
% \begin{corollary}\label{cor compincompgen}
% %%%%%%%%%%%%%%%%%%%%%%%%%%%%%%%%%%%%%%%%%%%%%
% Under the same hypothesis as above, if $a$ is moreover compact, then
% there is $n \in \mathbb{N}$ such that $a$ is a direct summand of
% $F_n$, and moreover we may assume that $F_0$ and the cones of the maps
% $f_i$ for $i<n$ are finite sums of elements in $\setofobjects$.  In
% particular there exists a finite subset $K \subset \setofobjects$ such
% that $a \in \Loc(K)$.
% \end{corollary}
% We therefore have:
\begin{corollary}\label{cor LocXTomega}
  For $\setofobjects$ a set of compact objects, we have
\[
 \Loc(\setofobjects)^\omega = \Loc(\setofobjects) \cap \T^\omega,
\]
where the superscript $\omega$ stands for the full subcategory of compact objects.
\end{corollary}

%%%%%%%%%%%%%%%%%%%%%%%%%%%%%%%%%%%%%%%%%%%%%
\subsubsection{Bousfield colocalizations}\label{subsubsec BousfieldLoc}
%%%%%%%%%%%%%%%%%%%%%%%%%%%%%%%%%%%%%%%%%%%%%

Recall that a \emph{Bousfield colocalization functor} in $\T$ is a pair
$(\Gamma,\eta)$, where $\Gamma : \T \rightarrow \T$ is an endofunctor, and $\eta: \Gamma  
\rightarrow \Id$ is a natural transformation such that $\Gamma\eta : \Gamma^2 \rightarrow \Gamma$
is an isomorphism, so $\Gamma$ is idempotent, and $\Gamma \eta = \eta \Gamma$.  A Bousfield colocalization functor on $\T\op$
is called a \emph{Bousfield localization functor}.

Consider also, for any full subcategory $\mathcal{C} \subset \T$ 
closed under suspension, its right and left orthogonal categories:
\begin{enumerate}
 \item $\mathcal{C}^\bot = \{ x \in \T \mid 
 \Hom(c,x)=0 \text{ for all } c \in \mathcal{C}\}$. This is always a colocalizing subcategory and $\mathcal{C}^\bot = \Loc(\mathcal{C})^\bot$. %for the right orthogonal,
 \item ${}^\bot \mathcal{C} = \{ x \in \T \mid 
 \Hom(x,c)=0 \text{ for all } c \in \mathcal{C}\}$.  This is always a localizing subcategory and  ${}^\bot \mathcal{C} ={}^\bot \Coloc(\mathcal{C})$. %for the left orthogonal,
\end{enumerate}

In the compactly generated case, which is enough for our purposes, Brown Representability \cite[Ch.8]{MR1812507}
ensures that both Bousfield localization and colocalization functors exist:

%%%%%%%%%%%%%%%%%%%%%%%%%%%%%%%%%%%%%%%%%%%%
\begin{proposition}\label{prop  LOcCisBous}
%%%%%%%%%%%%%%%%%%%%%%%%%%%%%%%%%%%%%%%%%%%%
Let   $\setofobjects$ be a set of compact
objects in $\T.$  Then there are both a Bousfield localization $L_\setofobjects$ and colocalization $\Gamma_\setofobjects$ functors associated to $\setofobjects$. The essential image of $\Gamma_\setofobjects$ is then $\Loc(\setofobjects)$ and every object $x \in \mathcal{T}$ fits into a localization triangle:
\[
\Gamma_\setofobjects (x) \longrightarrow x \longrightarrow L_\setofobjects (x) \longrightarrow \Sigma \Gamma_\setofobjects(x).
\]
Moreover in this case $\Loc (\setofobjects) = {}^\bot( \Loc (\setofobjects)^\bot)$.
%%%%%%%%%%%%%%%%%%%%%%%%%%%%%%%%%%%%%%%%%%%%
\end{proposition}
Borrowing from the terminology in algebraic topology, we will call  the  Bousfield colocalization functor $\Gamma_\setofobjects$,   a \emph{cellularization} functor. For more properties of Bousfield (co)lo\-ca\-li\-za\-tions in triangulated categories, we refer the interested reader to 
\cite[Chapter 9]{MR1812507} or \cite{MR2681709}.
% \begin{proof}
%  The localizing subcategory $\Loc(\setofobjects)$ is certainly a
%  compactly generated triangulated category.  
%  Given an object $x \in \T$, consider the functor:
% \[
% \begin{array}{rcl}
% H:  \Loc(\setofobjects) & \longrightarrow & \kat{Ab} \\
% y & \longmapsto & \Hom_{\T}(y,x).
% \end{array}
% \]
% This functor sends triangles to long exact sequences and preserves 
% sums, therefore \Joachim{by Brown Representability \cite[Ch.8]{MR1812507},} it is
% represented by an object $Lx$, the functor $x \mapsto Lx$ is then
% a right adjoint to the inclusion functor and $(5)$ is checked. 
% \end{proof}

%%%%%%%%%%%%%%%%%%%%%%%%%%%%%%%%%%%%%%%%%%%%%%%%%%
\subsubsection{Tensor triangulated categories}
%%%%%%%%%%%%%%%%%%%%%%%%%%%%%%%%%%%%%%%%%%%%%%%%%%

Sometimes we will use a \emph{tensor} triangulated category, $(\T, \otimes, \mathbf{1})$, that is a
 symmetric monoidal structure $\otimes$ on $\mathcal{T}$, compatible with the triangulated structure, i.e.~tensoring with an
object is an exact, triangulated endofunctor of
$\T$. The unit of the tensor product will be denoted by $\mathbf{1}$.
For our main example, the derived category $D(R)$, the tensor product
is the {\em derived} tensor product, denoted plainly $\otimes$, and henceforth 
plainly called the tensor product.  Note that
the tensor product of two perfect complexes is again perfect, and that
the unit object $R$ is perfect; hence also $D^\omega(R)$ is a tensor
triangulated category.

%%%%%%%%%%%%%%%%%%%%%%%%%%%%%%%%%%%%%%%%%%%%%
\begin{definition}\label{def tens-ideal}
%%%%%%%%%%%%%%%%%%%%%%%%%%%%%%%%%%%%%%%%%%%%%
A full subcategory $I$ of  $(\T, \otimes, \mathbf{1})$, 
is a \emph{thick tensor ideal}, if it is 
\begin{enumerate}
 \item a full triangulated subcategory,
 \item closed under finite sums,
 \item thick: if $a\oplus b \in  I$ then $a \in I$ 
 and $b \in I$,
 \item absorbing for the tensor product: if $a \in I$ and 
 $b \in \T$ then $a \otimes b \in I$.
\end{enumerate}
\end{definition}
%%%%%%%%%%%%%%%%%%%%%%%%%%%%%%%%%%%%%%%%%%%%%

% In a tensor triangulated category, the notions of localizing subcategory and
% tensor ideal are often closely related:

The following two observations follow from straightforward arguments.
%%%%%%%%%%%%%%%%%%%%%%%%%%%%%%%%%%%%%%%%%%%%%
\begin{lemma}\label{lem locistensideal}\label{lem tensideal}
%%%%%%%%%%%%%%%%%%%%%%%%%%%%%%%%%%%%%%%%%%%%%
Let $(\T, \otimes, \mathbf{1})$ be a tensor triangulated 
category such that $\Loc(\mathbf{1})=\T$.

\begin{enumerate}
  \item  Any localizing subcategory is a tensor ideal. 

  \item Let $x$, $y$ and $z$ be objects in $\T$. 
If $y \in \Loc(x)$, then $y \otimes z \in \Loc(x\otimes z)$.

\end{enumerate}

\end{lemma}
%%%%%%%%%%%%%%%%%%%%%%%%%%%%%%%%%%%%%%%%%%%%%
% \begin{proof}
% Given a localizing subcategory $\mathcal{L}$, let $\mathcal{C}$
% denote the full subcategory whose objects $x$ satisfy
% \[
%  \forall y \in \mathcal{L}, \quad y \otimes x \in \mathcal{L}.
% \]
% One checks directly that $\mathcal{C}$ is triangulated, closed under arbitrary sums and 
% contains $\mathbf{1}$, therefore $\mathcal{C}= \T.$ 
% \end{proof}
% 
% 
% 
% %%%%%%%%%%%%%%%%%%%%%%%%%%%%%%%%%%%%%%%%%%%%%
%  \begin{lemma}\label{lem tensideal}
% %%%%%%%%%%%%%%%%%%%%%%%%%%%%%%%%%%%%%%%%%%%%%
% Let $(\T, \otimes, \mathbf{1})$ be a tensor triangulated 
% category such that $\Loc(\mathbf{1})=\T$. 
% Let $x$, $y$ and $z$ be objects in $\T$. 
% If $y \in \Loc(x)$, then $y \otimes z \in \Loc(x\otimes z)$.
% \end{lemma}
% %%%%%%%%%%%%%%%%%%%%%%%%%%%%%%%%%%%%%%%%%%%%%
% \begin{proof}
% As tensoring with a given object  is a triangulated functor, the full subcategory
% $\mathcal{L}$  of those objects $a$ for which $a \otimes z$ is 
% $x \otimes z$-cellular is a localizing category that trivially  
% contains $x$. Hence $ \Loc(x) \subset \mathcal{L}$.
% \end{proof} 

In presence of a tensor structure a Bousfield (co)localization functor has often a
very peculiar form, for a proof see for instance Balmer and Favi~\cite{MR2806103}:

%%%%%%%%%%%%%%%%%%%%%%%%%%%%%%%%%%%%%%%%%%%%
\begin{theorem}\label{thm tensorloc}
%%%%%%%%%%%%%%%%%%%%%%%%%%%%%%%%%%%%%%%%%%%%
Let $\mathcal{S}$ be a thick tensor ideal in $(\T,\otimes,\mathbf{1})$
 for which there are both a Bousfield localization  $L_\mathcal{S}$ and a Bousfield colocalization $\Gamma_{\mathcal{S}}$ functors, for instance $S$ is generated by compact objects.  Let
\[
\Gamma_\mathcal{S} (\mathbf{1}) \longrightarrow \mathbf{1}
\longrightarrow L_\mathcal{S}(\mathbf{1}) \longrightarrow
\Sigma \Gamma_\mathcal{S}( \mathbf{1})
\]
denote the localization triangle for the tensor unit.
Then the following are equivalent:
\begin{enumerate}
 \item The subcategory $\mathcal{S}^\bot$ is a tensor ideal.
\item There is an isomorphism of functors
$L_{\mathcal{S}} \simeq L_{\mathcal{S}}(\mathbf{1}) \otimes -$.
\item There is an isomorphism of functors 
$\Gamma_{\mathcal{S}} \simeq \Gamma_\mathcal{S}(\mathbf{1}) \otimes-$.
\end{enumerate}

\end{theorem}
\subsection{Frames and Hochster duality}\label{subsec Fr-Hoc-Sup}
%%%%%%%%%%%%%%%%%%%%%%%%%%%%%%%%%%%%%%%%%%%%%

In this subsection we recall some generalities on frames and
Hochster duality.  Our main reference for this material 
is Johnstone~\cite{Johnstone:Stone-spaces}.
 
\begin{definition}
  A {\em frame} is a complete lattice $F$
 in which finite
meets distribute over arbitrary joins: 
\[
\forall a \in F, \ \forall S \subset F, \qquad  a \meet \bigjoin_{s \in S} s = 
\bigjoin_{s \in S} (a \meet s) .
\]
A {\em frame map} is a lattice map required to preserve arbitrary joins.
Let $\kat{Frm}$ denote the category of frames and frame 
maps.
\end{definition}
%%%%%%%%%%%%%%%%%%%%%%%%%%%%%%%%%%%%%%%%%%%%%%%%

The motivating example is the frame of open sets in a topological space.  There
the join operation is given by union of open sets, and
finite meet is given by intersection.  
Sending a topological space to its frame of open sets
constitutes a functor $\kat{Top}\to\kat{Frm}\op$.   This functor has a
right adjoint, the functor of \emph{points}: a {\em point} of a frame $F$ is a
frame map $x:F \to \{0,1\}$, and the set of points form a topological space in
which the open subsets are those of the form $\{x:F\to\{0,1\}\mid x(u) = 1\} $
for some $u\in F$.
The topological spaces occurring in this way are precisely the \emph{sober
spaces} (i.e.~every irreducible closed set has a unique generic point); these
include any Hausdorff space and the underlying topological space of any
scheme. 
Altogether the adjunction between topological spaces and frames restricts
to a contravariant equivalence between sober spaces and {\em spatial}
frames, i.e.~having enough points~\cite[II.1.5]{Johnstone:Stone-spaces}.

% An example of a sober space of particular interest to us is $\Spec_Z R$, the
% spectrum of a ring $R$ with the Zariski topology. 
Topological spaces homeomorphic to the spectrum of a ring were called
\emph{spectral spaces} by Hochster~\cite{MR0251026}, now more commonly called
{\em coherent spaces}~\cite{Johnstone:Stone-spaces}.  Hochster~\cite{MR0251026}
characterized the spectral spaces intrinsically as the sober spaces for which
the quasi-compact open sets form a sub-lattice that is a basis for the topology.
A {\em spectral map} between spectral spaces is a continuous map for which the
inverse image of a quasi-compact open is quasi-compact.

The frame-theoretic counterpart of a spectral space is a coherent frame.  Recall
from \cite[II.3.1]{Johnstone:Stone-spaces} that an element $c$ in a frame $F$ is called
\emph{finite} if whenever we have $c \leq \bigjoin_{s \in S} s$ for some subset
$S\subset F$, then there exists a finite subset $K\subset S$ such that already
$c \leq \bigjoin_{s \in K} s$.  A frame is \emph{coherent} when every element
can be expressed as a join of finite elements, and the finite elements form a
sub-lattice.  This amounts to requiring that $1$ is finite and that the meet of
two finite elements is finite.  A coherent frame is spatial~\cite[Thm.~II.3.4]{Johnstone:Stone-spaces}. 
A frame map is called {\em coherent} if it
takes finite elements to finite elements.

A coherent frame $F$ can be reconstructed
from its sub-lattice $F^\omega$ of finite elements~\cite[Proposition
II.3.2]{Johnstone:Stone-spaces}: $F$ is canonically isomorphic
to the frame of {\em ideals} in the lattice $F^\omega$.  Recall that an 
{\em ideal} of a lattice is a down-set, closed under finite joins. 
% Taking ideals amounts to the join completion of a lattice.
% 
% Every element generates a principal
% ideal, and every finite collection of elements generates an ideal which is the
% principal ideal of their join.  Introducing arbitrary ideals is the free
% join-completion procedure, because although the join of infinitely many elements
% may not exist as an element in the lattice, these elements still generate an ideal. 

The
functors `taking-ideals' (which amounts to join completion) 
and `taking-finite-elements' constitute an equivalence
of categories between distributive lattices and coherent frames (with coherent
maps)~\cite[Corollary II.3.3]{Johnstone:Stone-spaces}.
Altogether the relationships are summarized in the following theorem known as 
Stone duality:

%%%%%%%%%%%%%%%%%%%%%%%%%%%%%%%%%%%%%%%%%%%%%
\begin{theorem}[Stone 1939; Joyal~\cite{Joyal:NAMS1971}, 1971]
%%%%%%%%%%%%%%%%%%%%%%%%%%%%%%%%%%%%%%%%%%%%%
\label{spsp=cohfr=dlat}
  The category of spectral spaces and spectral maps is contravariantly
  equivalent to the category of coherent frames and coherent maps, which in turn
  is equivalent to the category of distributive lattices.
\end{theorem}
%%%%%%%%%%%%%%%%%%%%%%%%%%%%%%%%%%%%%%%%%%%%%

%%%%%%%%%%%%%%%%%%%%%%%%%%%%%%%%%%%%%%%%%%%%%
\subsubsection{Hochster duality}
%%%%%%%%%%%%%%%%%%%%%%%%%%%%%%%%%%%%%%%%%%%%%

For a spectral space $X$, Hochster~\cite{MR0251026}
con\-struc\-ted a new topology on $X$ by 
taking as basic open subsets the closed sets with quasi-compact complements.
This space $X^\vee$ is called the {\em Hochster dual} of $X$.
(When $X$ is a noetherian space, the Hochster open sets can also be characterized
as the specialization-closed subsets.)
He proved that $X^\vee$ is spectral again, and that $X^\vee{}^\vee \simeq X$.
% This relationship is called Hochster duality.

Hochster duality becomes a triviality in the setting of distributive lattices 
and frames: under the equivalences of Theorem~\ref{spsp=cohfr=dlat}, a spectral space 
$X$ corresponds to a coherent frame $F$ (the frame of open sets in $X$) and to a 
distributive lattice $F^\omega$ (the finite elements in $F$, or equivalently,
the lattice of quasi-compact open sets in $X$).  Now the following definitions
match the topological ones:

\begin{definition}\label{def hochdualframe}
The Hochster dual of a distributive lattice is simply the opposite 
lattice. The Hochster dual of a coherent frame $F$ is the ideal
lattice of $(F^\omega)\op$, i.e.~its join completion (corresponding to the
way Hochster defined the dual by generating a topology from the closed sets
with quasi-compact complement).
\end{definition}

%%%%%%%%%%%%%%%%%%%%%%%%%%%%%%%%%%%%%%%%%%%%%
\subsubsection{Points}\label{subsubsec points}
%%%%%%%%%%%%%%%%%%%%%%%%%%%%%%%%%%%%%%%%%%%%%
% 
% Even if the main focus in this work is on point-free methods, it is also useful
% to understand points, as they often give handier computational tools.  Let us
% briefly review how to get the points in a sober space from a frame-theoretic
% point of view; \Joachim{more detailed exposition} can be found in \cite[Chap.
% I]{Johnstone:Stone-spaces}.
% 
% %%%%%%%%%%%%%%%%%%%%%%%%%%%%%%%%%%%%%%%%%%%%%
% \begin{definition}
% %%%%%%%%%%%%%%%%%%%%%%%%%%%%%%%%%%%%%%%%%%%%%
%   Given a frame $F$, a \emph{point} $x$ of $F$ is a frame homomorphism $x: F
%   \rightarrow \{0,1\}$, where $\{0,1\}$ is the lattice of open sets of the 
%   one-point topological space.
% \end{definition}
% %%%%%%%%%%%%%%%%%%%%%%%%%%%%%%%%%%%%%%%%%%%%%
%    

The points of a frame $F$ correspond bijectively to
\emph{prime ideals} of $F$, that is, ideals $\mathcal I$ for which $ 1 \in
\mathcal{I}^c$ and ($a \meet b \in \mathcal{I} \Rightarrow a \in
\mathcal{I}$ or $ b \in \mathcal{I}$).  Namely, to a point $x: F \to 
\{0,1\}$ corresponds the prime ideal $x^{-1}(0)$.  In a frame, every prime 
ideal $\mathcal P$ is principal: the {\em generating element} is 
$u_{\mathcal P} := \bigjoin_{b \in \mathcal{P}} b$, then we have
$$
\mathcal{P} = \idealgen{u_{\mathcal P}} = \{ b \in F \mid b \leq  u_{\mathcal 
P}\}.
$$
% In particular, every prime ideal is principal.
% 
% % Topologically this corresponds to the simple statement that given a point
% % $x$ in a sober space $X$, then there exists a largest open set that does not
% % contain $x$, namely the complement of the closure of $x$.
A frame element generating a prime ideal is called a {\em prime element}.

It should be noted that for $F$ a coherent frame, the points of $F$ are in
natural bijection with the points of the Hochster dual frame $F^\vee$.

%%%%%%%%%%%%%%%%%%%%%%%%%%%%%%%%%%%%%%%%%%%%%%%%%%
\subsection{The Zariski frame}
%%%%%%%%%%%%%%%%%%%%%%%%%%%%%%%%%%%%%%%%%%%%%%%%%%

Let $R$ be a commutative ring.  The prime spectrum $\Spec_Z(R)$ (with the
Zariski topology) is by definition a spectral space.  The corresponding coherent
frame of open subsets of $\Spec_Z(R)$ is called the {\em Zariski frame} of $R$;
it can be described directly as the frame of radical ideals in $R$: the join of
a family of radical ideals is the radical of the ideal generated by their union,
and the bottom element is $\sqrt{0}$; the meet is intersection in the ring $R$,
and the top element is $\sqrt{1}=R$.  We denote this frame by $\RadId(R)$.  The
finite elements in $\RadId(R)$ are the radicals of finitely generated ideals.
These form the distributive lattice $\RadId(R)^\omega = \fgRadId(R)$, called the
\emph{Zariski lattice}.

% Checking that indeed the meet of two finite  elements is again finite amounts to Lemma~\ref{lem forradcapisprod} hereafter.
% In conclusion,
% $$
% \RadId(R)^\omega = \fgRadId(R).
% $$
% 
% %%%%%%%%%%%%%%%%%%%%%%%%%%%%%%%%%%%%%%%%%%%%%
% \begin{lemma}\label{lem forradcapisprod}
% %%%%%%%%%%%%%%%%%%%%%%%%%%%%%%%%%%%%%%%%%%%%%
% Let $\mathfrak{a}$ and $\mathfrak{b}$ be radical ideals in $R$,
% then $\mathfrak{a}\cap \mathfrak{b} = \sqrt {\mathfrak{a} \cdot \mathfrak{b}}$. 
% In particular if $I$ and $J$ are finitely generated ideals of $R$ then
% $\sqrt{I} \cap \sqrt{J} = \sqrt{I\cdot J}$ is the radical of a finitely generated ideal.
% \end{lemma}
% %%%%%%%%%%%%%%%%%%%%%%%%%%%%%%%%%%%%%%%%%%%%%
% \begin{proof}
% It   is clear that 
% $\mathfrak{a} \cdot \mathfrak{b} \subset \mathfrak{a}\cap \mathfrak{b}$,
% hence $\sqrt{\mathfrak{a} \cdot \mathfrak{b}}  \subset 
% \sqrt{\mathfrak{a} \cap \mathfrak{b}} = \mathfrak{a} \cap \mathfrak{b}$.
% For an element  $x \in \mathfrak{a} \cap \mathfrak{b}$ we have that 
% $x^2 \in \mathfrak{a} \cdot \mathfrak{b}$, and hence 
% $x \in \sqrt{\mathfrak{a} \cdot \mathfrak{b}}$.
% \end{proof}

Under the above correspondences between notions of point, the
points of the Zariski frame coincide with the usual points of $\Spec_Z(R)$,
the prime ideals of $R$. Precisely:

%%%%%%%%%%%%%%%%%%%%%%%%%%%%%%%%%%%%%%%%%%%%
\begin{lemma}\label{lem pointHochster}
%%%%%%%%%%%%%%%%%%%%%%%%%%%%%%%%%%%%%%%%%%%%
Given a frame-theoretic prime ideal $\mathcal{P} \subset \RadId(R)$,
its generating element $u_{\mathcal P} = \bigjoin_{I \in \mathcal{P}} I \in 
\RadId(R)$ is a prime ideal of the ring $R$.
Conversely, any ring-theoretic prime ideal $\mathfrak p \subset R$
% defines a frame prime ideal $\{b\in \RadId(R) \mid b \leq \mathfrak p\}$.
defines a frame prime ideal $\{b\in \RadId(R) \mid b \subset \mathfrak p\}$.
\end{lemma}
%%%%%%%%%%%%%%%%%%%%%%%%%%%%%%%%%%%%%%%%%%%%

% The Hochster dual frame does not admit so explicit a description as the Zariski 
% frame itself.  We will mostly deal with it in terms of the distributive lattice
% $\fgRadId(R)$.  As a point set, the Hochster dual space, denoted 
% $\Spec_H R$, is the same set of points as $\Spec_Z R$,
% but the topology is generated by the Zariski closed sets with quasi-compact
% complement. 

%%%%%%%%%%%%%%%%%%%%%%%%%%%%%%%%%%%%%%%%%%%%%
\begin{definition}\label{def classSupport}
%%%%%%%%%%%%%%%%%%%%%%%%%%%%%%%%%%%%%%%%%%%%%
(Joyal~\cite{Joyal:Cahiers1975}, 1975.)
A {\em support} for $R$ (with
values in a frame) is a pair $(F,d)$ where $F$ is a frame and $d$ is a map
$d: R \to F$ satisfying 
\begin{align*}
d(1)&=1 &  d(fg)&=d(f)\wedge d(g)\\
d(0)&=0 & d(f+g)&\leq d(f) \vee d(g).
\end{align*}
A {\em morphism of supports} is a frame map
compatible with the map from $R$.
% \end{definition}

% %%%%%%%%%%%%%%%%%%%%%%%%%%%%%%%%%%%%%%%%%%%%%
% \begin{definition}
% %%%%%%%%%%%%%%%%%%%%%%%%%%%%%%%%%%%%%%%%%%%%%
The {\em Zariski support} is the frame of radical ideals, with the map
\begin{eqnarray*}
  R & \longrightarrow & \RadId(R)  \\
  f & \longmapsto & \sqrt{(f)} .
\end{eqnarray*}
\end{definition}
%%%%%%%%%%%%%%%%%%%%%%%%%%%%%%%%%%%%%%%%%%%%%

%%%%%%%%%%%%%%%%%%%%%%%%%%%%%%%%%%%%%%%%%%%%%
\begin{theorem}[Joyal~\cite{Joyal:Cahiers1975}, 1975]\label{thm JoyalSupIsinit}
 %%%%%%%%%%%%%%%%%%%%%%%%%%%%%%%%%%%%%%%%%%%%
For any commutative ring  the Zariski support is the  initial support.
\end{theorem}
%%%%%%%%%%%%%%%%%%%%%%%%%%%%%%%%%%%%%%%%%%%%%
 
In other words, for any support $d:R\to F$, there is a unique frame map 
$\RadId(R) \to F$  making the following diagram commute:
$$\xymatrix{
R  \ar[rr]^d  \ar[rd]_{\sqrt{\phantom{i}}} && F    \\
&\RadId (R).\ar@{..>}[ru]_{\exists!}  }$$
This map is defined as $J\mapsto \bigvee \{d(f) \mid f\in J\}$.

Equivalently, this result can be formulated in terms of distributive lattices:
the initial distributive-lattice-valued support is the Zariski lattice
$\fgRadId(R)$.  In fact Joyal constructed this distributive lattice
syntactically by freely generating it by symbols $d(f)$ and dividing out by the
relations.  The syntactic Zariski lattice, often called the Joyal spectrum, is a
cornerstone in constructive commutative algebra (Hilbert's program), see for
example Coquand-Lombardi-Schuster~\cite{MR2595208} and the many references
therein.

%%%%%%%%%%%%%%%%%%%%%%%%%%%%%%%%%%%%%%%%%%%%
%%%%%%%%%%%%%%%%%%%%%%%%%%%%%%%%%%%%%%%%%%%%
\section{Localizing subcategories in $D(R)$}\label{sec LocsubcatDR}
%%%%%%%%%%%%%%%%%%%%%%%%%%%%%%%%%%%%%%%%%%%%
%%%%%%%%%%%%%%%%%%%%%%%%%%%%%%%%%%%%%%%%%%%%
Throughout this section we fix a commutative ring $R$, and we work in the
derived category $D(R)$.  We will write our complexes homologically:
differentials lower degree by one.

%%%%%%%%%%%%%%%%%%%%%%%%%%%%%%%%%%%%%%%%%%%%
\subsection{Compactly generated localizing categories and ${\Spec_{H} R}$}
%%%%%%%%%%%%%%%%%%%%%%%%%%%%%%%%%%%%%%%%%%%%

Recall that the compact objects in $D(R)$ are precisely the perfect complexes,
i.e.~quasi-isomorphic to bounded complexes of finitely generated projective 
modules, a.k.a.~strictly perfect complexes (for a precise definition of perfect complexes see \cite[Def. 2.2.10]{Thomason-Trobaugh} and for the relation to strictly perfect complexes see \cite[Thm. 2.4.3]{Thomason-Trobaugh}).
Since there is only a set of isomorphism classes of finitely
projective modules, there is only a set of compact objects in $D(R)$, up to
isomorphism.  As a consequence there is only a set of compactly generated
localizing subcategories in $D(R)$, and they are naturally ordered by inclusion.
The bottom element is clearly $\Loc(0) = \{0\}$ and the top element is $\Loc(R)
= D(R)$.  We wish now to understand this poset.

We first turn to a local description of these localizing subcategories. It
happens that in order to find nice parameterizing objects for the localizing
subcategories it is important not to stick to perfect complexes but also to
consider the non-compact objects in $D(R)$.

%%%%%%%%%%%%%%%%%%%%%%%%%%%%%%%%%%%%%%%%%%%%%
\subsubsection{Local structure of compactly generated localizing categories}
%%%%%%%%%%%%%%%%%%%%%%%%%%%%%%%%%%%%%%%%%%%%%

The following result due to Dwyer and Greenlees is an important 
building block in the present 
work.
%%%%%%%%%%%%%%%%%%%%%%%%%%%%%%%%%%%%%%%%%%%%%
\begin{proposition}[\cite{MR1879003}, Proposition~6.4] 
% \label{DG-6.10}
% Joachim{The correct reference for this result is in fact 6.4:
  \label{DG-6.4}
%%%%%%%%%%%%%%%%%%%%%%%%%%%%%%%%%%%%%%%%%%%%%
  If $I\subset R$ a finitely generated ideal then $\Loc(R/I) = \Loc(K(I))$,
  where $K(I)$ is any Koszul complex of the ideal $I$.
\end{proposition}
%%%%%%%%%%%%%%%%%%%%%%%%%%%%%%%%%%%%%%%%%%%%%
Recall that the Koszul complex of an element
$f\in R$ is the perfect complex
\[
\xymatrix{
K(f) :  0 \ar[r] & R \ar[r]^f & R \ar[r] & 0 \\
}
\]
where the source of $f$ is in degree $0$. Given a finite family of elements 
$f_1, \dots, f_n$ in $R$, 
the Koszul complex of this family is by definition the complex
\[
K(f_1) \otimes \cdots \otimes K(f_n).
\]

A Koszul complex $K(I)$ for a finitely generated ideal $I$ is the Koszul complex
of any of its finite generating subsets.  Complexes of the form $R/I$ in fact 
abound in localizing subcategories as shown by the next result:

%%%%%%%%%%%%%%%%%%%%%%%%%%%%%%%%%%%%%%%%%%%%%
\begin{lemma}[\cite{MR2970877}]\label{lem cycquot}
%%%%%%%%%%%%%%%%%%%%%%%%%%%%%%%%%%%%%%%%%%%%%
Let $E$ be a chain complex and $I \subset R$ an ideal. 
Assume that there  is a chain map $f : E \to  R/I$ 
that induces an epimorphism in homology. Then $R/I \in  \Loc(E)$.
\end{lemma}
%%%%%%%%%%%%%%%%%%%%%%%%%%%%%%%%%%%%%%%%%%%%%

% We will now generalize the Dwyer-Greenlees result by showing that in fact any
We now build up to Proposition~\ref{thm comp-celleq-cyc} which generalizes
the Dwyer-Greenlees result by showing that in fact any
perfect complex $C$ is cellularly equivalent to some quotient $R/I$ for some
finitely generated ideal $I$.

%%%%%%%%%%%%%%%%%%%%%%%%%%%%%%%%%%%%%%%%%%%%%
\begin{proposition}\label{prop comp-cycquot}
%%%%%%%%%%%%%%%%%%%%%%%%%%%%%%%%%%%%%%%%%%%%%
Let $C$ be a complex such that its homology $H_ \ast(C)$ is a finitely generated $R$-module. Then there exist finitely many ideals $J_1,\dots,J_m$ in $R$ such that
$$\Loc(C) = \Loc(R/J_1, \dots, R/J_m).$$
\end{proposition}
%%%%%%%%%%%%%%%%%%%%%%%%%%%%%%%%%%%%%%%%%%%%%
\begin{proof}
We argue by induction on the minimal number $d$ of homogenous generators of
$H_\ast C$.  If $d=0$, $C$ is acyclic, hence quasi-isomorphic to the $0$ complex
and $J=R$ is the desired ideal.

If $d=1$, then $C$ is quasi-isomorphic to a cyclic module.  Indeed, without loss
of generality we may assume that the only non-zero homology module is in degree
$0$.  We then have zig-zag of quasi-isomorphisms:
\[
\xymatrix{
\cdots \ar[r] & C_1 \ar[r]  & C_0 \ar[r] & C_{-1} \ar[r] & \cdots \\
\cdots \ar[r]  &C_1 \ar[r] \ar@{=}[u] \ar[d] & Z_0 \ar@{^(->}[u] \ar[r] 1\ar@{->>}[d] & 0 \ar[u] \ar@{=}[d] \ar[r] & \cdots \\
\cdots \ar[r] & 0 \ar[r] & H_0C \ar[r] & 0 \ar[r] & \cdots 
}
\]
and since in $D(R)$ we have $C \simeq H_0 C \simeq R/J$ for some ideal $J$, in this case the Proposition is trivial.

Assume that the Proposition has been proved for all complexes with homology generated by less that $d \geq 1$ homogeneous generators. Let $C$ be a complex with homology generated by $d+1$ homogeneous generators $x_1, \cdots, x_{d+1}$, which we may assume are ordered in decreasing homological degree. 
Without loss of generality we may also assume that $x_{d+1}$ is in degree $0$.
As  in the case $d=1$  we may also assume that $C$ is in fact zero in degree $< 0$. If $x_{j}, \dots ,x_{d+1}$ are the generators of $H_0 C$, denote by $D$ the submodule generated by $x_{j}, \dots, x_{d}$. Then we have a chain map, which by construction is an epimorphism in homology:
\[
\xymatrix{
\cdots \ar[r] &C_1 \ar[r]  \ar[d] & C_0  \ar[r] \ar@{->>}[d] & 0  \ar[d] \ar[r] & \cdots \\
\cdots \ar[r] & 0 \ar[r] \ar[d] & H_0C \ar@{->>}[d] \ar[r] &  0 \ar[d] \ar[r] & \cdots  \\
\cdots \ar[r] & 0  \ar[r] & H_0 C/D \simeq R/J_{d+1} \ar[r] & 0 \ar[r] & \cdots 
}
\]
Let $C' =  \ker(C \rightarrow R/J_{m+1})$. By direct inspection, the homology of the complex $C'$ is given by $H_{n} C' = H_{n} C$ if $n \neq 0$, and $H_0 C' = H_0 C/D = R/J_{d+1}$, in particular $H_{\ast} C'$ is generated by $d$ elements. The short exact sequence of complexes 
\[
\xymatrix{ 0 \ar[r] & C' \ar[r] & C \ar[r] & R/J_{d+1} \ar[r]  & 0}
\]
induces a triangle
\[
\xymatrix{
C' \ar[r] & C \ar[r] & R/J_{d+1} \ar[r] & \Sigma C',
}
\]
where the middle arrow is  surjective in homology by construction.
From Lemma~\ref{lem cycquot} we conclude that $R/J_{d+1} \in \Loc(C)$, 
and from the triangle that $C' \in \Loc(C)$. In particular 
$\Loc(C',R/J_{d+1} ) \subset \Loc(C)$. Conversely, the triangle 
shows that $\Loc(C',R/J_{d+1} ) \supset \Loc(C)$ and we conclude by
applying the induction hypothesis to the complex $C'$.
\end{proof}
Another important result by Dwyer--Greenlees provides a characterization
of the complexes in $\Loc(R/I)$:

% 2.1.5
%%%%%%%%%%%%%%%%%%%%%%%%%%%%%%%%%%%%%%%%%%%%%
\begin{lemma}[\cite{MR1879003}, Proposition 6.12]\label{lem blongstoLR/I}
%%%%%%%%%%%%%%%%%%%%%%%%%%%%%%%%%%%%%%%%%%%%%
Let $I \subset R$ be a finitely generated ideal.
Then a complex $E$ belongs to $\Loc(R/I)$ if and only if
for any $x \in H_*(E)$ there exists an 
integer $p \in \N$ such that $I^p\cdot x = 0$.
\end{lemma}
%%%%%%%%%%%%%%%%%%%%%%%%%%%%%%%%%%%%%%%%%%%%%

% 2.1.10
%%%%%%%%%%%%%%%%%%%%%%%%%%%%%%%%%%%%%%%%%%%%%
\begin{corollary}\label{cor inc}
%%%%%%%%%%%%%%%%%%%%%%%%%%%%%%%%%%%%%%%%%%%%%
For any two finitely generated ideals $I,  J$ in $R$  we have
  $$
\sqrt I \subset \sqrt  J   \Longleftrightarrow \Loc(R/J) \subset \Loc(R/I).
  $$
\end{corollary} %%%%%%%%%%%%%%%%%%%%%%%%%%%%%%%%%%%%%%%%%%%%%
\begin{proof}
  According to Lemma~\ref{lem blongstoLR/I}, $\Loc(R/J) \subset \Loc(R/I)$ if
  and only if $R/J$ is an $I$-torsion complex, and this happens if and only if
  $\exists n \in \N$ such that $I^n \subset J$ and hence if and only if $\sqrt
  I= \sqrt{I^n} \subset \sqrt J$.
\end{proof}

% The next two results are trivial consequences of this Proposition; we single them out
% for future convenience.

% 2.1.11
%%%%%%%%%%%%%%%%%%%%%%%%%%%%%%%%%%%%%%%%%%%%
\begin{corollary}\label{cor rad}
%%%%%%%%%%%%%%%%%%%%%%%%%%%%%%%%%%%%%%%%%%%%
  For finitely generated ideals $I$ and $J$, we have $\sqrt I = \sqrt J$
  if and only if $\Loc(R/I) = \Loc(R/J)$.
\end{corollary}
%%%%%%%%%%%%%%%%%%%%%%%%%%%%%%%%%%%%%%%%%%%%

% 2.1.12
%%%%%%%%%%%%%%%%%%%%%%%%%%%%%%%%%%%%%%%%%%%%
\begin{corollary}\label{cor contain-I}
%%%%%%%%%%%%%%%%%%%%%%%%%%%%%%%%%%%%%%%%%%%%
  Let $I$  be a finitely generated ideal in $R$, then 
  $$
  \Loc( R/J \mid J \supset I, J \text{ fin.~gen.}) = \Loc(R/I) .
  $$
\end{corollary}
%%%%%%%%%%%%%%%%%%%%%%%%%%%%%%%%%%%%%%%%%%%%

In \cite[Proposition 6.11]{MR1879003}, Dwyer and Greenlees show
that the cellularization of a module $M$ with respect to $R/I$ computes the
$I$-local cohomology of $M$. In particular, Corollary~\ref{cor rad} has the
following well-known interpretation in terms of local cohomology.
Denote by $H_{\ast}^I(M)$ the $I$-local cohomology of an $R$-module $M$.

% 2.1.13
%%%%%%%%%%%%%%%%%%%%%%%%%%%%%%%%%%%%%%%%%%%%
\begin{proposition}\label{propo loccohointerp}
%%%%%%%%%%%%%%%%%%%%%%%%%%%%%%%%%%%%%%%%%%%%
  Let $M$ be an $R$-module and $I$, $J$ two finitely generated ideals.
  If $\sqrt{I} = \sqrt{J}$, then there is a canonical isomorphism $H_{\ast}^I(M)
\simeq H_{\ast}^J(M)$.
\end{proposition}
%%%%%%%%%%%%%%%%%%%%%%%%%%%%%%%%%%%%%%%%%%%%
Notice that in the noetherian case this isomorphism is proved by showing that
both terms are isomorphic to $H_{\ast}^{\sqrt{I}}(M)$, so that these
isomorphisms are induced by the inclusions $ I\subset \sqrt{I} = \sqrt{J}
\supset J$, but this is definitely not true for non-noetherian rings.  In the
non-noetherian case, the isomorphisms are induced by the inclusions $I \subset I
+J \supset J$, for if $\sqrt{I} = \sqrt{J}$, then $\sqrt{I} = \sqrt{I+ J} =
\sqrt{J}$.
%%%%%%%%%%%%%%%%%%%%%%%%%%%%%%%%%%%%%%%%%%%%

% 2.1.6
%%%%%%%%%%%%%%%%%%%%%%%%%%%%%%%%%%%%%%%%%%%%%
\begin{proposition}\label{prop finiteR/I-to-singleR/J}
%%%%%%%%%%%%%%%%%%%%%%%%%%%%%%%%%%%%%%%%%%%%%
Let $I$ and $J$ be finitely generated ideals in $R$, then
$$
\Loc(R/I, R/J)=\Loc(R/I \oplus R/J) = \Loc(R/(I\cdot J)) = \Loc(R/I\cap J) .
$$
\end{proposition}
 %%%%%%%%%%%%%%%%%%%%%%%%%%%%%%%%%%%%%%%%%%%%%
\begin{proof}
The first equality is a triviality since localizing subcategories are closed
under retracts and direct sums, and the last because $\sqrt{I \cdot J}  = \sqrt{I \cap J}.$
For the second equality, first notice that both $R/I$ and $R/J$
are $R/(I\cap J)$-modules, so by Lemma~\ref{lem blongstoLR/I}, $\Loc(R/I, R/J)
\subset \Loc(R/(I \cap J))$.  For the reverse inclusion, consider the following
commutative diagram diagram of $R$-modules:
\[
\xymatrix{
0 \ar[r]  & I/ (I \cap J)  \ar[r] \ar@{=}[d] & R/I \cap J \ar[r] \ar[d] &R/J \ar[r] \ar[d] &  0 \\
 0 \ar[r] & (I+J)/J \ar[r] & R/J \ar[r]&  R/(I+J) \ar[r] & 0,
}
\]
where the left equality is one of the classical Isomorphism Theorems. 
This tells us that $R/(I \cap J)$ is a pullback in $R$-Mod of $R/J$, 
$R/I$ and $R/(I+J)$, and therefore that $R/(I \cap J)$ is in 
$D(R)$ a homotopy pullback of $R/J$, $R/I$ and $R/(I+J)$. 
Now, $R/(I+J)$ is both an $R/I$ and an $R/J$-module, so 
$R/(I+J)  \in \Loc(R/I,R/J)$, and since localizing subcategories
are closed under homotopy pullbacks, 
$R/(I \cap J) \in \Loc(R/I,R/J)$ as desired. 
\end{proof}

This Proposition together with  Proposition~\ref{prop comp-cycquot} yields immediately that:

% 2.1.7
%%%%%%%%%%%%%%%%%%%%%%%%%%%%%%%%%%%%%%%%%%%%%
\begin{corollary}\label{cor replacecompactbyquot}
%%%%%%%%%%%%%%%%%%%%%%%%%%%%%%%%%%%%%%%%%%%%%
  Let $S$ be a noetherian ring and let $C$ be a perfect complex over $S$.
  Then there exists a finitely generated ideal $J\subset S$ such that $\Loc(C) = \Loc(S/J)$.
\end{corollary}
%%%%%%%%%%%%%%%%%%%%%%%%%%%%%%%%%%%%%%%%%%%%%

To get rid of the noetherian assumption, we need the following classical result
(see the Appendix of \cite{MR1174255}):

% 2.1.8
%%%%%%%%%%%%%%%%%%%%%%%%%%%%%%%%%%%%%%%%%%%%%
\begin{lemma}\label{noeth-subring}
%%%%%%%%%%%%%%%%%%%%%%%%%%%%%%%%%%%%%%%%%%%%%
Let $R$ be a commutative ring and let $C$ be a perfect complex in $D(R)$.  Then
there exists a noetherian subring $S \subset R$ and a perfect complex $C_S$ in
$D(S)$ such that $C=C_S \otimes_S R$.
\end{lemma}
%%%%%%%%%%%%%%%%%%%%%%%%%%%%%%%%%%%%%%%%%%%%

% 2.1.9
%%%%%%%%%%%%%%%%%%%%%%%%%%%%%%%%%%%%%%%%%%%%%
\begin{proposition}\label{thm comp-celleq-cyc}
%%%%%%%%%%%%%%%%%%%%%%%%%%%%%%%%%%%%%%%%%%%%%
Let $R$ be a commutative ring, and let $C$ be a compact object in $D(R)$. 
Then there exists a finitely generated ideal $I \subset R$ such 
that $\Loc(C) = \Loc(R/I)$.
\end{proposition}
%%%%%%%%%%%%%%%%%%%%%%%%%%%%%%%%%%%%%%%%%%%%%
\begin{proof} By Lemma~\ref{noeth-subring}, we can find a noetherian subring
$S\subset R$ and a perfect complex $C_S$ such that $C_S \otimes_S R = C$.  For
this complex $C_S$, Corollary~\ref{cor replacecompactbyquot} provides us with a
finitely generated ideal $J \subset S$ such that $\Loc_S(C_S) = \Loc_S(S/J)$.
We claim that $JR \subset R$, the  $R$-ideal generated by the image of $J$,
in $R$ is the finitely generated ideal we are looking for.  To see this, note
first that by Proposition~\ref{DG-6.4} we have
$\Loc_S(C_S) = \Loc_S(S/J) = \Loc_S(K(J))$.  As $C_S$ and $K(J)$ are compact,
there exists a finite recipe as in 
Proposition~\ref{prop objincomLoc} in $D(S)$ 
to build $K(J)$ from $C_S$, say of length $n+1$,
\[
\xymatrix{
F_0 \ar[r]^{f_0} & F_1 \ar[r]^{f_1} & \cdots \ar[r]^{f_{n-1}} & F_n  .\\
}
\]

% Since perfect complexes are flat, the tensor product $C_S\otimes_S R = C$ 
% in the derived category can be computed using the usual tensor product of complexes. 
Applying the triangulated functor $- \otimes_S R : D(S) \to D(R)$
to the above sequence of maps shows that
$K(J) \otimes_S R$ belongs to $\Loc(C_S \otimes R) = \Loc(C)$.
The complex $K(J)$ is strictly perfect, hence flat, so to compute 
$K(J) \otimes_S R \in D(R)$, we may use the underived tensor product.
By direct
inspection, for any finite generating set of the ideal $J \subset S$ with
associated Koszul complex $K(J)$, the complex $K(J) \otimes_S R$ (underived)
is in fact equal
to a Koszul complex for the ideal $JR$, so $\Loc(R/JR) = \Loc(K(JR)) \subset
\Loc(C)$.  Exchanging the roles of $C_S$ and $K(J)$ in the above argument shows
in the same way that $\Loc(C) \subset \Loc(K(JR))= \Loc(R/JR)$.
\end{proof}

% Now that we know that compactly generated localizing categories are in fact
% generated by finitely presented cyclic modules, we turn to describe when two of
% these are equal.

% 2.1.14
%%%%%%%%%%%%%%%%%%%%%%%%%%%%%%%%%%%%%%%%%%%%
%%%%%%%%%%%%%%%%%%%%%%%%%%%%%%%%%%%%%%%%%%%%%
\subsubsection{The lattice of compactly generated localizing subcategories}
%%%%%%%%%%%%%%%%%%%%%%%%%%%%%%%%%%%%%%%%%%%%%
%%%%%%%%%%%%%%%%%%%%%%%%%%%%%%%%%%%%%%%%%%%%

We turn now to the global structure of the poset of
compactly generated localizing subcategories in $D(R)$. 
This poset has meets and arbitrary joins:
 if $\{\setofobjects_\alpha\}_{\alpha \in A}$ is a set of sets
 of compact objects, the join is given by
 \[
   \bigvee_{\alpha\in A} \Loc(\setofobjects_\alpha)  
   = \Loc(\bigcup_{\alpha\in A} \  \setofobjects_\alpha ).
 \]
 The meet is a bit more complicated, since a priori the intersection 
of $\Loc(\setofobjects_1)$ and $\Loc(\setofobjects_2)$ might not be compactly 
generated. Nevertheless, it contains a largest compactly generated subcategory, 
namely the localizing subcategory generated by the compact objects it contains:
\[
 \Loc(\setofobjects_1) \meet \Loc(\setofobjects_2) = \Loc 
 \Big(\big(\Loc(\setofobjects_1)\cap \Loc(\setofobjects_2) \big) \cap D^\omega(R)\Big).
\]

Denote by $\kat{CGLoc}(D(R))$ the lattice of compactly generated 
localizing subcategories of $D(R)$, and by  
$\kat{fgCGLoc}(D(R))$ the subposet of finite elements.
We shall see shortly that $\kat{fgCGLoc}(D(R))$ is a distributive lattice
and that
$\kat{CGLoc}(D(R))$
is a coherent frame.
We first characterize the finite elements:
\begin{lemma}
  The finite elements in $\kat{CGLoc}(D(R))$ are the localizing subcategories of
  $D(R)$ that can be generated by a single compact object.
\end{lemma}
\begin{proof}
  A localizing subcategory generated by a single compact object is easily seen
  to be finite by Proposition~\ref{cor compincompgen}.
  Conversely if 
  $\setofobjects$ is a set of compact objects and $\Loc(\setofobjects)$ is a 
  finite element in $\kat{CGLoc}(D(R))$, then the equality
  $\Loc(\setofobjects) = \bigjoin_{s\in S} \Loc(s)$ implies that also
  $\Loc(\setofobjects) = \bigjoin_{s\in K} \Loc(s) = \Loc(K)$ for some finite subset $K 
  \subset \setofobjects$.  Now for each $s\in K$ we have $\Loc(s)=\Loc(R/I_s)$
  for some finitely generated ideal $I_s$,  and Proposition~\ref{prop finiteR/I-to-singleR/J}
  shows that then $\Loc(K) = \Loc(R/J)$ where $J$ is the product of the ideals 
  $I_s$.
\end{proof}

We analyze the join and meet inside $\kat{fgCGLoc}(D(R))$:
in case the generating set is just a single compact object, we may replace it by
a cyclic module $R/I$.
For the join operation, Proposition~\ref{prop finiteR/I-to-singleR/J} gives
\[
 \Loc(R/I) \vee \Loc(R/J)  = \Loc(R/I, R/J) = \Loc(R/(I\cdot J)) .
\]
For the meet operation, the following shows 
that the intersection of two localizing subcategories
each generated by one compact object is again a compactly generated localizing
subcategory generated by a single compact, so in this case meet is just 
intersection:

%%%%%%%%%%%%%%%%%%%%%%%%%%%%%%%%%%%%%%%%%%%%
\begin{lemma}\label{lem interR/IetR/J}
%%%%%%%%%%%%%%%%%%%%%%%%%%%%%%%%%%%%%%%%%%%%
If  $I$ and $J$ are finitely generated ideals in $R$, then 
\[\Loc(R/I) \cap
\Loc(R/J) = \Loc(R/(I + J)) .
\]
\end{lemma}
%%%%%%%%%%%%%%%%%%%%%%%%%%%%%%%%%%%%%%%%%%%%
\begin{proof} For $E\in \Loc(R/I) \cap \Loc(R/J)$ and $x\in H_*(E)$,
by Lemma~\ref{lem blongstoLR/I} there exist $n,m \in \N$ such that $I^nx=0=J^m x$.  A direct computation
shows that $(I+ J )^{mn}x=0$, and therefore $E\in 
\Loc(R/(I + J))$.
Conversely, as $R/(I+ J)$ is both $I$ and $J$-torsion we
conclude that $\Loc(R/I) \cap \Loc(R/J) \supset \Loc(R/(I + J))$.
\end{proof}

The following result is a key point.
\begin{proposition}\label{prop main}
  The lattice $\kat{fgCGLoc}(D(R))$ of localizing subcategories 
generated by a single compact object is isomorphic to the opposite of the 
Zariski lattice:
\begin{eqnarray*}
  \kat{fgCGLoc}(D(R)) & \simeq & \fgRadId(R)\op  \\
  \Loc(R/I) & \leftrightarrow & \sqrt{I} .
\end{eqnarray*}
\end{proposition}

\begin{proof}
  The assignment from right to left, $I \mapsto \Loc(R/I)$, is well-defined by
  Proposition~\ref{DG-6.4}: $\Loc(R/I)=\Loc(K(I))$ which is compactly generated
  since $I$ is the radical of a finitely generated ideal, and 
  $\Loc(K(I))$ is
  insensitive to taking radical by Corollary~\ref{cor rad}.  The assignment
  $\Loc(R/I) \mapsto I$ is well-defined since by 
  Proposition~\ref{thm comp-celleq-cyc} for any perfect complex $C$ 
  there is a finitely generated
  ideal $I$ such that $\Loc(C) = \Loc(R/I)$, and by Corollary~\ref{cor inc},
  this finitely generated ideal is uniquely determined up to taking radical.
  Having established that the two assignments are well defined, it is obvious
  from the description that they constitute an inclusion-reversing bijection.
\end{proof}

We now extend this isomorphism to $\kat{CGLoc}(D(R))$.
We first give a rather formal argument, which relies on some results in
Section~\ref{sec ZarSpecTriaCat}, then give a more elementary proof
of more geometric flavor.
By Corollary~\ref{cor LocXTomega} we have
$\CGLoc(D(R)) = \Thick(D^\omega(R))$, and by \ref{lem locistensideal} 
and \ref{lem tensidIsrad}
all thick
subcategories are radical thick tensor ideals, so that $\CGLoc(D(R))$ is
a coherent frame by Theorem~\ref{thm coherence}.
\begin{theorem}\label{thm main}
  The isomorphism of Proposition~\ref{prop main} extends to an
  isomorphism of frames
$$
\CGLoc(D(R)) = \RadId(R)^\vee .
$$
\end{theorem}
\begin{proof}
  The two frames are precisely the ideal frames (join completions) of the
  distributive lattices in Proposition~\ref{prop main}.
\end{proof}

The description of the isomorphism in Proposition~\ref{prop main}, and hence
Theorem~\ref{thm main}, relies on Proposition~\ref{DG-6.4} and
Corollary~\ref{cor rad}.  We provide a more geometrical reformulation that lifts
this dependence:

%%%%%%%%%%%%%%%%%%%%%%%%%%%%%%%%%%%%%%%%%%%%
\begin{theorem}\label{thm maingeo}
%%%%%%%%%%%%%%%%%%%%%%%%%%%%%%%%%%%%%%%%%%%%
  There is a natural inclusion-preserving bijection
  $$
  \Big\{ \begin{tabular}{c}
          \text{\rm Compactly generated} \\ \text{\rm localizing subcategories of $D(R)$} 
         \end{tabular}
\Big\}
  \leftrightarrow
  \Big\{ \begin{tabular}{c}
  \text{\rm Hochster open}\\ \text{\rm sets in $\Spec(R)$}
  \end{tabular}
  \Big\}
  $$  
  The bijection is given from left to right by
  $$
 f: \Loc(\setofobjects)  \ \longmapsto \ 
 \bigcup_{\underset{I \text{ \rm fin.~gen.}}{R/I\in \Loc(\setofobjects)}} \zeros{I} ,
  $$
  and from right to left by
  $$
    \Loc( K(I) \mid \zeros{I} \subset U, I \text{ \rm fin.~gen.} )  \ \longmapsfrom \ U : g.
$$
 
\end{theorem}
%%%%%%%%%%%%%%%%%%%%%%%%%%%%%%%%%%%%%%%%%%%%
\begin{proof}
  Note first that the new description of $f$ agrees with that of 
  Proposition~\ref{prop main}: given a perfect complex $C$,
the subset 
$$
\bigcup_{R/J \in \Loc(C), \ J \text{ f.g.}}\zeros{J}
$$
is of the form $\zeros{I}$
for some finitely generated ideal $I$.  
By Proposition~\ref{thm comp-celleq-cyc},
there is a finitely generated ideal $I$ such that
$\Loc(C) = \Loc(R/I)$, and for any finitely generated ideal $J$, by
Corollary~\ref{cor inc},
\[
R/J \in \Loc(R/I) \Leftrightarrow \sqrt{J}
\supset \sqrt{I} \Leftrightarrow \zeros{J} \subset \zeros{I}.
\]
So indeed $f(\Loc(R/I))
= \zeros{I}$.

  We now check that  $f\circ g = \Id$.  Given an arbitrary Hochster open set $U$, to
show that $f \circ g (U) = U$, it is enough to prove that if $J$ is a finitely
generated ideal in $R$ such that $R/J \in \Loc\big( K(I) \mid \zeros{I} \subset U, I
\text{ fin.~gen.} \big )$, then $\zeros{J} \subset U$.  Choose a Koszul complex
$K(J)$ for the ideal $J$.  By hypothesis $K(J) \in \Loc\big( K(I) \mid \zeros{I}
\subset U, I \text{ fin.~gen.} \big )$, and since it is compact there exist
finitely many ideals $J_1, \dots, J_k$ such that $\zeros{J_k} \in U$ and $K(J) \in
\Loc(K(J_1), \dots , K(J_k))= \Loc(R/J_1, \dots, R/J_k)$.  But $\Loc (R/J_1,
\dots , R/J_k) = \Loc(R/(J_1 \dots J_k))$, so by Corollary~\ref{cor inc},
$$
\sqrt{J} \supset \sqrt{J_1 \cdots J_k},
$$
hence 
$$
\zeros{J} \subset \zeros{J_1 \cdots J_k} = \bigcup_{i=1}^k \zeros{J_i}\subset U.
$$

Finally we establish that $g \circ f = \Id$.  Since by 
Proposition~\ref{thm comp-celleq-cyc}, any
perfect complex is cellularly equivalent to a finitely
generated cyclic module, it is clear that for any compactly
generated category $\Loc (\setofobjects)$, we have $(g \circ f)(\Loc (\setofobjects)) \supset
\Loc (\setofobjects)$.
For the reverse inclusion we argue as follows.  Given a compactly generated
localizing subcategory $\Loc (\setofobjects)$, let $J$ be a finitely generated ideal such
that $\zeros{J} \subset \bigcup_{\underset{I \text{ fin.~gen.}}{R/I\in \Loc (\setofobjects)}}
\zeros{I}$.  Then, because Hochster opens of the form $\zeros{J}$ are finite elements in
the Hochster frame, there exist finitely many ideals $I_1, \dots, I_n$ such that
$R/I_j \in \Loc (\setofobjects)$ for $1 \leq j \leq n$ and $\zeros{J} \subset \bigcup_{j=1}^n
\zeros{I_j}$.  Again by Proposition~\ref{prop main},
\[
 \Loc(K(J)) \subset \Loc(K(I_1), \dots, K(I_n)) \subset \Loc (\setofobjects).
\]

\vspace*{-1.7\bigskipamount}

\end{proof}

%%%%%%%%%%%%%%%%%%%%%%%%%%%%%%%%%%%%%%%%%%%%
%%%%%%%%%%%%%%%%%%%%%%%%%%%%%%%%%%%%%%%%%%%%
\subsection{Hochster duality in $D(R)$ and $\Spec_Z R$}
%%%%%%%%%%%%%%%%%%%%%%%%%%%%%%%%%%%%%%%%%%%%
%%%%%%%%%%%%%%%%%%%%%%%%%%%%%%%%%%%%%%%%%%%%

Usually in algebraic geometry 
the topology of interest on $\Spec R$ is the Zariski topology (or closely 
related Grothendieck topologies such as the \'etale topology), not the Hochster dual topology.
As discussed in the Introduction, it is somewhat mysterious that 
compactly generated localizing subcategories in $D(R)$ yield the Hochster dual topology 
on $\Spec R$
(Theorem~\ref{thm main}), in spite of the fact that it is actually
a Zariski-like construction, as we shall see in 
Section~\ref{sec ZarSpecTriaCat}.

It is natural to ask whether also the  Zariski frame itself can be realized
as a sub-lattice inside $D(R)$. For the lattice of finite elements
$\Loc(R/I)$, the dual lattice can be
obtained by  passing to the right orthogonal categories.  We shall show how
to describe the join completion of this lattice inside $D(R)$.

By Proposition~\ref{prop LOcCisBous} localizing subcategories generated by sets of compact objects admit
Bousfield localizations, so for any set of compact objects $\setofobjects$, we have
${}^\bot(\Loc( \setofobjects)^\bot) = \Loc(\setofobjects)$.  In particular, \cite[Prop. 4.9.1(6)]{MR2681709}
%from Theorem~\ref{thm bousfieldloc} 
we have
the following order-reversing bijection of lattices:
\[
\xymatrix{
\big\{  \Loc(C) \mid C \text{ compact } \big\}  \ar@/^/[rr]^{(-)^\bot} & & 
\ar@/^/[ll]^{{}^\bot(-)} \big\{  \Loc(C)^{\bot} \mid C \text{ compact }\big\} \\
} 
\]
A priori on the right-hand side what we get are colocalizing subcategories as
explained in \ref{subsubsec BousfieldLoc}, but in this specific case we get
categories that are also localizing:

%%%%%%%%%%%%%%%%%%%%%%%%%%%%%%%%%%%%%%%%%%%%
\begin{proposition}\label{prop dualcatisloc}
%%%%%%%%%%%%%%%%%%%%%%%%%%%%%%%%%%%%%%%%%%%%
Let $\setofobjects$ be a set of compact objects
in a tensor
triangulated category $\T$ admitting arbitrary sums.
Then $\Loc (\setofobjects)^\bot$ is both a colocalizing and
a localizing subcategory.
\end{proposition}
%%%%%%%%%%%%%%%%%%%%%%%%%%%%%%%%%%%%%%%%%%%%
\begin{proof}
  As the subcategory $\Loc (\setofobjects)^\bot$ is colocalizing, it is
  triangulated; we just have to prove that it is closed under arbitrary sums.
  Consider an arbitrary sum $\coprod_{j \in J}N_j$ of objects $N_j \in \Loc
  (\setofobjects)^\bot$.  Given a compact generator $C$ of $\Loc (\setofobjects)$,
  consider any map $f \in \Hom_{\T}(C,\coprod_{j \in J}N_j)$.  Since $C$
  is compact, there exists a finite subset $K \subset J$ such that $f$ factors
  via $f_K : C \rightarrow \coprod_{j \in K} N_j = \prod_{j \in K} N_j$.  
  But then $f_K \in  \Hom_\T(C, \prod_{j \in
  K} N_j ) = \prod \Hom_\T(C, N_j ) = 0$, and $f=0$.  In particular
  $\coprod_{j \in J}N_j \in \Loc (\setofobjects)^\bot$ as we wanted.
\end{proof}

In the Zariski spectrum of a ring
there is a basis of principal open sets given by complements of Zariski closed sets defined by a single element of the ring. We first determine their corresponding localizing subcategories.

%%%%%%%%%%%%%%%%%%%%%%%%%%%%%%%%%%%%%%%%%%%%
\begin{proposition}\label{prop orthOneGen}
%%%%%%%%%%%%%%%%%%%%%%%%%%%%%%%%%%%%%%%%%%%%
For an element $f \in R$, we have $\Loc(R/\idealgen{f})^\bot = \Loc(R_f)$.
% where $\idealgen{f}$ denotes the ideal generated by $f$, and $R_f$ the
% localization of $R$ at the multiplicative system generated by $f$.
Moreover $\Loc(R_f)$ is the essential image of the functor 
$D(R_f) \rightarrow D(R)$ induced by restriction of scalars along 
the canonical map $R \rightarrow R_f$.
\end{proposition}
%%%%%%%%%%%%%%%%%%%%%%%%%%%%%%%%%%%%%%%%%%%%
\begin{proof} 
%By Theorem~\ref{thm bousfieldloc}
From \cite[Prop. 4.9.1 and 4.10.1]{MR2681709}, $\Loc(R/\idealgen{f})^\bot$ is the essential
  image of the Bousfield localization functor associated to the compactly 
  generated $\Loc(R/\idealgen{f})$, and since 
$\Loc(R/\idealgen{f})$ is a tensor ideal (Lemma~\ref{lem locistensideal}), by
Theorem~\ref{thm tensorloc}, this localization functor is isomorphic to
$L_{R/\idealgen{f}}(R) \otimes - \;$.  We proceed to compute the fundamental
triangle:
\[
\xymatrix{
\Gamma_{R/\idealgen{f}}(R) \ar[r] &R \ar[r] & L_{R/\idealgen{f}}(R) \ar[r] & \Sigma \Gamma_{R/\idealgen{f}}(R).
}
\]
In~\cite{MR1879003}, Dwyer and Greenlees showed how to compute the complex
$\Gamma_{R/\idealgen{f}}(R)$.  For each power $f^k$, the Koszul complex
$K(f^k)$ is given by $R
\stackrel{f^k}{\longrightarrow}R$, and we may form an inductive system $K(f^k)
\rightarrow K(f^{k+1})$ via the commutative diagram:
\[
\xymatrix{R \ar@{=}[r] \ar[d]_{f^k} & R \ar[d]^{f^{k+1}} \\
R \ar[r]^{f} & R.
} 
\]
The homotopy colimit of these complexes is by definition the complex denoted by
$K^\bullet(f^\infty)$.  It is shown in \cite[Proposition~6.10]{MR1879003} that
$K^\bullet(f^\infty)$ is quasi-isomorphic to $\Gamma_{R/\idealgen{f}}(R)$ and
that it is also quasi-isomorphic to the complex $R \rightarrow R_f$, where $R$
is in degree $0$ and the map $K^\bullet(f^\infty) \rightarrow R$ is simply the
map
\[
\xymatrix{R \ar@{=}[d] \ar[r] & R_f \ar[d] \\
R \ar[r] & 0.
} 
\]
By direct computation using the long exact sequence in homology we get that
$L_{R/\idealgen{f}}(R)$ is quasi-isomorphic to the complex $R_f$ concentrated in degree
$0$.

Finally, following the discussion in the beginning of this proof, a complex $M$
belongs to $\Loc(R/\idealgen{f})^\bot$ if and only if it is quasi-isomorphic to $ L_{R/\idealgen{f}}(R)
\otimes M = R_f \otimes M \in \Loc(R_f)$.  The reverse inclusion is trivial.

To prove the last assertion, just notice that the functor $D(R_f) \rightarrow
D(R)$ induced by restriction of scalars along $R \rightarrow R_f$ is exact,
commutes with both products and sums and that $R_f$ generates $D(R_f)$ as
a triangulated category with infinite sums.
\end{proof}

From this result it is easy to extract a criterion for a complex to be in
$\Loc(R/\idealgen{f})^\bot$, much in the spirit of Lemma~\ref{lem blongstoLR/I}:

%%%%%%%%%%%%%%%%%%%%%%%%%%%%%%%%%%%%%%%%%%%%
\begin{lemma}\label{lem belongstoLocRf}
%%%%%%%%%%%%%%%%%%%%%%%%%%%%%%%%%%%%%%%%%%%%
For $f \in R$, a complex belongs to $\Loc(R/\idealgen{f})^\bot= \Loc(R_f)$ if and only if
its homology modules are $R_f$-modules.
\end{lemma}
%%%%%%%%%%%%%%%%%%%%%%%%%%%%%%%%%%%%%%%%%%%%
\begin{proof}    By Bousfield
localization we know that $\Loc(R_f)$ is the essential image of the functor $R_f
\otimes -$. Given an arbitrary complex $M$, the K\"unneth spectral sequence
that computes the homology of $R_f \otimes M$ collapses onto the horizontal edge
at the page $E^2$ because $R_f$ is flat.  We conclude that 
$M \rightarrow R_f \otimes M$ is a quasi-isomorphism
if and only if for each $n \in \Z$ we have $H_n(M) \simeq R_f \otimes H_n(M)$,
and this happens precisely when the homology modules of $M$ are $R_f$-modules.
\end{proof}

We also get a description of $\Loc(R/I)^\bot$ for an arbitrary finitely
generated ideal $I$:

%%%%%%%%%%%%%%%%%%%%%%%%%%%%%%%%%%%%%%%%%%%%
\begin{theorem}\label{thm orthfinitgen}
 %%%%%%%%%%%%%%%%%%%%%%%%%%%%%%%%%%%%%%%%%%%%
For a finitely generated ideal $I = \idealgen{f_1,\ldots,f_n}$ in $R$,
we have
$$\Loc(R/I)^\bot = \Loc(R_{f_1},\dots, R_{f_n}).$$
\end{theorem}
%%%%%%%%%%%%%%%%%%%%%%%%%%%%%%%%%%%%%%%%%%%%
\begin{proof}
First, we have $\Loc(R/I)^\bot=(\Loc(R/\idealgen{f_1}) \cap \cdots 
\cap \Loc(R/\idealgen{f_n}))^\bot$, so
$R_{f_1},\dots,R_{f_n}$ all belong to $\Loc(R/I)^\bot$, and since this is a
localizing subcategory, we have $\Loc(R_{f_1},\dots, R_{f_n}) \subset \Loc(R/I)^\bot$.
For the reverse inclusion consider again the localization triangle:
\[
 \xymatrix{
\Gamma_{R/I}(R) \ar[r] & R \ar[r] & L_{R/I}(R) \ar[r] & \Sigma \Gamma_{R/I}(R) 
 }
\]
By (the proof of) Proposition~\ref{prop orthOneGen},
a model for $L_{R/I}(R)$ is
$K^\bullet(f_1^\infty)\otimes \cdots \otimes K^\bullet(f_n^\infty)$.
Since each of these
complexes is quasi-isomorphic to a flat complex, 
$K^\bullet(f_i^\infty) \simeq (R \rightarrow R_{f_i})$, the
derived tensor may be computed using the  ordinary tensor product of complexes.  This is
then a complex of the following form:
\[
\xymatrix{
 R \ar[r] & R_{f_1} \oplus \cdots \oplus R_{f_n} \ar[r] & \cdots \ar[r] & R_{f_1\cdots f_n} 
}
\]
where in degree $-p+1$ we have the direct sum of the $\binom{n}{p}$ modules
obtained by choosing $p$ elements among the $n$ generators and localizing $R$ at
their product.  Then the cone of the map of complexes
\[
\xymatrix{
 R \ar@{=}[d] \ar[r] & R_{f_1} \oplus \cdots \oplus R_{f_n} 
 \ar[r] & \cdots \ar[r] & R_{f_1\cdots f_n} \\
 R &&&
}
\]
is quasi-isomorphic to the suspension of the complex 
\[
\xymatrix{
 R_{f_1} \oplus \cdots \oplus R_{f_n} \ar[r] & \cdots \ar[r] & R_{f_1\cdots f_n} 
}
\]
and this is clearly an element in $\Loc(R_{f_1},\dots, R_{f_n})$.  Since this
localizing subcategory is a tensor ideal, we find that $\Loc(R/I)^\bot$, the 
essential image of the Bousfield
localization functor $L_{R/I}(M)= L_{R/I}(R) \otimes M$, 
is contained in $\Loc(R_{f_1},\dots, R_{f_n})$.
\end{proof}

Observe that  for any finite set $f_1, \dots,f _n$ the localizing subcategory $
\Loc(R_{f_1},\dots, R_{f_n})$ only depends on the radical ideal generated by
$f_1, \dots, f_n$. For future reference we record two immediate consequences: 
 
%%%%%%%%%%%%%%%%%%%%%%%%%%%%%%%%%%%%%%%%%%%%
\begin{corollary}\label{cor raddescrLocRf}
%%%%%%%%%%%%%%%%%%%%%%%%%%%%%%%%%%%%%%%%%%%%
For any finitely generated ideal $J = \idealgen{f_1, \dots, f_n}$ we have:
\[
 \Loc(R_{f_1},\dots, R_{f_n}) = \Loc( R_f \mid f \in \sqrt{J}).
\]
\end{corollary}
%%%%%%%%%%%%%%%%%%%%%%%%%%%%%%%%%%%%%%%%%%%%

%%%%%%%%%%%%%%%%%%%%%%%%%%%%%%%%%%%%%%%%%%%%
\begin{corollary}\label{cor radidlsinLocRf}
%%%%%%%%%%%%%%%%%%%%%%%%%%%%%%%%%%%%%%%%%%%%
Let $I$ and $J$ be finitely generated ideals in $R$. Then 
\[
 \Loc( R_f \mid f \in \sqrt{I}) \subset \Loc( R_f \mid f \in \sqrt{J}) \Leftrightarrow \sqrt{I} \subset \sqrt{J}.
\]
\end{corollary}
%%%%%%%%%%%%%%%%%%%%%%%%%%%%%%%%%%%%%%%%%%%%
\begin{proof}
This follows readily from the fact that taking right orthogonal is an
order-reversing operation, and from Corollary~\ref{cor inc}.
\end{proof}

Summing up we have proved the following result.
%%%%%%%%%%%%%%%%%%%%%%%%%%%%%%%%%%%%%%%%%%%%
\begin{proposition}
%%%%%%%%%%%%%%%%%%%%%%%%%%%%%%%%%%%%%%%%%%%%
 The poset of localizing categories of the form $$\Loc( R_f \mid f \in
\sqrt{J}),$$
where $J$ is a finitely generated ideal, is isomorphic to the
poset of radicals of finitely generated ideals in $R$.  In particular since the
latter is a distributive lattice so is the former.
\end{proposition}
%%%%%%%%%%%%%%%%%%%%%%%%%%%%%%%%%%%%%%%%%%%%
Recalling that radicals of finitely generated ideals correspond to quasi-compact
open sets in $\Spec_Z R$, we have the following
topological formulation:
%%%%%%%%%%%%%%%%%%%%%%%%%%%%%%%%%%%%%%%%%%%%
\begin{proposition}\label{prop principaldual}
%%%%%%%%%%%%%%%%%%%%%%%%%%%%%%%%%%%%%%%%%%%%
There is a natural isomorphism of lattices between the
lattice of localizing subcategories of $D(R)$ generated by finitely many
localizations of the ring $R$ and the lattice of quasi-compact open 
sets in $\Spec_Z R$  given by
$$\Loc(R_{f_1},\dots, R_{f_n}) \longmapsto \bigcup_{i=1}^n D(f_i)$$
  and
  $$
  \Loc(R_{f} \mid D(f) \subset U) \longmapsfrom U.
  $$
\end{proposition}
%%%%%%%%%%%%%%%%%%%%%%%%%%%%%%%%%%%%%%%%%%%%

The join in
the lattice of localizing subcategories is given by
\begin{eqnarray*}
 \Loc( R_f \mid f \in \sqrt{J}) \join \Loc( R_f \mid f \in \sqrt{I})  
 & = & \Loc( R_f \mid f \in \sqrt{I + J}) \\
 & = & \Loc( R_f \mid f \in \sqrt{I} \cup \sqrt{J}).
\end{eqnarray*}
For the meet operation, we find
\begin{eqnarray*}
\Loc( R_f \mid f \in \sqrt{J}) \meet  \Loc( R_f \mid f \in \sqrt{I})  
& = & \Loc( R_f \mid f \in \sqrt{I}\cdot \sqrt{J}),
\end{eqnarray*}
and this is in fact intersection:

%%%%%%%%%%%%%%%%%%%%%%%%%%%%%%%%%%%%%%%%%%%%
\begin{lemma}\label{lem meetis inters}
%%%%%%%%%%%%%%%%%%%%%%%%%%%%%%%%%%%%%%%%%%%%
For any two finitely generated ideals $I$ and $J$ in $R$ we have
 \[
 \Loc( R_f \mid f \in \sqrt{I}\cdot \sqrt{J})  
= \Loc( R_f \mid f \in \sqrt{J}) \cap  \Loc( R_f \mid f \in \sqrt{I}).
\]
\end{lemma}
%%%%%%%%%%%%%%%%%%%%%%%%%%%%%%%%%%%%%%%%%%%%
\begin{proof}
Compute using the fact that we are dealing with right orthogonals
\begin{eqnarray*}
  \Loc( R_f \mid f \in \sqrt{I}\cdot \sqrt{J}) & = &  \Loc( R_f \mid f \in \sqrt{I\cdot J}) \\
  & = &  \Loc(R/(I \cdot J))^\bot \\
  & = & \Loc(R/I, R/J)^\bot \\
  & = & \{R/I , R/J\}^\bot \\
  & = & \{R/I\}^\bot \cap \{R/J\}^\bot \\
  & = & \Loc( R_f \mid f\in \sqrt{I}) \cap \Loc( R_f \mid f \in \sqrt{J}).
\end{eqnarray*}
\end{proof}

   Just as for the case of compactly generated localizing
   subcategories (Theorem~\ref{thm main}), we proceed to establish that the
   correspondence of Proposition~\ref{prop principaldual} extends by
   join-completion to a frame isomorphism, realizing the whole Zariski
   frame inside $D(R)$.  Again this amounts to dropping the ``finite
   generation'' assumption, considering now localizing categories
   generated by an arbitrary number of localizations of the ring $R$.

   We consider now subcategories
   $\Loc( R_f \mid f \in J)$, where a priori $J \subset R$ is an
   arbitrary subset.     The following lemma tells us that this is no more general than
   requiring $J$ to be a radical ideal.

%%%%%%%%%%%%%%%%%%%%%%%%%%%%%%%%%%%%%%%%%%%%
\begin{lemma}\label{lem locparbyradidl}
%%%%%%%%%%%%%%%%%%%%%%%%%%%%%%%%%%%%%%%%%%%%
Let $J$ be a arbitrary subset of $R$ and 
$\sqrt{J}$ the radical ideal it generates. Then 
\[
 \Loc( R_f \mid f \in J) = \Loc( R_f \mid f \in \sqrt{J}).
\]
\end{lemma}
%%%%%%%%%%%%%%%%%%%%%%%%%%%%%%%%%%%%%%%%%%%%
\begin{proof} 
  The inclusion $\Loc( R_f \mid f \in J) \subset \Loc( R_f \mid f \in 
  \sqrt{J})$ is clear.  To establish the other inclusion,
  we proceed as follows.  Let $g \in \sqrt{J}$ be an
  arbitrary element.  Then there exists a finite subset $K \subset J$ such that
  $g \in \sqrt{K}$.  By Corollary~\ref{cor raddescrLocRf}, we know that
  $R_g \in \Loc( R_f \mid f \in \sqrt{K})$.  Again by Corollary~\ref{cor
  raddescrLocRf} we have that $\Loc( R_f \mid f \in \sqrt{K}) = \Loc( R_f \mid f \in K)
  \subset \Loc( R_f \mid f \in J)$, whence the desired inclusion.
\end{proof}

From now on we will parametrize our categories by radical ideals, still denoted
as $\sqrt{J}$ to emphasize the \emph{radical} property.  Notice that in the
poset $\{\Loc( R_f \mid f \in \sqrt{J}) \}$, for any family of objects
$\{\Loc( R_f \mid f \in \sqrt{J_i})\}_{i \in I}$, the join is
\[
\bigjoin_{i \in I}\Loc( R_f \mid f \in \sqrt{J_i}) = \Loc( R_f \mid f \in \sqrt{\Sigma_{i \in I} J_i}).
\]

Concerning the meet operation, all we can say is that:
\begin{multline*}
\Loc( R_f \mid f \in \sqrt{J}) \meet \Loc( R_f \mid f \in \sqrt{I}) \\
\subset  \Loc( R_f \mid f \in \sqrt{J}) \cap \Loc( R_f \mid f \in \sqrt{I}),
\end{multline*}
with equality if and only if $\Loc( R_f \mid f \in \sqrt{J}) \cap \Loc( R_f \mid f \in
\sqrt{I})$ is an element in our poset; this is because $\Loc( R_f \mid f \in
\sqrt{J}) \cap \Loc( R_f \mid f \in \sqrt{I})$ is the largest localizing subcategory
contained in both $\Loc( R_f \mid f \in \sqrt{J})$ and $\Loc( R_f \mid f \in
\sqrt{I})$.  The last difficulty we have to cope with is to show that this frame
has as finite elements precisely the localizing subcategories parametrized by
radicals of finitely generated ideals.  

To prove this we need to be a little bit
more precise in our description of a localizing subcategory generated by a set
of objects $\setofobjects$.  Recall that for any ordinal $\alpha$, its cardinal
$\norm{\alpha}$ is the initial ordinal in the set of ordinals that can be put in
bijection with $\alpha$, so we can write $\norm{\alpha} \leq \alpha$ as
\emph{ordinals}.  Let us define a filtration of $\Loc(\setofobjects)$ as follows:

\begin{enumerate}
 \item $\Loc^0(\setofobjects)$ is the full subcategory consisting of the zero object.
 \item $\Loc^1(\setofobjects)$ is the full subcategory whose objects are 
 isomorphic to an arbitrary suspension of  elements  in $\setofobjects$.
 \item If $\alpha \geq 1$ is a successor ordinal,  then $\Loc^{\alpha + 1}(\setofobjects)$ 
 is the full subcategory consisting of objects that are 
\begin{itemize}
\item[--] either isomorphic to
 an arbitrary suspension of direct sums of less than
 $\norm{\alpha}$ objects  in $\Loc^\alpha(\setofobjects)$,
 \item[--] or isomorphic to an arbitrary suspension of a cone between two objects 
 in $\Loc^\alpha(\setofobjects)$.
        \end{itemize}
 \item If $\beta$ is a limit ordinal, then 
 $\Loc^\beta(\setofobjects) = \bigcup_{\alpha< \beta} \Loc^\alpha(\setofobjects)$.
\end{enumerate}

Notice that with this definition $\Loc^\alpha(\setofobjects)$ is always an
essentially small category, and if $\alpha \leq \beta$ then
$\Loc^\alpha(\setofobjects) \subset \Loc^\beta(\setofobjects)$.  Moreover for any
ordinal $\alpha$, $\Loc^\alpha(\setofobjects) \subset \Loc(\setofobjects)$.  For any
object $M$, if $\alpha$ is the least ordinal such that $M \in
\Loc^\alpha(\setofobjects)$, then we will say that $M$ can be constructed in
$\alpha$ steps from $\setofobjects$ or has length $\alpha$.
%%%%%%%%%%%%%%%%%%%%%%%%%%%%%%%%%%%%%%%%%%%%
\begin{lemma}\label{lem ordfiltiscompl}
%%%%%%%%%%%%%%%%%%%%%%%%%%%%%%%%%%%%%%%%%%%%
For any set of objects $\setofobjects$, we have
$$\bigcup_{\alpha} \Loc(\setofobjects) = \Loc(\setofobjects)$$.
\end{lemma}
%%%%%%%%%%%%%%%%%%%%%%%%%%%%%%%%%%%%%%%%%%%%
\begin{proof}
  By construction $\bigcup_{\alpha} \Loc(\setofobjects)$ is closed 
  under suspension;
  we first prove that it is triangulated.  For this consider two objects $M,N
  \in \bigcup_{\alpha} \Loc(\setofobjects)$.  By definition there exists an
  ordinal $\alpha$ such that $M,N \in \Loc^\alpha(\setofobjects)$, for instance
  take any ordinal that is larger than the lengths of $M$ and $N$.  For any
  morphism $f : M \rightarrow N$, the cone of $f$ is in
  $\Loc^{\alpha+1}(\setofobjects)$.  It remains to show that $\bigcup_{\alpha}
  \Loc(\setofobjects)$ is closed under arbitrary sums.  For this, let
  $\{M\}_{i \in I}$ be a set of objects in $\bigcup_{\alpha}
  \Loc(\setofobjects)$, let $\{\alpha_{i}\}_{i \in I}$ be the set of lengths of
  these elements.  Then there exists an ordinal $\beta$ such that $\forall i \in
  I, \ \alpha_i \leq \beta$, and without loss of generality we may also assume
  that $\beta$ is larger than the cardinality of $I$.  Then, by definition,
  \[
  \coprod_{i \in I} M_i \in \Loc^\beta(\setofobjects).
  \]
  
  \vspace*{-2.5\bigskipamount}
  
\end{proof}

From this we can understand the meet of any two objects in our category,
starting with the meet with our (potential) finite elements:

%%%%%%%%%%%%%%%%%%%%%%%%%%%%%%%%%%%%%%%%%%%%
\begin{proposition}\label{prop interinfty-to-one}
%%%%%%%%%%%%%%%%%%%%%%%%%%%%%%%%%%%%%%%%%%%%
Let $\{f_j\}_{j \in J}$ be a family of elements in $R$, and let $g$ be an element in
$R$.  Then
$$\Loc(R_g) \cap \Loc( R_{f_j} \mid j \in J)
= \Loc( R_{gf_j } \mid j \in J).$$
\end{proposition}
%%%%%%%%%%%%%%%%%%%%%%%%%%%%%%%%%%%%%%%%%%%%
\begin{proof} Since $R_g$ is flat,  for any $j \in J$ we have $ R_g \otimes R_{f_j} \simeq R_{gf_j }$,
  and since localizing subcategories in $D(R)$ are all 
  tensor ideals by Lemma~\ref{lem locistensideal},
  from this we get the inclusion
  $$
  \Loc(R_g) \cap \Loc( R_{f_j} \mid j \in J) \supset \Loc( R_{gf_j} \mid j \in J).
  $$
  To get the reverse inclusion, consider an object $M \in \Loc(R_g) \cap \Loc(
  R_{f_j} \mid j \in J)$.  Because $M \in \Loc(R_{g})$, we know by
  Lemma~\ref{lem belongstoLocRf} that $M \simeq R_g \otimes M$.  Let $\alpha$
  be the length of $M$ in $\Loc( R_{f_j} \mid j \in J)$.  We prove by
  transfinite induction on $\alpha$ that $M \in \Loc( R_{gf_j} \mid j \in J)$.

  If $\alpha = 0$, there is nothing to prove. 

  If $\alpha = 1$, then there exists an integer $n \in \mathbb{Z}$ and an element
  $j \in J$ such that $M = \Sigma^n R_{f_j}$, but then:
  \[
  M= R_g \otimes M =
  \Sigma^n (R_{gf_j }) \in \Loc( R_{gf_j} \mid j \in J).
\]

  Let $\beta$ be an ordinal $\geq 1$ and assume by induction that the statement
  has been proved for all objects of length $\alpha < \beta $.  If 
  $\beta$ is a limit ordinal, then by construction,
  for any object $M \in  \Loc^\beta(R_{gf_j}\mid j \in J)$, there exists $\alpha < \beta$ such that $M \in
  \Loc^\alpha(R_{gf_j}\mid j \in J)$, and so by induction hypothesis $M
  \in \Loc( R_{f_j g} \mid j \in J)$.  If $\beta$ is a successor ordinal, say $\beta
  = \alpha +1$, then there are two cases.  Either there exists an integer $n \in
  \Z$, a set $Y$ of cardinality $< \norm{\beta}$, a family of objects $N_y
  \in  \Loc( R_{gf_j} \mid j \in J)^\alpha$ such that $M =\Sigma^n \coprod_{y \in Y} N_y$, and
  then $M = R_g \otimes M = \Sigma^n \coprod_{y \in Y} (R_g \otimes N_y)$.
  In this case, by induction hypothesis for all $y \in Y$ we have $R_g \otimes N_y \in
  \Loc(R_{gf_j} \mid j \in J)$, and because this class is localizing we conclude
  that $M \in \Loc( R_{gf_j} \mid j \in J)$.  Or there exists an integer $n \in
  \Z$ and two objects $N_1$ and $N_2$ in $\Loc^\alpha(R_{f_j}\mid j \in J)$
  such that we have a triangle
\[
\xymatrix{
   \Sigma^n N_1 \ar[r] &  \Sigma^n N_2 \ar[r] & \Sigma^n M \ar[r] & \Sigma^{n+1} N_1.
}  
\]
  Tensoring this triangle with $R_g$, which is flat, we get the exact triangle
 \[
\xymatrix{
 R_g \otimes \Sigma^n N_1  \ar[r] & R_g\otimes \Sigma^n N_2 \ar[r] 
 & R_g\otimes\Sigma^n M   \ar[r] &  R_g\otimes\Sigma^{n+1} N_1,
}  
\]
and by induction hypothesis $N_1$ and $R_g \otimes N_2$ belong to 
$\Loc( R_{gf_j} \mid j \in J)$,
which is triangulated, so $M \in \Loc( R_{gf_j} \mid j \in J)$ as we wanted.
\end{proof}

%%%%%%%%%%%%%%%%%%%%%%%%%%%%%%%%%%%%%%%%%%%%
\begin{corollary}\label{cor interinfty-to-finite}
%%%%%%%%%%%%%%%%%%%%%%%%%%%%%%%%%%%%%%%%%%%%
Let $\{f_j\}_{j \in J}$ be a family of elements in $R$, and let $\{g_i\}_{1 \leq i
\leq n}$ be a finite family of elements in $R$.  Then $$\Loc( R_{f_j} \mid j \in
J) \cap \Loc( R_{g_i} \mid 1 \leq i \leq n) = \Loc(R_{f_j g_i} \mid j \in J, \ 1
\leq i \leq n).$$
\end{corollary}
%%%%%%%%%%%%%%%%%%%%%%%%%%%%%%%%%%%%%%%%%%%%
\begin{proof}
The same proof works, but instead of tensoring with the flat module $R_g$ one
tensors with the flat complex $L_{R/ \idealgen{g_i, 1 \leq i \leq n}}(R)$ described in the proof of Theorem~\ref{thm orthfinitgen}.
\end{proof}

As an immediate consequence we get that if $J$ is an arbitrary ideal and $I$ a
finitely generated ideal then:
\begin{multline*}
 \Loc( R_f \mid f \in \sqrt{J}) \meet  \Loc( R_f \mid f \in \sqrt{I}) 
 \ = \ \Loc( R_f \mid f \in \sqrt{I}\cdot \sqrt{J}) \\
 = \ \Loc( R_f \mid f \in \sqrt{J}) \cap  \Loc( R_f \mid f \in \sqrt{I}).
\end{multline*}

%%%%%%%%%%%%%%%%%%%%%%%%%%%%%%%%%%%%%%%%%%%%
\begin{proposition}\label{prop compactinjoincopl}
%%%%%%%%%%%%%%%%%%%%%%%%%%%%%%%%%%%%%%%%%%%%
In the  poset of localizing categories of the form  $\Loc( R_{f} \mid f \in 
\sqrt{J})$, ordered by inclusion, the finite 
elements are exactly those of the form $\Loc( R_f \mid f \in \sqrt{J})$ with $J$ 
a finitely generated ideal.
\end{proposition}
%%%%%%%%%%%%%%%%%%%%%%%%%%%%%%%%%%%%%%%%%%%%
\begin{proof}
  First notice that  if $\Loc( R_{f} \mid f \in \sqrt{J})$ cannot be generated by finitely
  many localizations of the ring $R$, then it is not finite, for then we know that the equality 
  $\Loc( R_{f} \mid f \in \sqrt{J}) = \bigjoin_{f \in J} \Loc(R_f)$
  cannot factor  through a join over any finite subset of $J$.

  Also, if one proves that for any $g \in R$, $\Loc(R_g)$ is finite, then it is
  immediate that for any finite set $K \subset R$, $\Loc( R_g \mid g \in K)$ is
  finite, as it would be the join of finitely many finite elements.

  It remains to prove that, if $\Loc(R_g) \subset \Loc( R_f \mid f\in J)$, with $J$
  infinite, then there exists a finite subset $K \subset J$ such that $\Loc(R_g)
  \subset \Loc( R_f \mid f\in K)$.  By Proposition~\ref{prop interinfty-to-one}, we
  know that $\Loc(R_g) = \Loc(R_g) \cap \Loc( R_f \mid f \in J) =
  \Loc( R_{fg} \mid f \in J)$, and since $\Loc(R_g)$ is equivalent to $D(R_g)$ we may assume without
  loss of generality that $R= R_g$.  We now have to prove that if $\Loc( R_f \mid
  f\in J) = D(R)$, then there exists a finite set $K \in J$ such that 
  $\Loc( R_f \mid f \in K) = D(R)$.  There are two cases to consider.

  If the ideal generated by $J$ is $R$, then there exists a finite subset $K \in
  J$ such that this finite subset already generates $R$ as an ideal.  In this
  case we apply Theorem~\ref{thm orthfinitgen} to get $\Loc( R_f \mid f \in K) =
  \Loc(R/\idealgen{K})^\bot = \Loc(\{0\})^\bot = D(R)$ and we have $D(R) =
  \Loc( R_f \mid f \in K) \subset \Loc( R_f \mid f\in J) \subset D(R)$.

  On the other hand if the ideal generated by $J$, call it again $J$, 
  is a proper ideal, then $\Loc( R_f \mid f\in J)$ is a proper
  subcategory of $D(R)$, in contradiction with the assumption.
  Indeed, consider an injective envelope $E(R/J)$ of
  $R/J$, then a direct computation shows that for all $f \in J$ we have
  $\Hom_{D(R)}(R_f,E(R/J))=0$, so $\Loc( R_f \mid f \in J)^\bot \ni E(R/J) \neq
  0$.  In particular $\Loc( R_f \mid f \in J) \neq D(R)$.
\end{proof}
\begin{definition}
  Denote by $\RfGLoc(D(R))$ the poset 
  of localizing subcategories of $D(R)$ generated by sets of localizations of 
  $R$, and by $\fgRfGLoc(D(R))$ the sub-lattice of finite elements 
  (i.e.~localizing subcategories of $D(R)$ generated by finite sets of localizations of 
  $R$).
\end{definition} 
  
  The following theorem shows that $\RfGLoc(D(R))$ is in fact a coherent frame
  (and therefore $\fgRfGLoc(D(R))$ is a 
  distributive lattice).

%%%%%%%%%%%%%%%%%%%%%%%%%%%%%%%%%%%%%%%%%%%%
\begin{theorem}\label{thm loccatisZar}
%%%%%%%%%%%%%%%%%%%%%%%%%%%%%%%%%%%%%%%%%%%%
There is a natural isomorphism of posets
\[
\RfGLoc(D(R)) \simeq \RadId(R).
\]
given by:
\[
 \Loc(R_f\mid f \in J) \stackrel{f}{\longmapsto} \sqrt{\idealgen{f \mid f \in J}}
\]
\[
\Loc(R_f \mid f \in I  ) \stackrel{g}{\longmapsfrom} I
\]
Moreover, when restricted to their finite parts this isomorphism
induces an isomorphism of distributive lattices:
\[
\fgRfGLoc(D(R)) \simeq { \bf fgRadId}(R).
\]
\end{theorem}
%%%%%%%%%%%%%%%%%%%%%%%%%%%%%%%%%%%%%%%%%%%%
\begin{proof}
It is obvious that $f \circ g$ is the identity, so we just have to prove that $g
\circ f$ is the identity, namely that if $J \subset R$ then $\Loc( R_f \mid f \in J)
= \Loc( R_f \mid f \in \sqrt J )$, where $\sqrt  J $ stands for
the radical ideal generated by $J$.  It is clear that $\Loc( R_f \mid f \in J) \subset
\Loc( R_f \mid f \in \sqrt{ J })$.  Let $g \in \sqrt{ J
}$, then some power of $g$, say $g^n$, is a linear combination of finitely
many elements in $J$, say $j_1, \dots, j_n$, in particular we know that in this
case $\Loc(R\idealgen{g}) \supset \Loc(R/\idealgen{j_1, \dots, j_n})$, and taking right orthogonals
we get $\Loc(R_g) \subset \Loc(R_{j_1},  \dots ,R_{j_n}) \subset 
\Loc( R_f \mid f\in J)$.  So $\Loc( R_f \mid f \in J) \supset 
\Loc( R_f \mid f \in \sqrt{ J})$, as we wanted.
\end{proof}

As for the Hochster frame, this poset isomorphism shows that on the left-hand side we have a coherent  frame.  It is straightforward to check that the join operation on the left 
$$
\Loc(R_f\mid f \in I) \meet \Loc(R_g\mid g \in J) = 
\Loc(R_{fg}\mid (f,g) \in I \times J) 
$$
is given by taking ``localization closure''.  
We do not know whether the meet operation is always given by intersection or not.
Corollary~\ref{cor interinfty-to-finite} shows that ``meet is intersection'' if
just one of the localizing subcategories is generated by a finite number of 
objects (i.e.~is a finite element in the lattice).
We suspect that in general the meet may be strictly smaller than the 
intersection.

%%%%%%%%%%%%%%%%%%%%%%%%%%%%%%%%%%%%%%%%%%%%%%%%%%
\subsubsection{Colocalizing subcategories}
%%%%%%%%%%%%%%%%%%%%%%%%%%%%%%%%%%%%%%%%%%%%%%%%%%
In Theorem~\ref{thm loccatisZar} above, the involved localizing subcategories are also 
colocalizing.
Neeman~\cite{MR2794632}
has recently proved a theorem classifying colocalizing subcategories of
$D(R)$ in the case where $R$ is a noetherian ring:  they are in
inclusion-preserving one-to-one correspondence with arbitrary subsets of the
prime spectrum $\Spec R$.  This result does not involve any topology at all.
As a corollary to  Theorem~\ref{thm loccatisZar}
we obtain an interesting addendum to Neeman's
colocalizing classification in the noetherian case, namely a characterization of
those colocalizing subcategories that correspond to Zariski open subsets: they
are precisely the right orthogonals to the subcategories of the form $\Loc(R/I)$.

%%%%%%%%%%%%%%%%%%%%%%%%%%%%%%%%%%%%%%%%%%%%%
%%%%%%%%%%%%%%%%%%%%%%%%%%%%%%%%%%%%%%%%%%%%%
\subsubsection{Functoriality}\label{sec functorial}
%%%%%%%%%%%%%%%%%%%%%%%%%%%%%%%%%%%%%%%%%%%%%
%%%%%%%%%%%%%%%%%%%%%%%%%%%%%%%%%%%%%%%%%%%%%

To a commutative ring $R$, we can associate the Zariski frame or the Hochster 
frame.  These assignments are the object part of two covariant functors 
$\kat{Ring} \rightarrow \kat{Frm}$.  In
this subsection we describe these two functors in a way fitting our
description as frames embedded in the derived category of a ring.

Fix a ring homomorphism $\phi: S \rightarrow R$.  
Extension of scalars functor induces a triangulated functor
\[
\begin{array}{rcl}
\phi_\ast: D(S) & \longrightarrow & D(R) \\
E & \longmapsto & E \otimes_S R.
\end{array}
\]
Since extension of scalars sends the module $S$ onto $R$ and commutes with
arbitrary sums, the derived functor preserves compact objects and sends
localizing subcategories to localizing subcategories.  In particular,
\[
 \forall E \in D(S), \ \Loc(E) \otimes_S R \subset \Loc(E \otimes_S R).
\]

Hence we have a canonical map of frames:
\[
\begin{array}{ccc}
\Loc( C_i \mid i\in I) & \longmapsto & \Loc( C_i \otimes_S R \mid i \in 
I).
\end{array}
\]

For the Zariski spectrum the situation is similar.  Recall that if $I$ is a
subset of $S$, and $\zeros{I}$ is the Zariski \emph{closed} set associated to $I$,
then the preimage of $\zeros{I}$ via $\phi^\ast: \Spec_Z R \rightarrow \Spec_Z S$ is
$\zeros{\phi(I)}$.  At the level of open sets, this means that the preimage of the
Zariski open set $\bigcup_{f \in I} D(f)$ is $\bigcup_{f \in I} D(\phi(f))$.
But the extension of scalars is also compatible with localization, in fact since
for any element $f \in S$, the module $S_f$ is flat and is a ring it is
straightforward to check using the universal property of localization that
\[
 S_f \otimes_S R = R_{\phi(f)} . 
\]

In particular we have an induced map of frames
\[
\begin{array}{ccc}
\RfGLoc(D(S))	 & \longrightarrow& \RfGLoc(D(R)) \\
\Loc( S_f \mid f \in I) & \longmapsto & \Loc( R_{\phi(f)} \mid f \in I)
\end{array}
\]
 which
coincides with the map induced by the map $\Spec_Z R \rightarrow
\Spec_Z S$.

%%%%%%%%%%%%%%%%%%%%%%%%%%%%%%%%%%%%%%%%%%%%%
\subsection{Points in $\Spec R$}
%%%%%%%%%%%%%%%%%%%%%%%%%%%%%%%%%%%%%%%%%%%%%

We digress to give point-set characterizations of the Zariski and Hochster
open sets.
% To avoid the confusion between a prime ideal in $R$, a prime in a frame and the
% associated point, we denote by $\mathfrak{p}$ a prime ideal in $R$ or in a frame
% and we denote the associated topological point by $x_{\mathfrak{p}}$.  From the
% interpretation of a prime filter as those open sets that contain a given point
% we deduce the following characterization:

%%%%%%%%%%%%%%%%%%%%%%%%%%%%%%%%%%%%%%%%%%%%
\begin{proposition}\label{prop pointsinLocRf}
%%%%%%%%%%%%%%%%%%%%%%%%%%%%%%%%%%%%%%%%%%%%
Given a prime ideal $\mathfrak{p}\subset R$, and a finitely generated ideal $I$,
the following conditions are equivalent:
\begin{enumerate}
  \item[i)] The point $\mathfrak{p}$ belongs to the Zariski
  open set corresponding to $\Loc(R_f \mid f \in I)$,
  \item[ii)] $\exists f \in I $ such that $\kappa(\mathfrak{p}) \in \Loc(R_f)$.
 \end{enumerate}
 
\end{proposition}
%%%%%%%%%%%%%%%%%%%%%%%%%%%%%%%%%%%%%%%%%%%%
\begin{proof}
  Condition (i) is clearly equivalent to the condition:
  $\exists f \in I $ such that $f \notin \mathfrak{p}$.  But $f
  \notin \mathfrak{p}$ if and only if multiplication by $f$ is an isomorphism in
  the residue field $\kappa(\mathfrak{p})$; this happens if and only if
  $\kappa(\mathfrak{p})$ is canonically an $R_f$-module, and we conclude by
%   the
%   characterization of the complexes in $\Loc(R_f)$,
  Lemma~\ref{lem belongstoLocRf}.
%   , gives us the result.
\end{proof}

% For the  lattice of localizing subcategories
% generated by a single compact (or equivalently generated by the quotient of $R$
% by a finitely generated ideal), we now find  the following handier
% characterization.

%%%%%%%%%%%%%%%%%%%%%%%%%%%%%%%%%%%%%%%%%%%%
\begin{proposition}\label{prop pointsinLocComp}
%%%%%%%%%%%%%%%%%%%%%%%%%%%%%%%%%%%%%%%%%%%%
Given a finitely generated ideal $I$ in $R$ and a prime ideal $\mathfrak{p}$,
the following are equivalent:
\begin{enumerate}
\item[i)] The point $\mathfrak{p}$ belongs to the Hochster open set $\Loc(R/I)$,
 \item[ii)] As ideals in $R$, we have $I \subset \mathfrak{p}$,
\item[iii)] $R/\mathfrak{p} \in \Loc(R/I)$,
\item[iv)] $R/I \otimes R_\mathfrak{p} \neq 0$
\item[v)] $\exists E \in \Loc(R/I)$ such that $E \otimes R_\mathfrak{p} \neq 0$.
\end{enumerate}
% where $R_\mathfrak{p}$ denotes as usual the ring localized at the multiplicative
% system $R\setminus\mathfrak{p}$.
\end{proposition}
%%%%%%%%%%%%%%%%%%%%%%%%%%%%%%%%%%%%%%%%%%%%
\begin{proof}
The equivalences $i) \Leftrightarrow ii) \Leftrightarrow iii)$ are immediate 
consequences of
the characterization of the objects in $\Loc(R/I)$ given in 
Lemma~\ref{lem blongstoLR/I},
and the fact that $\mathfrak{p}$ is a prime ideal.

Let us prove that condition $ii)$ implies $ iv)$.
Since $I \subset \mathfrak{p}$, we have a surjection $R/I \twoheadrightarrow
R/\mathfrak{p}$.  Tensoring with $R_\mathfrak{p}$ we get a surjective map 
$R/I \otimes R_{\mathfrak{p}}
\twoheadrightarrow \kappa_\mathfrak{p}$ onto
the residue field at $\mathfrak{p}$, hence $R/I \otimes
R_\mathfrak{p} \neq 0$.

Conversely, let us prove by contraposition that 
$iv) \Rightarrow iii)$:  If $I \not\subset \mathfrak{p}$
then $R/I \otimes R_\mathfrak{p} = 0$.  For this
consider the exact sequence of $R$-modules:
\[
 \xymatrix{
 0 \ar[r] & I \ar[r]  & R \ar[r]  & R/I \ar[r] & 0 \\
 }
\]
Tensoring with the flat module $R_\mathfrak{p}$, gives the exact sequence
\[
 \xymatrix{
 0 \ar[r] & I\otimes R_\mathfrak{p} \ar[r]  & R_\mathfrak{p} \ar[r]  & R/I\otimes R_\mathfrak{p} \ar[r] & 0 \\
 }
\]
But as $I \not\subset R_\mathfrak{p}$, there is an element in $I$ that becomes invertible
in $R_\mathfrak{p}$, in particular the first arrow has to be an epimorphism and 
hence $R/I \otimes R_\mathfrak{p} =0$.

The implication $iv) \Rightarrow v)$ is trivial. To show the converse, observe
that, since the triangulated functor $-\otimes R_\mathfrak{p}$
commutes with arbitrary sums, if  $R/I$ belongs
to its kernel then so does the entire localizing subcategory 
generated by $R/I$, hence by contraposition $v) \Rightarrow iv)$.
\end{proof}

From this we recover Neeman's description of the correspondence
between compactly generated localizing
subcategories and Hochster open sets, but extended to the non-noetherian case,
see~\cite[Theorem 2.8]{MR1174255}:

%%%%%%%%%%%%%%%%%%%%%%%%%%%%%%%%%%%%%%%%%%%%%
\begin{corollary}\label{cor generdetectpoint}
%%%%%%%%%%%%%%%%%%%%%%%%%%%%%%%%%%%%%%%%%%%%%
Let $\setofobjects$ be a set of compact objects. Then 
a point $\mathfrak{p}$ belongs to the Hochster open set
$\Loc(\setofobjects)$ if and only if there exists
$C \in \setofobjects$ such that $C \otimes R_\mathfrak{p} \neq 0$.
\end{corollary}
%%%%%%%%%%%%%%%%%%%%%%%%%%%%%%%%%%%%%%%%%%%%%
\begin{proof}
  We know that $\Loc(\setofobjects) =\bigvee_{C \in \setofobjects} \Loc(C)$, and
  since the join operation corresponds to the union of open sets, the point
  $\mathfrak{p}$ belongs to $\Loc(\setofobjects)$ if and only if it
  belongs to $\Loc(C)$ for some $C\in \setofobjects$.  The condition is then
  clearly necessary as it fulfills condition $v)$ in Proposition~\ref{prop
  pointsinLocComp}.
 
 Conversely, if for any $C \in \setofobjects$ we have $C \otimes R_\mathfrak{p}
 =0$, then given any $E \in \Loc(\setofobjects)$ and a recipe for $E$, tensoring
 this recipe with $R_\mathfrak{p}$ we conclude that $E \otimes R_\mathfrak{p} =
 0$ and again by $v)$ in Proposition~\ref{prop pointsinLocComp} we conclude that
 the point $\mathfrak{p}$ does not belong to the open set $\Loc(\setofobjects)$.
\end{proof}

More geometrically we have:
%%%%%%%%%%%%%%%%%%%%%%%%%%%%%%%%%%%%%%%%%%
\begin{corollary}\label{cor hsuppisHochOpen}
%%%%%%%%%%%%%%%%%%%%%%%%%%%%%%%%%%%%%%%%%%
For any perfect complex $C$ in $D(R)$ its homological support
\[
\supp C= \{ \mathfrak{p} \in \Spec R \mid C \otimes R_\mathfrak{p} \neq 0 \}
\]
is a Hochster open set.
\end{corollary}
%%%%%%%%%%%%%%%%%%%%%%%%%%%%%%%%%%%%%%%%%%

%%%%%%%%%%%%%%%%%%%%%%%%%%%%%%%%%%%%%%%%%%
%%%%%%%%%%%%%%%%%%%%%%%%%%%%%%%%%%%%%%%%%%
\section{Tensor triangulated categories}\label{sec ZarSpecTriaCat}
%%%%%%%%%%%%%%%%%%%%%%%%%%%%%%%%%%%%%%%%%%
%%%%%%%%%%%%%%%%%%%%%%%%%%%%%%%%%%%%%%%%%%

In this section we revisit Balmer's theory of spectra and supports of tensor
triangulated categories. 
The point-free approach reveals that  this construction
and its basic properties are so similar to the ring case, that they can be seen as a variation of 
Joyal's constructive account of the Zariski spectrum in terms of 
supports, dating back to the early 70s \cite{Joyal:Cahiers1975}.

%%%%%%%%%%%%%%%%%%%%%%%%%%%%%%%%%%%%%%%%%%%%%%%%%%
\subsection{The Zariski spectrum of a tensor triangulated category}
%%%%%%%%%%%%%%%%%%%%%%%%%%%%%%%%%%%%%%%%%%%%%%%%%%

\newcommand{\genstep}{G}
\begin{definition}\label{def gen}  
  Let $\setofobjects$ be a set of objects in a tensor triangulated category 
  $(\T,\otimes,\mathbf{1})$.  Define $\genstep(\setofobjects)$ to be the set consisting of
those objects of the form:
\begin{enumerate}
  \item[i)] an iterated suspension or desuspension of an object in $\setofobjects$,
\item[ii)] or a finite sum of objects in $\setofobjects$,
\item[iii)] or an object $s\otimes t$ with  $s\in \setofobjects$ and $t\in \T$,
\item[iv)] or an extension of two objects in $\setofobjects$,
\item[v)] or a direct summand of an object in $\setofobjects$,
\end{enumerate}

Clearly, if a thick tensor ideal contains $\setofobjects$ then it also contains $\genstep(\setofobjects)$, and
hence by induction it contains $\genstep^\omega(\setofobjects) := \bigcup_{n\in\N}\genstep^n(\setofobjects)$.  On the
other hand, it is easy to see that $\genstep^\omega(\setofobjects)$ is itself a thick tensor ideal,
hence it is the smallest thick tensor ideal containing $\setofobjects$.  We denote it
$\gen{\setofobjects}$.
\end{definition}

The following result expresses the finiteness in the definition of 
thick tensor ideal. 
  
\begin{lemma}\label{lem fingen}
  Let $\setofobjects$ be a set of objects and suppose $x\in \gen{\setofobjects}$.  Then there exists
  a finite subset $K \subset \setofobjects$ such that also $x\in \gen{K}$.
\end{lemma}
\begin{proof}
    We have $x\in \genstep^n(\setofobjects)$ for some $n\in\N$.  This means $x$ is 
    obtained by one of the construction steps in $\genstep$ 
    from finitely many objects in $\genstep^{n-1}(\setofobjects)$.  By downward 
    induction, $x$ is then obtained from a finite set of objects  
    $K \subset \genstep^0(\setofobjects)=\setofobjects$, hence $x\in \gen K$.
\end{proof}

%%%%%%%%%%%%%%%%%%%%%%%%%%%%%%%%%%%%%%%%%%%%%%%%%%
\subsubsection{Radical thick tensor ideals}
%%%%%%%%%%%%%%%%%%%%%%%%%%%%%%%%%%%%%%%%%%%%%%%%%%

Fix a tensor triangulated category $(\T,\otimes, \mathbf{1})$.
To any thick tensor ideal $I$ (cf.~Definition~\ref{def tens-ideal})
we may associate its radical closure $\sqrt{I}$ just as in the ring case:
\[
 \sqrt{I} = \{ a \in \T \mid \exists n \in \N \textrm{ such that } a^{\otimes n} \in I\}.
\]
A thick tensor ideal $I$ is called {\em radical} when $I=\sqrt I$.

More generally, for any set of objects $\setofobjects$ we denote by $\sqrt \setofobjects$ the radical of
the thick tensor ideal $\gen \setofobjects$.

\begin{corollary}\label{finitethickrad}
    Let $\setofobjects$ be a set of objects and suppose $x\in \sqrt \setofobjects$.  Then there exists
  a finite subset $K \subset \setofobjects$ such that also $x\in \sqrt K$.
\end{corollary}
\begin{proof}
  Apply Lemma~\ref{lem fingen} to a suitable power of $x$.
\end{proof}

%%%%%%%%%%%%%%%%%%%%%%%%%%%%%%%%%%%%%%%%%%
\begin{lemma}\label{lem radofidealisideal}
%%%%%%%%%%%%%%%%%%%%%%%%%%%%%%%%%%%%%%%%%%
If $I$ is a thick tensor ideal, then $\sqrt{I}$  is a radical thick tensor ideal.
\end{lemma}
%%%%%%%%%%%%%%%%%%%%%%%%%%%%%%%%%%%%%%%%%%
Balmer proved this \cite[Lemma~4.2]{Balmer:2005} 
by establishing the classical formula 
$\sqrt{I} = \bigcap_{\mathfrak{p} \supset I} \mathfrak{p}$, valid assuming Zorn's lemma. We offer instead a direct point-free proof:
\begin{proof}
It is immediate to check from the definitions that $\sqrt{I}$ is closed under
suspension and desuspension, finite sums, direct summands,  and under tensoring 
with objects of $\T$.
Finally for a triangle $x \to y \to z \to \Sigma x$, the following general lemma
shows that if $x$ and $y$ belong to $\sqrt I$ then so does $z$.
\end{proof}

\begin{lemma}
  Let $I$ be a tensor ideal in a tensor triangulated category, and consider
  a triangle
  $x \rightarrow y \rightarrow z \rightarrow \Sigma x$.
  If $x^p$ and $y^q$ belong to $I$, then $z^{p+q-1}$ belongs to $I$.
\end{lemma}

\begin{proof}
  More generally we show by induction on $k$ that
  $$
  x^i y^j z^k \in I \quad \forall i,j,k \text{ such that } i+j+k = p+q-1 
  $$
  (where for economy we omit the tensor sign between the factors).
  The case $k=0$ is clear since $I$ is a tensor ideal.
  For the monomial $x^i y^j z^{k+1}$ (with $i+j+k+1=p+q-1$), tensor the triangle
  $x \to y \to z \to \Sigma x$ with $x^i y^j z^k$.  By induction the first two
  vertices in the resulting triangle belong to $I$, and hence so does the third.
\end{proof}

Radical thick tensor ideals in $\T$ are naturally ordered by inclusion with a
top element $\T$ itself, and a bottom element
$$
\sqrt{0} = \{ a \in \T \mid \exists n \in \N \ 
\textrm{ such  that } a^{\otimes n} = 0\},
$$
the full subcategory of \emph{nilpotent} elements.
Although radical thick tensor ideals might not form a set,
we have well-defined frame operations:
\begin{enumerate}

    \item If $I_1$ and $I_2$ are two radical thick
tensor ideals then $I_1 \bigwedge I_2 =
I_1 \bigcap I_2$.  

   \item If $\{I_j\}_{j
\in J}$ is a set of radical thick tensor ideals, $\bigvee_{j \in J}
I_j$ is the radical of the thick tensor ideal generated by
the union $\bigcup_{j\in J} I_j$.  This is well defined by
Definition~\ref{def gen} and Lemma~\ref{lem radofidealisideal}.
\end{enumerate}

The main theorem in this subsection (Theorem~\ref{thm coherence} below)
states that the radical thick tensor ideals
of  a tensor triangulated category $\T$ form a coherent frame.  For this to make 
sense
it is necessary that there is only a set of them.
The easiest way to ensure this is to 
assume that $\T$ is essentially small, as in Balmer~\cite{Balmer:2005}.
A source of examples of this situation comes from
starting with a compactly generated triangulated category $\T$,
for then (as explained for instance in \cite[Chapter 3 and Remark 4.2.6]{MR1812507}),
the full subcategory of compact objects $\T^\omega$ is essentially
small.  If we add the assumption that the tensor unit is
compact and that the tensor product of two compact objects is again compact then
the full subcategory of compact objects $\T^\omega$ is an essentially small
tensor triangulated category.
Our main example is when $\T$ is 
the derived category of a commutative ring, or the derived
category of a coherent scheme as in Section~\ref{sec RecSpecSch}.
It follows that $\T^\omega$, the 
derived category of perfect complexes, is essentially small.

As pointed out by Balmer~\cite{Balmer:2005}, in many 
important situations, passage to the radical is a
harmless operation. For instance, in $D^\omega(R)$, all thick tensor ideals are
radical, which follows from the fact that every perfect complex is strongly 
dualizable, as we proceed to briefly recall.
Observe that in $D(R)$ all thick subcategories are automatically tensor 
ideals (by an argument similar to the proof of Lemma~\ref{lem 
locistensideal}), hence all thick subcategories are radical thick 
tensor ideals.

For an object $a\in \T$, put $a^\vee=\Hom(a,\mathbf{1})$.
An object $a$ is {\em strongly dualizable} if and only if the natural 
transformation
\[
- \otimes a^\vee \rightarrow \Hom(a,-)
\]
is an isomorphism.  It is well known \cite{MR866482} that any strongly dualizable object $a$ is
a direct summand of $a \otimes a \otimes a^\vee$, hence:

%%%%%%%%%%%%%%%%%%%%%%%%%%%%%%%%%%%%%%%%%%%%
\begin{lemma}\label{lem tensidIsrad}
%%%%%%%%%%%%%%%%%%%%%%%%%%%%%%%%%%%%%%%%%%%%
  If all compact objects in $\T$ are strongly dualizable, then all thick tensor
  ideals in $\T^\omega$ are radical.
\end{lemma}
%%%%%%%%%%%%%%%%%%%%%%%%%%%%%%%%%%%%%%%%%%%%

%%%%%%%%%%%%%%%%%%%%%%%%%%%%%%%%%%%%%%%%%%
\begin{lemma}\label{prop ZarFrameTriaCat}
%%%%%%%%%%%%%%%%%%%%%%%%%%%%%%%%%%%%%%%%%%
  In the poset of radical thick tensor ideals, the infinite distributive law 
  holds:
for any radical thick tensor ideal $J$ and any set of radical thick
  tensor ideals $(I_\alpha)_{\alpha \in A}$, we have
  $$
  \bigvee_\alpha (J\wedge I_\alpha ) = J \cap (\bigvee_\alpha I_\alpha).
  $$  
\end{lemma}
%%%%%%%%%%%%%%%%%%%%%%%%%%%%%%%%%%%%%%%%%%
\begin{proof}
  The inclusion $\subset$ is clear.  To get the reverse inclusion fix an object $x\in  J \cap (\bigvee_\alpha I_\alpha)$. By radicality it is enough to prove
  that $x\otimes x \in  \bigvee_\alpha (J\wedge I_\alpha )$.  Define 
  $$
  C_x = \{ k \in \bigvee_\alpha I_\alpha \mid x \otimes k \in \bigvee_\alpha 
  (J\wedge I_\alpha ) \};
  $$
  we are done if we can prove that $C_x$ is all of $\bigvee_\alpha I_\alpha$.
 It is trivial to check that $C_x$ is a triangulated category, because tensoring
  with $x$  preserves triangles.
 
  First we prove that $C_x$ is a thick subcategory. Suppose $a\oplus b 
  \in C_x$, this means that $x \otimes (a\oplus b) \in \bigvee_\alpha (J\wedge 
  I_\alpha )$.  But  the tensor product distributes over
  sums, so also $(x\otimes a) \oplus (x \otimes b) \in \bigvee_\alpha (J\wedge 
  I_\alpha )$. As the latter is a thick subcategory,  we conclude that already each of
  $(x\otimes a)$ and  $(x \otimes b)$ belong here, which is to say that 
  $a$ and $b$ are in  $C_x$.
  
  We show now $C_x$ is an ideal: let $a$ be an arbitrary object of the triangulated
  category, and let $k\in C_x \subset \bigvee_\alpha I_\alpha$.  Since
  $\bigvee_\alpha I_\alpha $ is an ideal, we also have $a\otimes k \in 
  \bigvee_\alpha I_\alpha$. For the same reason  $x \otimes (a \otimes k)  = a \otimes (x \otimes k)$ belongs to $\bigvee_\alpha (J\wedge 
  I_\alpha )$.  By definition of $C_x$ we therefore
  find that $a \otimes k\in C_x$ as required.
  
  Finally radicality of $C_x$:
  suppose $k^{\otimes n} \in C_x$.  This means that $x\otimes k^{\otimes n}
  \in \bigvee_\alpha (J\wedge I_\alpha )$.  But then we can tensor $n-1$ times
  more with $x$ to conclude that $(x\otimes k)^{\otimes n}  \in \bigvee_\alpha 
  (J\wedge I_\alpha )$,
  and since this is a radical ideal, it then follows that 
  $x\otimes k  \in \bigvee_\alpha (J\wedge I_\alpha )$, which is to say
  that $k\in C_x$.
  
  In conclusion, $C_x$ is a radical thick tensor ideal contained in
  $\bigvee_\alpha I_\alpha$, and it contains each $I_\alpha$, so it also
  contains their join, hence is equal to the whole join.
\end{proof}

%%%%%%%%%%%%%%%%%%%%%%%%%%%%%%%%%%%%%%%%%%%%
\begin{theorem}\label{thm coherence}
%%%%%%%%%%%%%%%%%%%%%%%%%%%%%%%%%%%%%%%%%%%%%
  The radical thick tensor ideals of a tensor triangulated category
  $\T$ form a coherent frame, provided there is only a set of them.
  The finite elements are the principal radical thick tensor ideals, i.e.~of
  the form $\sqrt{a}$ for some $a \in \T$.
\end{theorem}
%%%%%%%%%%%%%%%%%%%%%%%%%%%%%%%%%%%%%%%%%%%%%

\begin{proof}
  The proof follows the same lines as the proof that the radical
  ideals in a commutative ring form a coherent frame, but instead of
  relying of finiteness of sums in a ring, it uses finiteness of
  generation of thick tensor ideals.  Some of the arguments have a 
  different flavor because of the thickness condition which has no
  reasonable analogue for commutative rings.

  Lemma~\ref{prop ZarFrameTriaCat} establishes that the radical
  thick tensor ideals form a frame.  
  We now establish that this frame is coherent.
  We first show that finite elements are generated by a single object.  Let $K$
  be a finite element in the frame.  Since there is only a set of principal
  radical thick tensor ideals by assumption, there is certainly a set $M(K)$ of
  those that are contained in $K$.  Then trivially $K = \bigvee_{\sqrt{c} \in
  M(K)} \sqrt{c}$, and, as $K$ is a finite element in the frame, there exists a
  finite subset $J \subset M(K)$ such that $K = \bigvee_{c \in J} \sqrt{c}$, so
  $K$ is generated by a finite set consisting of one generator for each $c \in
  J$.  It is now a direct consequence of the thickness assumption that if $K$ is
  generated by $c_1, \dots, c_k$ then it is generated by the single object $c_1
  \oplus \dots \oplus c_k$.

  Finally we show each ideal of the form $\sqrt{a}$ is indeed a finite
  element in the frame.  Given a set of radical thick tensor ideals
  $\{J_\alpha\}_{\alpha\in A}$ such that $\sqrt{a} \leq \bigvee_{\alpha\in A} 
  J_\alpha$,
  we need to find a finite 
  subset $B \subset A$ such that also $\sqrt{a} \leq  \bigvee_{\alpha\in B}
  J_\alpha$.  Since the join in question is a radical thick tensor ideal,
  it is enough to find $B\subset A$ such that $a\in \bigvee_{\alpha\in B}
  J_\alpha$. 
  Let $S$ denote the union of the ideals $J_\alpha$.  Then 
  $\bigvee_{\alpha\in A} J_\alpha = \sqrt S$.
  We have $a \in \sqrt S$.  But then by Corollary~\ref{finitethickrad},
  there is a finite subset $K \subset S$, such that $a\in \sqrt K$.
  Finitely many $J_\alpha$ are needed to contain this finite subset $K$,
  so take those.
\end{proof}

\begin{definition}
    The frame of radical thick tensor ideals in $\T$ is denoted $\Zar(\T)$ and
    called the {\em Zariski frame}.  The spectral space associated to $\Zar(\T)$ we
    call the {\em Zariski spectrum} of $\T$,  denoted $\Spec \T$.
\end{definition}

% A second main theorem of~\cite{Balmer:2005}, the universal property 
% of $\Zar(\T)$ in terms of supports now follows the classical pattern:

%%%%%%%%%%%%%%%%%%%%%%%%%%%%%%%%%%%%%%%%%%%%%%%%%%
\subsection{Supports of a tensor triangulated category}
%%%%%%%%%%%%%%%%%%%%%%%%%%%%%%%%%%%%%%%%%%%%%%%%%%

Let $(\T,\otimes,\mathbf{1})$ be a tensor triangulated category, form
now on assumed to have only a set of radical thick tensor ideals.  Just as in
the ring case, the coherent frame $\Zar(\T)$ comes equipped with a canonical
notion of support.  The universal property of the Zariski frame of a ring
(Theorem~\ref{thm JoyalSupIsinit}) readily carries over to the Zariski frame of
$\T$, and yields one of the main theorems of \cite{Balmer:2005}, as we proceed
to explain.  In order to stress the parallel with the classical case, we shall
use a slight modification of Balmer's notions:

%%%%%%%%%%%%%%%%%%%%%%%%%%%%%%%%%%%%%%%%%%%%%
\begin{definition}\label{def TriaSupport}
%%%%%%%%%%%%%%%%%%%%%%%%%%%%%%%%%%%%%%%%%%%%%
% Let $(\T,\otimes,\mathbf{1})$ be a tensor triangulated category.  
A \emph{support} on $(\T,\otimes,\mathbf{1})$ is a pair $(F,d)$ where
$F$ is a frame and $d: \operatorname{obj}(\T) \to F$ is a map satisfying
\begin{enumerate}
\item $d(0)  =  0$ and $d(\mathbf{1})  =  \mathbf{1},$
\item $\forall a \in \T : \ d(\Sigma a)  =  d(a),$
\item $\forall a,b \in \T :\  d(a\oplus b)  =  d(a) \vee d(b), $
\item $\forall a,b \in \T : \ d(a\otimes b) =  d(a) \wedge d(b),$
\item If  $a\to b\to c \to \Sigma a$ is a triangle in $\T$, then  $d(b) \leq d(a) \vee d(c).$
\end{enumerate}
A {\em morphism of supports} from $(F,d)$ to $(F',d')$ is a frame map $F \to F'$
compatible with the maps $d$ and $d'$.
\end{definition}
%%%%%%%%%%%%%%%%%%%%%%%%%%%%%%%%%%%%%%%%%%%%%

%%%%%%%%%%%%%%%%%%%%%%%%%%%%%%%%%%%%%%%%%%%%
\begin{lemma}\label{lem ZarSup}
%%%%%%%%%%%%%%%%%%%%%%%%%%%%%%%%%%%%%%%%%%%%
Let $(\T,\otimes,\mathbf{1})$ be a tensor triangulated category assumed to have
only a set of radical thick tensor ideals.  Then the assignment
  \begin{eqnarray*}
    \operatorname{obj}(\T) & \longrightarrow & \Zar(\T)  \\
    a & \longmapsto & \supz(a) :=\sqrt{a}
  \end{eqnarray*}
  is a support.
\end{lemma}
%%%%%%%%%%%%%%%%%%%%%%%%%%%%%%%%%%%%%%%%%%%%

\begin{proof}
  Items $(1)$, $(2)$ and $(5)$ in Definition~\ref{def TriaSupport} are trivially
  satisfied.  Let us check item $(3)$: given $a,b \in \T$, we have $\sqrt{a
  \oplus b} = \sqrt{a} \vee \sqrt{b}$.  Since we are dealing with thick ideals,
  $\sqrt{a} \subset \sqrt{a \oplus b}$ and $\sqrt{b} \subset \sqrt{a\oplus b}$,
  so $\sqrt{a} \vee \sqrt{b} \subset \sqrt{a \oplus b}$.  Conversely, $a \oplus
  b$ certainly is in $\sqrt{a} \vee \sqrt{b}$.
  
Finally let us check $(4)$.  Given $a,b \in \T$, we wish to show that $\sqrt{a
\otimes b} = \sqrt{a} \wedge \sqrt{b}$.  Certainly $a \otimes b$ belongs to both
$\sqrt{a}$ and $\sqrt{b}$, so $\sqrt{a \otimes b} \subset \sqrt{a} \wedge
\sqrt{b}$.  For the converse we will adapt the proof of Lemma~\ref{lem tensideal}.
Let $R(a) = \{ x \in \T \mid a \otimes x \in \sqrt{a \otimes b} \}$.  Then
$R(a)$ is a radical thick tensor ideal that trivially contains $b$, hence
$\sqrt{b} \subset R(a)$.  Now fix $c \in \sqrt{b}$ and consider $L(c) = \{ x
\in \T \mid x \otimes c \in \sqrt{a \otimes b} \}$.  Then $L(c)$ is a radical
thick tensor ideal that contains $a$ by the previous step.  Now, let $y \in
\sqrt{a} \cap \sqrt{b}$.  From the ideal $L(y)$ we know that $y \otimes y \in
\sqrt{a \otimes b}$, so $\sqrt{a} \cap \sqrt{b} \subset \sqrt{a\otimes b}$ as we
wanted.
\end{proof}

%%%%%%%%%%%%%%%%%%%%%%%%%%%%%%%%%%%%%%%%%%%%
\begin{theorem}\label{thm ZarisInitialTria}
%%%%%%%%%%%%%%%%%%%%%%%%%%%%%%%%%%%%%%%%%%%%
  Let $(\T,\otimes,\mathbf{1})$ be a tensor triangulated category, assumed
  to have only a set of radical thick tensor ideals. Then the  support 
  \begin{eqnarray*}
    \T & \longrightarrow & \Zar(\T)  \\
    a & \longmapsto & \sqrt{a}
  \end{eqnarray*}
  is initial among supports.
\end{theorem}
%%%%%%%%%%%%%%%%%%%%%%%%%%%%%%%%%%%%%%%%%%%%
\begin{proof}
  For an arbitrary support $d: \T \to F$, we need to exhibit a frame map
  $u:\Zar(\T) \to F$, compatible with the maps from $\T$, and check that this
  map is unique.  Since $\Zar(\T)$ is coherent, every element is a join of
  finite elements, so $u$ is completely determined by its value on finite
  elements.  The finite elements are those of form $\sqrt{a}$ and there is no
  choice: we must send $\sqrt a $ to $d(a)$.  So there is at most one support 
  map
  $u$.
  We only need to check it is well-defined, this means to check that
\[
\forall a,b \in \T, \ \sqrt a = \sqrt b \Rightarrow d(a) = d(b).
\]
For $a \in \T$, define $I(a) = \{c \in \T \mid d(c) \leq d(a)\}$.
Then the properties of a support show that $I(a)$ is a radical thick tensor ideal
containing $a$ and hence $\sqrt{a}$. If $\sqrt{b} \subset \sqrt{a}$ we deduce that $d(b) \leq d(a)$ and by symmetry we get our result.
\end{proof}

The fact that this support is initial implies  functoriality: 
any triangulated functor
$F : \T \rightarrow \mathcal{S}$ 
induces a coherent frame map $\Zar(\T) \rightarrow \Zar(\mathcal{S})$,
taking $\sqrt{I}$ to $\sqrt{F(I)}$.

%%%%%%%%%%%%%%%%%%%%%%%%%%%%%%%%%%%%%%%%%%%%%%%%%%
\subsection{Tensor nilpotence}
%%%%%%%%%%%%%%%%%%%%%%%%%%%%%%%%%%%%%%%%%%%%%%%%%%

The tensor nilpotence theorem by Devinatz, Hopkins and Smith~\cite{MR960945},
one of the deep theorems in stable homotopy, has a version for derived
categories, which in the work Neeman~\cite{MR1174255} and
Thomason~\cite{Thomason:1997} is a basic tool to analyze localizing
subcategories.  The theorem says that if a morphism has empty
support, then it is tensor nilpotent.  As observed by Balmer~\cite{Balmer:2005},
this is in fact a consequence of general topological
properties of the spectrum of a tensor triangulated category, and as we shall
see, it comes out very elegantly in the point-free setting.

Let $(\T,\otimes,\mathbf{1})$ be a tensor triangulated category, assumed to 
have only a set of radical thick tensor ideals.
Given a morphism $f:x\to y$ in $\mathcal 
T$, we write $z | f$ to express that there exists a factorization
$$
\xymatrix{
& z \ar[rd] \\
x \ar[ru] \ar[rr]_f && y\,.
}
$$
We define the {\em support} of $f$ as
$$
\supz(f) := \bigwedge_{z| f} \supz(z) = \bigcap_{z| f} \sqrt z  \ \in 
\Zar(\mathcal T).
$$
(The notation $z | f$  is inspired by rings: $\sqrt{n} = \bigcap_{p|n} 
\sqrt{p}$.)  
Note that strictly speaking, to ensure that the meet is indexed over a set,
we should write it as the meet of all radical thick tensor ideals that occur
as support of an element $z|f$.

The notion of support for morphisms extends the usual notion of support of objects:
\begin{lemma}
  We have $\supz(\id_x) = \supz (x)$.
\end{lemma}
\begin{proof}
  The trivial factorization of $\id_x$ shows that $\supz(\id_x) \leq
  \supz(x)$.  Any other factorization $x \to y \to x$, exhibits $x$ as
  a retract of $y$, and hence $\supz(x) \leq \supz(y)$ by thickness.
\end{proof}

% Other immediate properties of the support include
% $\supz(g\circ f) \leq \supz(f) \meet \supz(g)$, and $\supz(f\otimes g) \leq \supz(f) 
% \meet \supz(g)$.  $\supz(f\oplus g) = \supz(f) \meet \supz(g)$.

\begin{theorem}(Tensor nilpotence.)\label{tensornilpotencemorphisms}
  If $\supz(f)$ is the bottom element of $\Zar(\mathcal{T})$, then $f$ is tensor nilpotent,
  i.e.~there is $n\in \N$ such that $f^{\otimes n}$ is the zero map.
\end{theorem}
\begin{proof}
  The premise says that $\supz(0) =\bigwedge_{z| f} \supz(z)$, and by
  coherence a finite meet of such elements will do: there exist 
  $z_1|f,\,\ldots\,,z_k|f$ such that $\supz(0) =\bigwedge_{i=1}^k \supz(z_i)$.
  In other words, 
  $$
  \sqrt 0 = \bigcap_{i=1}^k \sqrt{z_i} = \sqrt{z_1\otimes\cdots \otimes z_k} .
  $$
  So there is $n\in \N$ such that $\big( z_1\otimes\cdots \otimes 
  z_k\big)^{\otimes n} = 0$.  Now
  $f^{\otimes kn}$ factors through $\big( z_1\otimes\cdots \otimes 
  z_k\big)^{\otimes n}$, hence is the zero map.  So $f$ is tensor nilpotent.
\end{proof}

%%%%%%%%%%%%%%%%%%%%%%%%%%%%%%%%%%%%%%%%%%%%%
\subsection{Points}
%%%%%%%%%%%%%%%%%%%%%%%%%%%%%%%%%%%%%%%%%%%%%

In order to relate the results of this section to Balmer's work and the earlier 
literature, we proceed to
extract the points.

% which we call the {\em generating element}.
% Interpreting the frame as a frame of open sets of a topological 
% space, in our case $\Spec \T$,  we also call it {\em the defining 
% open set} of the point,

According to Balmer~\cite{Balmer:2005}, the spectrum of an essentially small
tensor triangulated category $\T$ is the topological space $X = \Spec \mathcal T$
whose set of points is the set of \emph{prime} thick tensor ideals
in $\T$.  Just as in the ring case, a thick tensor 
ideal $\mathfrak{p}$ is {\em prime} when
\[
\forall a,b \in \T : [\ a \otimes b \in \mathfrak{p} \Rightarrow a \in 
\mathfrak{p} \textrm{ or } b \in \mathfrak{p}].
\]
The topology on $X$ is defined by taking as a basis of open sets
the subsets of the form
$$
U(a) := \{ \mathfrak p\in \Spec \T \mid a \in \mathfrak p\}
$$
for each object $a\in\T$.  We shall see that this is the Hochster dual of the 
Zariski spectrum.

Recall from \ref{subsubsec points} that for a frame point $x:F\to \{0,1\}$,
the corresponding frame prime ideal is ${\mathcal P}_x := \{ u \in F \mid x(u)=0
\}$, and that the prime ideals in turn are in natural bijection
with the prime elements of $F$: the prime element corresponding to $x$ is
$e_x := \bigvee_{b \in \mathcal{P}_x} b$, and then $\mathcal{P}_x = (e_x)$.
A point $x$ belongs to the open set corresponding to a frame element $u\in F$
iff $u \notin \mathcal{P}_x$, iff $u\not\leq e_x$.
% 
% Hence clearly $x$ belongs to the open set corresponding to $u$ iff
% $u \notin {\mathcal P}_x$.
% 
% is given by a frame prime ideal $\mathcal{P}$, or equivalently by its generating element 
% \[
%  u_x= u_{\mathcal P} = \bigvee_{b \in \mathcal{P}} b .
% \]
% 
% 
% 
%   Regarding the open sets, in complete generality, for a frame $F$, 
% %   let X denote the space of
% % points, with open set U corresponding to frame element u.
% recall that for a frame point $x: F \to \{0,1\}$ 
Now specialize to the case $F=\Zar(\T)$.  It is easy to see that the prime 
elements in $\Zar(\T)$ are precisely the prime thick tensor ideals, in analogy 
with \ref{lem pointHochster}.
A point $x$ corresponding to a prime element $e_x = \mathfrak p$
belongs to the 
open set corresponding to $u = \sqrt a \in 
\Zar(\T)$ iff $\sqrt a \not\leq \mathfrak p$.  Altogether we have:
% 
% 
% i.e.~a prime thick tensor ideal $\mathfrak p$,
% the corresponding frame-theoretic
% prime ideal is ${\mathcal P} = \{ I \in \Zar(\T) \mid I \subset \mathfrak p \}$.
% Assume $a \notin \mathfrak p$.  Then for each such $I$ we have $a \notin I$.
% Therefore, for each such $I$ we have $\sqrt a \not\subset I$, and
% in particular $\sqrt a$ is not in the ideal $\mathcal P$.
% On the other hand, assume $a \in \mathfrak p$.  Then $\sqrt a \subset \mathfrak 
% p$, and hence belongs to $\mathcal P$.
% 

% 
% Observe that a prime ideal is always radical, and hence an
% element in our frame.  So a priori there are two notions of points in $\mathcal{T}$:
%   the frame theoretic notion and Balmer's definition;  fortunately
% these turn to be the same:
%%%%%%%%%%%%%%%%%%%%%%%%%%%%%%%%%%%%%%%%%%%%%
\begin{proposition}\label{prop BalptisFrpt}
%%%%%%%%%%%%%%%%%%%%%%%%%%%%%%%%%%%%%%%%%%%%%
The frame-theoretic points in $\Zar(\T)$ correspond bijectively to 
prime thick tensor ideals in $\T$.  Under this correspondence,
a finite element $\sqrt a \in \Zar(\T)$ corresponds to the set of prime thick
tensor ideals
$$
\{\mathfrak p\in \Spec \T \mid a \notin \mathfrak p\}.
$$
\end{proposition}

Balmer's $U(a)$ is the complement of this, so we get in particular:
\begin{corollary}
  Balmer's spectrum is the Hochster dual of the Zariski spectrum.
\end{corollary}

We wish to point out that modulo the passage between frames and point-set spaces,
and the identification of points just established, our 
Theorem~\ref{thm coherence} subsumes several results from Balmer's seminal
paper~\cite{Balmer:2005}, and in particular his Classification of Radical Thick
Tensor Ideals.  
Balmer proves that the 
topological space $X=\Spec\mathcal T$ of prime ideals in $\mathcal T$, with
basic open sets $U(a)$ as above,
is a spectral space, and proceeds to set up an
order-preserving bijection (his Classification Theorem 4.10) between radical
thick tensor ideals in $\T$ and subsets of $X$ of the form ``arbitrary unions of
closed sets with quasi-compact complement''.  These are clearly precisely the
Hochster dual open sets of $X$.  So after eliminating the implicit double
Hochster duality, his Classification Theorem says that there is an
order-preserving bijection between radical thick tensor ideals in $\T$ and
Zariski open sets in the Zariski spectrum (in our terminology).  From the
viewpoint of Theorem~\ref{thm coherence}, this is a tautology.

\bigskip

Finally we express tensor nilpotence in terms of points, recovering the
now classical result of Balmer.
% Note first that OUR PREVIOUS POINT CORREPONDENCE AMOUNTS TO
% \begin{lemma}
%   $$
%   \mathfrak p \in \supz(a) \ \Leftrightarrow \ a \notin \mathfrak p .
%   $$
% \end{lemma}
% In particular, if $\supz(a)$ has no points, then $a$ belongs to every
% prime ideal in $\mathcal T$.  But these are precisely the prime elements in
% $\Zar(\mathcal T)$.  But for every coherent frame (or just a frame with enough 
% points) we have that the meet of all prime elements is the bottom element.
% 
% 
We first characterize the points of $\supz(f)$:

\begin{lemma}
  For $f$ a morphism in $\mathcal T$, and $\mathfrak p\in \Spec\mathcal T$, we 
  have
  $$
  \mathfrak p \in \supz(f) \ \Longleftrightarrow \ f \neq 0 \ \operatorname{mod} \mathfrak p .
  $$
\end{lemma}
\begin{proof}
  The first two steps of the biimplication,
  $$
  \mathfrak p \in \supz(f) 
  \ \Leftrightarrow \
  \forall z | f : \mathfrak p \in \supz(z)
  \ \Leftrightarrow \
  \forall z | f : z \notin \mathfrak p
  $$
  follow by definition of support for morphisms and by a support reformulation of 
  Proposition~\ref{prop BalptisFrpt}.  The final step is easier to do negated:
%   
%   version.  First we have
%   $$
%   \mathfrak p \notin \supz(f)
%    \Longleftrightarrow 
%   \exists z | f : \mathfrak p \notin \supz(z).
%   $$
%   Indeed, by definition, $\supz(f)$ is the intersection of all 
%   $\supz(z)$ with $z | f$.
%   But we also have 
%   $$
%     \exists z | f : \mathfrak p \notin \supz(z)
%   \Longleftrightarrow
%     \exists z | f : z\in \mathfrak p
% $$
%   -- this follows from the fact 
%   $\mathfrak p \in \supz(a) \ \Leftrightarrow \ a \notin \mathfrak p$,
% which is a support reformulation of XXXX.
%   
  $$
    \exists z | f : z\in \mathfrak p
  \ \Leftrightarrow \
  f = 0 \ \operatorname{mod}  \mathfrak p ,
  $$
  which is straightforward (see~\cite[Lemma 2.22]{Balmer:2005}).
\end{proof}

\begin{corollary}(Balmer)
  If $f = 0 \ \operatorname{mod}  \mathfrak p$ for all prime thick tensor ideals $\mathfrak p$,
  then $f$ is tensor nilpotent.
\end{corollary}
\begin{proof}
  The premise says that there are no points in $\supz(f)$.  But since
  $\Zar(\mathcal T)$ is coherent, the only frame element
  without points is the bottom element --- this is a special case of the fact
  that a coherent frame has enough points, i.e.~is 
  spatial~\cite[Theorem II.3.4]{Johnstone:Stone-spaces}.  We conclude by
  Theorem~\ref{tensornilpotencemorphisms}.
\end{proof}

%%%%%%%%%%%%%%%%%%%%%%%%%%%%%%%%%%%%%%%%%%%%%
\section{Reconstruction of coherent schemes}\label{sec RecSpecSch}
%%%%%%%%%%%%%%%%%%%%%%%%%%%%%%%%%%%%%%%%%%%%%

In this section we show how to assemble our results as local data to 
obtain a new proof of the classical results of
Thomason~\cite{Thomason:1997} on the classification of thick subcategories 
of $D_{\qc}^\omega(X)$ for $X$ a
coherent scheme, but again without having to bother about points.
We also reconstruct the structure sheaf of $X$ from its derived 
category of perfect complexes.

The key ingredients are on one hand our explicit results in the affine case, and
on the other hand the result that the Zariski frame of a tensor triangulated
category is coherent and is the recipient of the initial support.  From this we
will establish that the Zariski frame of
$D^\omega_{\qc}(X)$  is isomorphic to the Hochster dual of the Zariski frame of $X$;
 we establish this by checking it in an affine open cover of $X$.  In each
such affine open, the isomorphism is essentially Theorem~\ref{thm main}.  We
then pass from local to global using the fact that coherent schemes are the
schemes finitely built from affine schemes.
%%%%%%%%%%%%%%%%%%%%%%%%%%%%%%%%%%%%%%%%%%%%%%%%%%
\subsection{Coherent schemes and the Hochster topology}
%%%%%%%%%%%%%%%%%%%%%%%%%%%%%%%%%%%%%%%%%%%%%%%%%%

%%%%%%%%%%%%%%%%%%%%%%%%%%%%%%%%%%%%%%%%%%%%%%%%%%
\subsubsection{Coherent schemes}
%%%%%%%%%%%%%%%%%%%%%%%%%%%%%%%%%%%%%%%%%%%%%%%%%%

Recall that a scheme is {\em coherent} when it is
quasi-compact and quasi-separated; this is the  terminology 
recommended in SGA4~\cite[exp.~VI]{SGA4.2}.  A scheme is coherent precisely
when its frame of Zariski open sets is coherent.  
In terms of distributive lattices, coherent schemes are those ringed lattices which can be covered by a
finite number of Zariski lattices, cf.~Coquand-Lombardi-Schuster~\cite{MR2595208}
who call such schemes ``spectral schemes''.
The fact that coherent schemes are thus ``finitely built'' from affine schemes
allows a natural passage from local to global, and is encompassed in the 
following Reduction Principle,  which we learned from
\cite{MR1996800}, where a proof can be found.
% .  We include a proof for the convenience of the reader although
% it is essentially the same as in op.~cit.

%%%%%%%%%%%%%%%%%%%%%%%%%%%%%%%%%%%%%%%%%%%%%
\begin{lemma}[Reduction Principle]\label{prop extprin}
%%%%%%%%%%%%%%%%%%%%%%%%%%%%%%%%%%%%%%%%%%%%%
    Let P be a property of schemes.  Assume that
    
    (H0): Property P holds for all affine schemes.
    
    (H1): If $X$ is a scheme and $X= X_1 \cup X_2$ is an open cover with 
    intersection $X_{12}$, and if property P holds for $X_{12}$, 
    $X_1$, and $X_2$, then property P holds for $X$.
    
    Then property P holds for all coherent schemes.
\end{lemma}
%%%%%%%%%%%%%%%%%%%%%%%%%%%%%%%%%%%%%%%%%%%%%

% \tiny
% 
% \begin{proof}
%   We split the induction argument into two parts: first we establish the
%   extension from affine schemes to separated quasi-compact schemes, and in a
%   second step we establish the extension from separated to quasi-separated.  For
%   the first step, consider a separated scheme $X$ and write it as a finite union
%   of affine schemes $X=X_1\cup \dots \cup X_n$.  Put $X' = X_2 \cup \dots \cup
%   X_n$, so that $X=X_1\cup X'$.  By induction, property P holds for $X_1$ and
%   $X'$.  It also holds for $X_1\cap X'$, because this scheme is the union of the
%   schemes $X_1 \cap X_i$, and these are affine since in the separated scheme $X$
%   the intersection of two affine schemes is affine.  So by (H1) we conclude that
%   P holds for $X$.  The second step, from separated to quasi-separated is the
%   same argument, using this time induction on the number of separated
%   quasi-compact open subschemes needed to cover $X$, and the obvious fact that
%   the intersection of two such separated subschemes is again separated.
% \end{proof}
% 
% \normalsize

%%%%%%%%%%%%%%%%%%%%%%%%%%%%%%%%%%%%%%%%%%%%%
\subsubsection{Hochster topology}\label{subsec RecSpecHoch}
%%%%%%%%%%%%%%%%%%%%%%%%%%%%%%%%%%%%%%%%%%%

For a coherent scheme, we denote by $D_{\qc}(X)$ the derived category
of complexes of $\cO_X$-modules with quasi-coherent homology.  This is
a compactly generated triangulated category (see Bondal and van den
Bergh~\cite[Theorem 3.1.1]{MR1996800}, and also \cite{MR2861069}).
Its subcategory $D^\omega_{\qc}(X)$ of compact objects is the category
of perfect complexes, i.e.~locally isomorphic to bounded complexes of
finitely generated projective $\cO_X$-modules, and Bondal and van den
Bergh~\cite[Theorem 3.1.1]{MR1996800} show that it is generated by a
single compact object.

The following observation shows that one
can apply the results from Section~\ref{sec ZarSpecTriaCat}.

%%%%%%%%%%%%%%%%%%%%%%%%%%%%%%%%%%%%%%%%%%%%
\begin{lemma}\label{lem CompgenisDual}
%%%%%%%%%%%%%%%%%%%%%%%%%%%%%%%%%%%%%%%%%%%%
Let $E,F \in D^\omega_{qc}(X)$, then
\begin{enumerate}
 \item[i)] $\Hom(E,F) \in D_{cq}^\omega(X)$.
 \item[ii)] the complex $E$ is strongly dualizable.
\end{enumerate}

\end{lemma}
%%%%%%%%%%%%%%%%%%%%%%%%%%%%%%%%%%%%%%%%%%%%
\begin{proof}
  (i) In \cite[Lemma 3.3.8]{MR1996800},
  it is shown using the reduction principle, that the complex 
  $\Hom(E,F)$ is  bounded. So compactness
  can then be checked locally, and the statement is clearly
  true on an affine scheme.
  
  (ii) We have to check that the canonical map 
  $F \otimes \Hom(E,\mathcal{O}_X) \rightarrow \Hom(F,E)$
  is an isomorphism for all compact objects $E,F$.
%  Again this can be checked locally:
%  but for a map to be an isomorphism between complexes 
%  is a condition that can be checked locally,
  But isomorphisms can be detected locally,  
  and the statement is clearly true on an affine 
  scheme, where all the involved objects are bounded complexes 
  of finitely projective modules.
\end{proof}

From this lemma we get that  $D^\omega_{\qc}(X)$
satisfies the assumptions of the
theorems in Section~\ref{sec ZarSpecTriaCat}. 
 In particular by Theorem~\ref{thm coherence}, radical
thick tensor ideals in $D_{\qc}^\omega(X)$ form a coherent frame,
and the map $C \mapsto
\sqrt{C}$ is the initial support by Theorem~\ref{thm ZarisInitialTria}.
We will now compare this with a
homologically defined support, cf.~Thomason~\cite[Definition~3.2]{Thomason:1997}.

%%%%%%%%%%%%%%%%%%%%%%%%%%%%%%%%%%%%%%%%%%%%
\begin{definition}\label{def hsupp}
%%%%%%%%%%%%%%%%%%%%%%%%%%%%%%%%%%%%%%%%%%%%%
For $C \in D_{\text{qc}}^\omega(X)$, the {\em homological support} is the subspace 
$\supp(C) \subset X$ of those points $x$ at which the stalk complex of
$\cO_{X,x}$-modules $C_x$ is not acyclic.
\end{definition}
%%%%%%%%%%%%%%%%%%%%%%%%%%%%%%%%%%%%%%%%%%%%

%%%%%%%%%%%%%%%%%%%%%%%%%%%%%%%%%%%%%%%%%%%%
\begin{lemma}
%%%%%%%%%%%%%%%%%%%%%%%%%%%%%%%%%%%%%%%%%%%%
  For any perfect complex $C$, $\supp(C)$ is a Zariski closed set with
  quasi-compact complement (and in particular a Hochster open set).
\end{lemma}
%%%%%%%%%%%%%%%%%%%%%%%%%%%%%%%%%%%%%%%%%%%%
\begin{proof}
  By quasi-compactness of $X$, we can cover $X$ by finitely many open affine
  subschemes on which $C$ is quasi-isomorphic to a bounded complex of finitely
  generated projective modules.  Since an affine scheme $\Spec R$ is
  quasi-compact, it is enough to show that on each of these, $\supp(C) \cap
  \Spec(R)$ is of the form $\zeros{I}$ for some finitely generated ideal $I \subset
  R$.  But $C\vert_{\Spec R}$ is a perfect complex of $R$-modules, and the stalk
  at a point $\mathfrak{p}$ can be computed by tensoring the complex with
  $R_\mathfrak{p}$, so the statement is our Corollary~\ref{cor hsuppisHochOpen}.
\end{proof}

We wish to avoid points as much as possible, 
so we express instead this notion of support in a more conceptual  manner:

%%%%%%%%%%%%%%%%%%%%%%%%%%%%%%%%%%%%%%%%%%%%
\begin{lemma}\label{lem supp=supp}
%%%%%%%%%%%%%%%%%%%%%%%%%%%%%%%%%%%%%%%%%%%%
The assignment
  \begin{eqnarray*}
    D_{\text{qc}}^\omega(X) & \longrightarrow & \Zar(X)^\vee  \\
    C & \longmapsto & \supp C
  \end{eqnarray*}
is a notion of support in the sense of Definition~\ref{def TriaSupport}.
\end{lemma}
%%%%%%%%%%%%%%%%%%%%%%%%%%%%%%%%%%%%%%%%%%%%
\begin{proof}
The fact that $\supp(\Sigma C) = \supp(C)$ is trivial,
as is the fact that $\supp 0 = \emptyset$ and $\supp \cO_X = X$.
For  the thickness property observe that if $ C_1,C_2$ are perfect complexes, then  $(C_1 \oplus C_2)\vert_{x} = C_1\vert_x \oplus C_2\vert_x$, and that  $C_1\vert_x \oplus C_2\vert_x$ is acyclic if and only if  both $C_1\vert_x$  and $ C_2\vert_x$ are acyclic.
For the compatibility with the tensor product, observe first that $(C_1 \otimes C_2)\vert_x = C_1\vert_x \otimes C_2\vert_x$. Then, by the K\"unneth formula we have that $C_1\vert_x \otimes C_2\vert_x$ is acyclic if and only if $C_1\vert_x$ or $C_2\vert_x$ are acyclic. So indeed $\supp(C_1 \otimes C_2) = \supp(C_1) \cup \supp(C_2)$.
For the compatibility with triangles, let 
\[
\xymatrix{
 C_1 \ar[r] & C_2\ar[r]  & C_3 \ar[r] &\Sigma C_1
 }
\]
be a triangle of perfect complexes. Since taking stalks is an exact functor, the long exact sequence in homology associated to the triangle of stalks shows immediately that $\supp C_3 \subset \supp C_1 \cup \supp C_2$.
\end{proof}

The next two results can be found in Thomason~\cite{Thomason:1997} as Lemmas~3.4
and 3.14 respectively.  The difference lies in the fact that we 
deduce them from an analysis of the affine case. In this way we
avoid the use of points in the proof of the first result  and 
avoid both the Tensor Nilpotence  and the  Absolute Noetherian 
Approximation theorems in the second proof.

%%%%%%%%%%%%%%%%%%%%%%%%%%%%%%%%%%%%%%%%%%%%
\begin{lemma}\label{lem one}
%%%%%%%%%%%%%%%%%%%%%%%%%%%%%%%%%%%%%%%%%%%%
  In a coherent scheme $X$, let $Z\subset X$ be a Zariski closed set with quasi-compact
  complement.  Then there exists $E \in D_{\text{qc}}^\omega(X)$ with
  $\supp E = Z$.
\end{lemma}
%%%%%%%%%%%%%%%%%%%%%%%%%%%%%%%%%%%%%%%%%%%%
 \begin{proof} We apply the Reduction Principle (\ref{prop extprin}) to the
  property $P$ that asserts that for any $Z \subset X$, a Zariski closed with
  quasi-compact complement in a scheme the lemma holds.
The affine case  is Theorem~\ref{thm maingeo}.

For the induction step we need the following deep result due to
Thomason-Trobaugh~\cite[Lemma 5.6.2a]{Thomason-Trobaugh}: Fix a coherent scheme
$X$, $U$ a Zariski open set in $X$ and $Z$ a closed set with quasi-compact
complement.  Let $F$ be a perfect complex on $U$, acyclic on $U \setminus U \cap
Z$.  Then there exists a perfect complex $E$ on $X$, acyclic on $X\setminus Z$
such that $E_{|U} \simeq F$ if and only if $[F] \in K_0(U \textrm{ on } U \cap
Z)$ is in the image of the map induced by restriction:
\[
K_0(X \textrm{ on }Z) \rightarrow K_0(U \textrm{ on }U \cap Z).
\]

Let $X = X_1 \cup X_2$ be a scheme covered by two open subschemes.  We assume
that $P$ is true on $X_i$, $i=1,2$, so we have a perfect complex $F_i$ on each
$X_i$ such that $\supp F_i = Z_i = X_i \cap Z$.  Observe that $\supp (F_i \oplus
\Sigma F_i) = Z_i$, but also that by definition of the sum in $K$-theory (see
the proof of \cite[Theorem 2.1]{Thomason:1997})$[F_i \oplus \Sigma F_i] = 0$ in
$K_0(X_i \textrm{ on } Z_i)$.  Therefore we have two perfect complexes $E_1$ and
$E_2$ on $X$, with support included in $Z$ such that $E_{i\vert U_i} \simeq
F_i$.
We claim that $E_1 \oplus E_2$ is the perfect complex we are looking for. Indeed
\[
\supp (E_1 \oplus E_2) \supset \supp E_1 \cup \supp E_2 \supset Z_1 \cup Z_2 = Z,
\]
and by construction, $\supp (E_1 \oplus E_2) \subset Z$.
\end{proof}

%%%%%%%%%%%%%%%%%%%%%%%%%%%%%%%%%%%%%%%%%%%%
\begin{lemma}\label{lem two}
%%%%%%%%%%%%%%%%%%%%%%%%%%%%%%%%%%%%%%%%%%%%
Let $X$ be a coherent scheme.
Given two perfect complexes $E,F \in D_{\text{qc}}^\omega(X)$, we have:
$$
\supp(E) \subset \supp(F)  \ \Leftrightarrow \ \sqrt{E} \subset 
\sqrt{F}.
$$
\end{lemma}
%%%%%%%%%%%%%%%%%%%%%%%%%%%%%%%%%%%%%%%%%%%%
\begin{proof}
The implication ``$\Leftarrow$'' is obvious: if for some $n \geq 1$ we have
$E^{\otimes n}$ can be built from $F$, consider any  finite recipe for $E^{\otimes n}$.
Then at any point $x$, if the stalk of $F$ at $x$ is zero so is the stalk of the
recipe and hence $E^{\otimes n}$ itself.  Therefore $\supp(E^{\otimes n}) =
\supp(E) \subset \supp(F)$.

For the converse implication $\Rightarrow$, we first enlarge a bit the setting and consider $\Loc(F) \subset D_{qc}(X)$. Denote by $L_F$ and $\Gamma_F$ the Bousfield localization and cellularization functors associated to $\Loc(F)$. Then we have a triangle:
\[
 \xymatrix{
 \Gamma_F E \ar[r] & E \ar[r]  & L_F E \ar[r] & \Sigma \Gamma_F E,
 }
\]
obtained by tensoring the triangle
\[
 \xymatrix{
 \Gamma_F \cO_X \ar[r] & \cO_X \ar[r]  & L_F \cO_X \ar[r] & \Sigma \Gamma_F \cO_X
 }
\]
by $E$.  We claim that $L_F E =0$ in $D_{qc}(X)$, so that the leftmost morphism
in the top triangle is an isomorphism.  If this is the case, then $E \in
\Loc(F)$ by definition of $L_F E$, and as $E$ is compact, $E \in \Loc(F) \cap
D_{qc}^\omega(X) = \sqrt{F}$.  That a complex is quasi-isomorphic to the trivial
complex can be checked on the stalks.  Since $E$ is perfect we may apply the
K\"unneth spectral sequence to compute the homology of $(L_F E)_x = L_F
(\cO_X)_x \otimes E_x$.  First restrict the triangle to an affine open set
$\Spec R$.  Since restriction is a triangulated functor that preserves arbitrary
sums and respects compact objects we have that $(\Gamma_F\cO_X)_{\vert_R} = \Gamma_{F_{\vert_R}}$ and we get the triangle:

\[
 \xymatrix{
 \Gamma_{F_{\vert_R}} R \otimes E_{\vert_R}\ar[r] & E_{\vert_R} \ar[r]  & L_{F_{\vert_R}} R \otimes E_{\vert_R} \ar[r] & \Sigma \Gamma_{F_{\vert_R}} R \otimes E_{\vert_R}.
 }
\]

In $D(R)$, by Proposition~\ref{thm comp-celleq-cyc}, 
$F_{\vert_R}$ is cellularly equivalent to $R/I$ for some finitely generated ideal $I$.
An explicit description of $L_{R/I}(R)$ is provided by Dwyer-Greenlees,
see the proof of Theorem~\ref{thm orthfinitgen}, and from this it is immediate to 
check that for a point $x \in \supp F$, $(L_{F_{\vert_R}}R)_{x} = 0$. 
As a consequence  the $E_2$ page of the K\"unneth spectral sequence is trivial
for these points. Now,  if on the contrary $x \notin \supp F$, then 
as $\supp E \subset \supp F$ we have that $E_x = 0$ by definition, and 
the spectral sequence is again trivial. 
\end{proof}

\begin{theorem}\label{thm Zar D = Hoch}
%%%%%%%%%%%%%%%%%%%%%%%%%%%%%%%%%%%%%%%%%%%%%%%%
  For $X$ a coherent scheme, the Zariski frame of $D^\omega_{\qc}(X)$ 
  is the Hochster dual of the Zariski frame of $X$ itself.
\end{theorem}
%%%%%%%%%%%%%%%%%%%%%%%%%%%%%%%%%%%%%%%%%%%%%%%%
\begin{proof}
  By Lemma~\ref{lem supp=supp}, $(\Zar(X)^\vee,\supp)$
  is a support.  Now we invoke the universal 
  property of the Zariski frame of $D_{\qc}^\omega(X)$
  to get a unique morphism of supports
  $$
  \Zar(D_{\qc}^\omega(X))  \stackrel u \longrightarrow 
  \Zar(X)^\vee 
  $$
  sending $\sqrt{C}$ to $\supp(C)$.  It is surjective by 
  Lemma~\ref{lem one} and injective by Lemma~\ref{lem two}. 
\end{proof}

%%%%%%%%%%%%%%%%%%%%%%%%%%%%%%%%%%%%%%%%%%%%%%%%
\subsection{Zariski topology, structure sheaf, and reconstruction of schemes}
%%%%%%%%%%%%%%%%%%%%%%%%%%%%%%%%%%%%%%%%%%%%%%%%
\label{sec strucsheaves}

Theorem~\ref{thm Zar D = Hoch} shows that the underlying topological space
of a coherent scheme $X$ can be reconstructed from its derived category.
We wish to reconstruct also the structure sheaf $\cO_X$. The structure sheaf
refers to the Zariski topology on $\Spec X$, not to the Hochster dual topology, so to get
it we need to pass to the Hochster dual of the Zariski frame of 
$D_{\qc}^\omega(X)$. 
The key point is the standard fact
that in a tensor triangulated category $(\T, \otimes, \mathbf{1})$,
the endomorphism ring of the tensor unit $\mathrm{End}_{\T}(\mathbf{1})$ 
is a commutative ring, by the Eckmann-Hilton argument.

Recall that a sheaf of rings on a frame $F$ is a functor $F\op\to
\kat{Ring}$ satisfying an exactness condition. For a coherent
frame it is enough to specify the values on the finite elements (playing the
role of a basis for a topology). 

%%%%%%%%%%%%%%%%%%%%%%%%%%%%%%%%%%%%%%%%%%%%%%%%%%
\subsubsection{The affine case}
%%%%%%%%%%%%%%%%%%%%%%%%%%%%%%%%%%%%%%%%%%%%%%%%%%

For an affine scheme $X= \Spec_Z R$, the structure sheaf on the Zariski frame
$\Zar(X)= \RadId(R)$ is completely specified by the assignment
\begin{eqnarray*}
  \RadId(R)\op & \longrightarrow & \kat{Ring}  \\
  \sqrt{f} & \longmapsto & R_f ,
\end{eqnarray*}
corresponding to the fact that the principal open sets $D(f) = \Spec R 
\setminus \zeros{f}$ form a basis for the Zariski topology.

We are concerned with the coherent frame
$\kat{RfGLoc}(D(R))$ and its distributive lattice of finite elements
$\kat{fgRfGLoc}(D(R))$ consisting of localizing subcategories generated by a 
finite number of modules of the form $R_f$.
These are both localizing and colocalizing (see
Proposition~\ref{prop dualcatisloc}), and we have at our disposal a Bousfield
localization functor $L_{\Loc(R_{f_1}, \dots, R_{f_n})}$ with values in our
categories $\Loc(R_{f_1}, \dots, R_{f_n})$.  All these are naturally tensor
triangulated categories as they are tensor ideals in $D(R)$, and have as tensor
unit the localization of the unit in $D(R)$.
The natural presheaf
\[
\begin{array}{ccl}
\fgRfGLoc(D(R))\op	 & \longrightarrow & \kat{Ring} \\
 \Loc(R_{f_1}, \dots, R_{f_n})  & \longmapsto & 
 \mathrm{End}_{D(R)}(L_{\Loc(R_{f_1}, \dots, R_{f_n})}(R))
\end{array}
\]
yields by sheafification a sheaf  
\[
\begin{array}{rcl}
\mathcal{E}nd : 
\RfGLoc(D(R))\op	 & \longrightarrow & \kat{Ring} .
\end{array}
\]

%%%%%%%%%%%%%%%%%%%%%%%%%%%%%%%%%%%%%%%%%%%%%
\begin{proposition}\label{thm ZarSheaf}
%%%%%%%%%%%%%%%%%%%%%%%%%%%%%%%%%%%%%%%%%%%%%
Under the isomorphism $$\kat{RfGLoc}(D(R)) \simeq \RadId(R)$$ of 
Theorem~\ref{thm loccatisZar},
  the sheaf $\mathcal{E}nd$ is canonically isomorphic to the structure sheaf on 
$\Spec_Z R$.
\end{proposition}
%%%%%%%%%%%%%%%%%%%%%%%%%%%%%%%%%%%%%%%%%%%%%
\begin{proof}
  It is enough to compute the sheaf on a basis of the topology, and for this we
  take the lattice of localizing subcategories generated by a single
  localization, corresponding to the basis of principal open sets in $\Spec_Z R$.  We know
  that as tensor triangulated categories $\Loc(R_f) \simeq D(R_f)$
  (Proposition~\ref{prop orthOneGen}), hence:
  \begin{eqnarray*}
   \End_{D(R)}(L_{\Loc(R_f)} (R)) & = &  \End_{D(R_f)}(L_{\Loc(R_f)} (R)) \\
  & = &  \End_{D(R_f)}(R_f) \\
   & = & R_f,
  \end{eqnarray*}
as rings. But  $R_f$ is precisely  the value of the structure sheaf of $\Spec_Z R$
on the principal open set $D(f)=\Spec R \setminus \zeros{f}$.
\end{proof}

%%%%%%%%%%%%%%%%%%%%%%%%%%%%%%%%%%%%%%%%%%%%%%%%%%
\subsubsection{Reconstruction of a general coherent scheme}
%%%%%%%%%%%%%%%%%%%%%%%%%%%%%%%%%%%%%%%%%%%%%%%%%%

First we enlarge the framework to
that of the whole derived category of complexes of modules with quasi-coherent
homology $D_{\qc}(X)$. 
It follows from Corollary~\ref{cor LocXTomega} we have an isomorphism of posets
between the poset of localizing subcategories of $D_{\qc}(X)$ generated by a
single perfect complex and the poset of thick subcategories of
$D_{\qc}^\omega(X)$ generated by a single perfect complex; via the assignment
$\mathcal L \mapsto \mathcal L \cap D_{\qc}^\omega(X)$.  Since $\cO_X$ generates
$D_{\qc}(X)$ as a localizing category, all localizing subcategories are tensor
ideals by Lemma~\ref{lem locistensideal}, and hence also all the thick
subcategories are thick tensor ideals, and since all perfect complexes are
strongly dualizable, all thick tensor ideals are radical thick tensor ideals.
Altogether we have an isomorphism of posets between the localizing subcategories
of $D_{\qc}(X)$ generated by a single perfect complex, and the Zariski lattice
$\Zar(D_{\qc}^\omega(X))^\omega = \{ \sqrt C \mid C \in D_{\qc}^\omega(X)\}$ of
principal radical thick tensor ideals, i.e.~the distributive lattice of finite
elements in $\Zar(D_{\qc}^\omega(X))$.

We know that the Hochster dual of this lattice is the basis of the
topology of $X$ given by the quasi-compact open sets in $X$. To flip this lattice
as in the affine case we take right orthogonals. The relations between right and left orthogonals for (co)localizing subcategories as stated for instance in  \cite[Prop 4.9.1 and 4.10.1]{MR2681709} imply that we have order-reversing inverse bijections:
\[
\xymatrix{
\big\{  \Loc(C) \mid C \in  D_{\qc}^\omega(X) \big\}  \ar@/^/[rr]^{(-)^\bot} & & 
\ar@/^/[ll]^{{}^\bot(-)} \big\{  \Loc(C)^{\bot} \mid C \in  D_{\qc}^\omega(X)\big\} \\
} 
\]

We therefore have:

%%%%%%%%%%%%%%%%%%%%%%%%%%%%%%%%%%%%%%%%%%%%%%%%
\begin{proposition}\label{thm recSpecSch}
%%%%%%%%%%%%%%%%%%%%%%%%%%%%%%%%%%%%%%%%%%%%%%%%
Let $X$ be a coherent scheme.  There is a canonical isomorphism between the
distributive lattice $\big\{ \Loc(C)^{\bot} \mid C \in D_{\qc}^\omega(X)\big\}$
and the Zariski lattice of $X$ (i.e.~the lattice of quasi-compact open sets in 
$X$).
\end{proposition}
%%%%%%%%%%%%%%%%%%%%%%%%%%%%%%%%%%%%%%%%%%%%%%%%

To reconstruct the sheaf we proceed again as in the affine case.  The categories
$\Loc(C)^{\bot}$ are localizing as they are the right orthogonal categories to a
compact object (cf. Proposition~\ref{prop dualcatisloc}).  By Lemma~\ref{lem locistensideal}, they are tensor ideals and
we may apply Bousfield localization techniques, see for instance \cite{MR2681709}.  The localization of the tensor unit $\cO_X$ at $\Loc(C)^\bot$,
namely $L_{C^\bot}(\cO_X)$ is the tensor unit in this tensor triangulated
category; its ring of endomorphisms is a commutative ring and we get a presheaf
\begin{eqnarray*}
  \big\{ \Loc(C)^{\bot} \mid C \in D_{\qc}^\omega(X)\big\}\op & \longrightarrow 
  & \kat{Ring}  \\
  \Loc(C)^\bot & \longmapsto & \End_{\Loc(C)^{\bot}}(L_{C^\bot}(\cO_X)) .
\end{eqnarray*}
Sheafification of this presheaf defines the sheaf $\mathcal{E}nd$.

Finally we get the reconstruction theorem, slightly generalizing that proved by
Balmer~\cite{MR1938458}, who  did the special case 
where $X$ is topologically noetherian:

\begin{theorem}
  Under the isomorphism  $\Zar(D_{\qc}^\omega(X))^\vee \simeq \Zar(X)$
  of Theorem~\ref{thm Zar D = Hoch}, 
  the sheaf $\mathcal{E}nd$ is canonically isomorphic to the structure sheaf on 
  $X$.
\end{theorem}

\begin{proof}
  The isomorphism of sheaves can be checked on the sub-basis of affine open
  subsets, whence we reduce to the case of Proposition \ref{thm ZarSheaf}.
\end{proof}

%%%%%%%%%%%%%%%%%%%%%%%%%%%%%%%%%%%%%%%%%%%%%%%%%%
\subsubsection{The domain sheaf}
%%%%%%%%%%%%%%%%%%%%%%%%%%%%%%%%%%%%%%%%%%%%%%%%%%

An affine scheme $X=\Spec R$ also has a natural sheaf for the Hochster dual 
topology, given by sheafification of the presheaf
 \[
\begin{array}{rcl}
(\fgRadId\op)\op & \longrightarrow & \kat{Ring} \\ 
I & \longmapsto & R/I .
\end{array}
\]
Note that while the usual structure sheaf for the Zariski topology is a 
local-ring object in the petit Zariski topos, the structure sheaf for the Hochster dual 
topology is instead a domain object~\cite[V.4]{Johnstone:Stone-spaces}.
(Or in terms of points: the stalk of this sheaf at a prime $\mathfrak{p}$ is
the domain $R/\mathfrak{p}$.)

Also the domain sheaf of $\Spec_H R$ can be reconstructed from the
derived category $D(R)$, simply by copying over the definition of the sheaf
as sheafification of the presheaf
 \[
\begin{array}{rcl}
\kat{fgCGLoc}(D(R))\op & \longrightarrow & \kat{Ring} \\ 
\Loc(R/I) & \longmapsto & R/I .
\end{array}
\]

In principle this local description can be globalized to account for some
notion of scheme defined as ``ringed space which is locally the domain spectrum
of a commutative ring''.  Having no feeling for this notion, we postpone further 
investigations of  this point.

%%%%%%%%%%%%%%%%%%%%%%%%%%%%%%%%%%%%%%%%%%%%
%%%%%%%%%%%%%%%%%%%%%%%%%%%%%%%%%%%%%%%%%%%%

\section*{Acknowledgment}
It's a pleasure to thank Wojciech Chach\'olski, J\'er\^ome Scherer, Henning
Krause, Fei Xu and Vivek Mallick for many stimulating discussions related to
this work.  We also wish to thank Patrick Brosnan for pointing out a flaw in a
proof in an earlier version of this work, and the anonymous referee for many 
valuable remarks.

\bibliographystyle{plain}\label{biblography}

\begin{thebibliography}{10}

\bibitem{SGA4.2}
{\em Th\'eorie des topos et cohomologie \'etale des sch\'emas. {T}ome 2}.
\newblock Number 270 in Lecture Notes in Mathematics. Springer-Verlag, Berlin,
  1972.
\newblock S\'eminaire de G\'eom\'etrie Alg\'ebrique du Bois-Marie 1963--1964
  (SGA 4), Dirig\'e par M. Artin, A. Grothendieck et J. L. Verdier. Avec la
  collaboration de N. Bourbaki, P. Deligne et B. Saint-Donat.

\bibitem{MR2861069}
L.~Alonso~Tarr{\'{\i}}o, A.~Jerem{\'{\i}}as~L{\'o}pez,
  M.~P{\'e}rez~Rodr{\'{\i}}guez, and M.~J. Vale~Gonsalves.
\newblock On the existence of a compact generator on the derived category of a
  {N}oetherian formal scheme.
\newblock {\em Appl. Categ. Structures}, 19(6):865--877, 2011.

\bibitem{MR1938458}
P.~Balmer.
\newblock Presheaves of triangulated categories and reconstruction of schemes.
\newblock {\em Math. Ann.}, 324(3):557--580, 2002.

\bibitem{Balmer:2005}
P.~Balmer.
\newblock The spectrum of prime ideals in tensor triangulated categories.
\newblock {\em J. Reine Angew. Math.}, 588:149--168, 2005.

\bibitem{MR2806103}
P.~Balmer and G.~Favi.
\newblock Generalized tensor idempotents and the telescope conjecture.
\newblock {\em Proc. Lond. Math. Soc. (3)}, 102(6):1161--1185, 2011.

\bibitem{Benson-Carlson-Rickard}
D.~J. Benson, J.~F. Carlson, and J.~Rickard.
\newblock Thick subcategories of the stable module category.
\newblock {\em Fund. Math.}, 153(1):59--80, 1997.

\bibitem{MR1996800}
A.~Bondal and M.~van~den Bergh.
\newblock Generators and representability of functors in commutative and
  noncommutative geometry.
\newblock {\em Mosc. Math. J.}, 3(1):1--36, 258, 2003.

\bibitem{Buan-Krause-Solberg}
A.~B. Buan, H.~Krause, and {\O}.~Solberg.
\newblock Support varieties: an ideal approach.
\newblock {\em Homology, Homotopy Appl.}, 9(1):45--74 (electronic), 2007.

\bibitem{MR2595208}
T.~Coquand, H.~Lombardi, and P.~Schuster.
\newblock Spectral schemes as ringed lattices.
\newblock {\em Ann. Math. Artif. Intell.}, 56(3-4):339--360, 2009.

\bibitem{MR960945}
E.~S. Devinatz, M.~J. Hopkins, and J.~H. Smith.
\newblock Nilpotence and stable homotopy theory. {I}.
\newblock {\em Ann. of Math. (2)}, 128(2):207--241, 1988.

\bibitem{MR1879003}
W.~G. Dwyer and J.~P.~C. Greenlees.
\newblock Complete modules and torsion modules.
\newblock {\em Amer. J. Math.}, 124(1):199--220, 2002.

\bibitem{MR0251026}
M.~Hochster.
\newblock Prime ideal structure in commutative rings.
\newblock {\em Trans. Amer. Math. Soc.}, 142:43--60, 1969.

\bibitem{MR932260}
M.~J. Hopkins.
\newblock Global methods in homotopy theory.
\newblock In {\em Homotopy theory ({D}urham, 1985)}, volume 117 of {\em London
  Math. Soc. Lecture Note Ser.}, pages 73--96. Cambridge Univ. Press,
  Cambridge, 1987.

\bibitem{Iyengar-Krause:1105.1799}
S.~B. Iyengar and H.~Krause.
\newblock The {B}ousfield lattice of a triangulated category and
  stratification.
\newblock {\em Math. Z.}, 273(3-4):1215--1241, 2013.

\bibitem{Johnstone:Stone-spaces}
P.~T. Johnstone.
\newblock {\em Stone spaces}, volume~3 of {\em Cambridge Studies in Advanced
  Mathematics}.
\newblock Cambridge University Press, Cambridge, 1986.
\newblock Reprint of the 1982 edition.

\bibitem{Joyal:NAMS1971}
A.~Joyal.
\newblock Spectral spaces and distributive lattices.
\newblock {\em Notices A.M.S.}, 18:393, 1971.

\bibitem{Joyal:Cahiers1975}
A.~Joyal.
\newblock Les th{\'e}or{\`e}mes de {C}hevalley-{T}arski et remarques sur
  l'alg{\`e}bre constructive.
\newblock {\em Cahiers Top. et G\'eom. Diff.}, XVI-3:256--258, 1975.

\bibitem{MR2970877}
J.~Kiessling.
\newblock Properties of cellular classes of chain complexes.
\newblock {\em Israel J. Math.}, 191:483--505, 2012.

\bibitem{kock:specletter}
J.~Kock.
\newblock Spectra, supports, and {H}ochster duality.
\newblock HOCAT talk (November 2007), and letter to P.~Balmer, G.~Favi, and
  H.~Krause (December 2007), \url{http://mat.uab.cat/~kock/cat/spec.pdf}.

\bibitem{MR2681709}
H.~Krause.
\newblock Localization theory for triangulated categories.
\newblock In {\em Triangulated categories}, volume 375 of {\em London Math.
  Soc. Lecture Note Ser.}, pages 161--235. Cambridge Univ. Press, Cambridge,
  2010.

\bibitem{MR866482}
L.~G. Lewis, Jr., J.~P. May, M.~Steinberger, and J.~E. McClure.
\newblock {\em Equivariant stable homotopy theory}, volume 1213 of {\em Lecture
  Notes in Mathematics}.
\newblock Springer-Verlag, Berlin, 1986.
\newblock With contributions by J. E. McClure.

\bibitem{MR1174255}
A.~Neeman.
\newblock The chromatic tower for {$D(R)$}.
\newblock {\em Topology}, 31(3):519--532, 1992.
\newblock With an appendix by Marcel B{\"o}kstedt.

\bibitem{MR1812507}
A.~Neeman.
\newblock {\em Triangulated categories}, volume 148 of {\em Annals of
  Mathematics Studies}.
\newblock Princeton University Press, Princeton, NJ, 2001.

\bibitem{MR2794632}
A.~Neeman.
\newblock Colocalizing subcategories of {$\mathbf D(R)$}.
\newblock {\em J. Reine Angew. Math.}, 653:221--243, 2011.

\bibitem{MR2434186}
R.~Rouquier.
\newblock Dimensions of triangulated categories.
\newblock {\em J. K-Theory}, 1(2):193--256, 2008.

\bibitem{MR2681712}
R.~Rouquier.
\newblock Derived categories and algebraic geometry.
\newblock In {\em Triangulated categories}, volume 375 of {\em London Math.
  Soc. Lecture Note Ser.}, pages 351--370. Cambridge Univ. Press, Cambridge,
  2010.

\bibitem{Thomason:1997}
R.~W. Thomason.
\newblock The classification of triangulated subcategories.
\newblock {\em Compositio Math.}, 105(1):1--27, 1997.

\bibitem{Thomason-Trobaugh}
R.~W. Thomason and T.~Trobaugh.
\newblock Higher algebraic {$K$}-theory of schemes and of derived categories.
\newblock In {\em The {G}rothendieck {F}estschrift, {V}ol.\ {III}}, volume~88
  of {\em Progr. Math.}, pages 247--435. Birkh\"auser Boston, Boston, MA, 1990.

\end{thebibliography}

\end{document}